\documentclass[review,sort&compress,3p]{elsarticle}
\usepackage{lineno,hyperref}
\usepackage{graphicx}
\usepackage{amsmath, amssymb, amsfonts, amsthm}
\usepackage{mathrsfs}
\usepackage{url}
\usepackage{latexsym}
\usepackage[ruled,vlined]{algorithm2e}
\usepackage{color}
\usepackage{subfigure}
\usepackage{multirow}
\usepackage{epstopdf}
\usepackage{float}

\definecolor{darkblue}{rgb}{0,0,0.4}
\newtheorem{thm}{Theorem}[section]

\newtheorem{lem}[thm]{Lemma}
\theoremstyle{remark}
\newtheorem{remark}[thm]{Remark}

\graphicspath{{./figs/}}

\journal{}

\begin{document}

\begin{frontmatter}

\title{A prediction-correction based iterative convolution-thresholding method for topology optimization of heat transfer problems}

\author[a]{Huangxin Chen}
\ead{chx@xmu.edu.cn}

\author[a]{Piaopiao Dong}
\ead{dongpiaopiao@stu.xmu.edu.cn}

\address[a]{School of Mathematical Sciences and Fujian Provincial Key Laboratory on Mathematical Modeling and High Performance Scientific Computing, Xiamen University, Fujian, 361005, China}

\author[b,c]{Dong Wang\corref{cor1}}
\ead{wangdong@cuhk.edu.cn}

\author[b,c]{Xiao-Ping Wang}
\ead{wangxiaoping@cuhk.edu.cn}

\address[b]{School of Science and Engineering, The Chinese University of Hong Kong, Shenzhen, Guangdong 518172, China}
\address[c]{Shenzhen International Center for Industrial and Applied Mathematics, Shenzhen Research Institute of Big Data, Guangdong 518172, China}
\cortext[cor1]{Corresponding author}

\begin{abstract}
In this paper, we propose an iterative convolution-thresholding method (ICTM) based on prediction-correction for solving the topology optimization problem in steady-state heat transfer equations. The problem is formulated as a constrained minimization problem of the complementary energy, incorporating a perimeter/surface-area regularization term, while satisfying a steady-state heat transfer equation. The decision variables of the optimization problem represent the domains of different materials and are represented by indicator functions. The perimeter/surface-area term of the domain is approximated using Gaussian kernel convolution with indicator functions. In each iteration, the indicator function is updated using a prediction-correction approach. The prediction step is based on the variation of the objective functional by imposing the constraints, while the correction step ensures the monotonically decreasing behavior of the objective functional. Numerical results demonstrate the efficiency and robustness of our proposed method, particularly when compared to classical approaches based on the ICTM.
\end{abstract}

\begin{keyword}
iterative convolution-thresholding method, prediction-correction, topology optimization, heat transfer.
\end{keyword}

    \end{frontmatter}


\section{Introduction}

With the rapid advancement of science and technology, electronic devices continue to decrease in size, resulting in an increasing demand for effective heat dissipation solutions \cite{bejan1997constructal}. As devices become smaller, the available surface area for heat dissipation diminishes, making it more challenging for heat to transfer from the device to the surrounding environment. To address these issues, improving the design of the device to enhance heat dissipation becomes crucial. Determining the distribution of high conductivity material within a fixed domain, while considering appropriate boundary conditions, becomes a key consideration \cite{bejan1997constructal}.

At its core, this problem involves optimizing a free interface under constraints expressed as partial differential equations, with the interface serving as the decision variable. To tackle this problem, various models and numerical methods have been developed. These methods primarily focus on representing the interface, approximating the objective functional within this representation, and employing optimization approaches tailored to specific objective functionals. Examples of such methods include the homogenization method \cite{bendsoe1988generating}, the density interpolation method \cite{bendsoe2003topology}, the level set method \cite{allaire2005level,allaire2004structural,liu2005structure,yamada2010topology}, topological derivatives \cite{novotny2012topological}, the phase field method \cite{garcke2015numerical,jin2023adaptive,li2022provably,xie2023effective,li2022unconditionally}, the evolutionary structural optimization method \cite{xie1993simple}, and the moving morphable components method \cite{guo2014doing}.

In the context of heat transfer topology optimization, several methods have been investigated to maximize thermal efficiency in engineering systems, taking into account heat transfer considerations \cite{bejan1997constructal,dbouk2017review}. A notable benchmark problem was introduced by Bejan in \cite{bejan1997constructal}, which focused on optimizing structures for area/volume-to-point heat generation to minimize average or maximum temperature within a domain. The goal was to distribute a limited amount of high-conductivity material. Review papers such as \cite{dbouk2017review,fawaz2022topology} provide an overview of the development of various numerical methods for topology optimization in heat transfer problems.

For example, the level set method has been employed for topology optimization in heat conduction problems under multiple load cases \cite{zhuang2007level}, as well as in two-dimensional thermal problems \cite{ullah2021parametrized}. Parametric L-systems and the finite volume method were used in \cite{ikonen2018topology} to solve the topology optimization problem in two-dimensional heat conduction, aiming to minimize average and maximum temperatures. Three-dimensional topology optimization for heat conduction in a cubic domain, focusing on volume-to-point/surface problems, was investigated in \cite{burger2013three} using the method of moving asymptotes (MMA). Furthermore, an overview of the entransy theory-based method was presented in \cite{chen2013entransy} for optimizing thermal systems, discussing fundamental concepts and applications in various heat transfer scenarios.

Among the various methods utilized for topology optimization, the threshold dynamics method \cite{merriman1992diffusion,esedog2015threshold} has attractted  attention due to its ability to maintain energy stability while reducing computational complexity. This method has been extended to address a wide range of engineering problems, including image segmentation \cite{esedog2006threshold,ma2021characteristic,wang2017efficient,wang2022iterative,liu2023active}, flow network design \cite{chen2022efficient,hu2022efficient}, minimum compliance problems \cite{cen2023iterative}, and heat transfer problems with constant  heat generation rates \cite{cen2023heat}. In \cite{wang2022iterative}, the iterative convolution-thresholding method (ICTM) was introduced as a framework for interface-related optimization problems, achieved by altering the convolution and thresholding steps. This approach has also been extended to target-valued harmonic maps, including Dirichlet partition problems \cite{wang2019diffusion,wang2022efficient}, orthogonal matrix-valued harmonic maps \cite{osting2020diffusion,wang2019interface}, target-valued imaging denoising \cite{osting2020diffusion2}, foam bubbles \cite{wang2019dynamics}, and other areas.

In this paper, we aim to apply the ICTM to solve the topology optimization problem in heat transfer, where the heat generation rate is linked to material distributions. However, we encountered challenges in maintaining energy stability when applying the ICTM directly to such problems. To address this, we propose a prediction-correction-based iterative convolution-thresholding method that proves to be efficient and robust for topology optimization in heat transfer problems. The method exhibits robustness to parameter variations and ensures a monotonically decaying objective functional during iterations. We believe that this novel method can also be applied to other general topology optimization problems with multi-physics constraints.

The paper is organized as follows: In Section \ref{sec:heat-transfer}, we introduce the mathematical model for topology optimization of steady-state heat transfer problems. Section \ref{sec:prediction-correction} presents the derivation of the prediction-correction-based iterative convolution-thresholding method. We discuss the numerical implementation and provide numerical experiments for two and three-dimensional problems to showcase the robustness of our proposed method in Section \ref{sec:numerical}. Finally, we present our conclusions in Section \ref{sec:conclusion}.

\section{Steady-state conductive heat transfer problems} \label{sec:heat-transfer}
In this section, we consider the mathematical model for topology optimization of steady-state heat transfer problem. Denote $\Omega\in \mathbb{R}^d \ (d=2,3)$ as the computational domain, which remains fixed during the optimization process and assume that $\Omega$ is a bounded Lipschitz domain with an outer unit normal $\mathbf{n}$ and that $\mathbb{R}^d \setminus \overline{\Omega}$ is connected. Additionally, we define $\Omega_1\subset \Omega$ as the region of high-conduction materials with a low heat generation rate and $\Omega_2=\Omega\setminus\Omega_1\in\Omega$ as the region of the device that contains electrical components with low conductivity and high heat generation rate. Denote $q_i$ and $\kappa_i$ as the heat generation rate and the thermal conductivity in $\Omega_i$, respectively. The boundary of $\Omega$ is denoted by $\Gamma \colon = \Gamma_D \cup \Gamma_N$ where the Dirichlet boundary condition is imposed on $\Gamma_D$ and the Neumann boundary condition is imposed on $\Gamma_N$. See Figure~\ref{m1} for a diagram.
\begin{figure}
	\centering
	\setlength{\abovecaptionskip}{-0.2cm}	
        \includegraphics[width=0.4\linewidth]{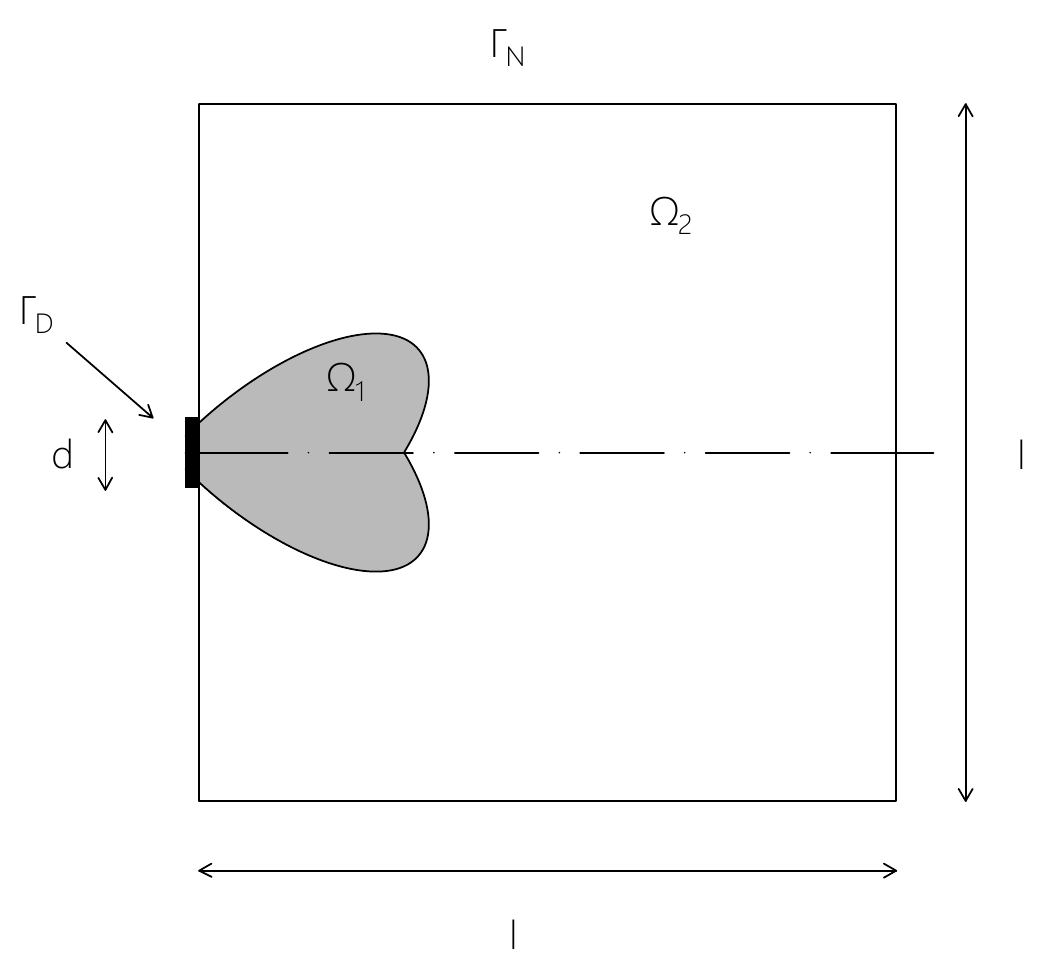}
	\caption{$\Omega_1\in \Omega$ is the region of high conductive material with low heat-generation rate. $\Omega_2=\Omega\setminus\Omega_1\in\Omega$ is the region of low conductive material with high heat-generation rate, and it represents an area of the device that is filled with electrical components (cf. \cite{ikonen2018topology}). See Section~\ref{sec:heat-transfer}.}
    \label{m1}
\end{figure}

Motivated by \cite{boichot2016genetic,ikonen2018topology}, we consider the following constrained optimization problem,
\begin{equation}\label{o1}
	\min_{(\Omega_1,T)}J(\Omega_1,T)=\int_{\Omega}q T d\textbf{x} + \gamma|\Gamma|,
\end{equation}
subject to
\begin{align}\label{t1}\left\{\begin{aligned}
		\nabla\cdot(\kappa \nabla T)+q &=0,\ \   &&\rm in\  \Omega, \\
		(\kappa\nabla T)\cdot \mathbf{n}&=0, \ \ &&\rm on\  \Gamma_N,\\
		T&=0,\ \ &&\rm on\  \Gamma_D,\\
		|\Omega_1| &= \beta|\Omega|,\ \ &&\rm with\ a\ fixed\ parameter\ \beta\in(0,1),
	\end{aligned}\right.\end{align}
where $T$ is the temperature field, $\textbf{n}$ is the outward normal vector of the boundary, $\kappa = \kappa_1 \chi_{\Omega_1}+ \kappa_2 \chi_{\Omega_2}$, $q = q_1 \chi_{\Omega_1}+ q_2 \chi_{\Omega_2}$, and $\chi_{\Omega_i}$ is the indicator function of $\Omega_i$. This problem is a classic area/volume-to-point heat conduction problem. It involves integrating a specific quantity of highly conductive material within an area or volume, with the purpose of dissipating the generated heat to a heat sink (point) by pure heat conduction with different heat generation rate within the domain. This model can also be derived based on taking average over the height of the domain from a three dimensional heat sink design problem \cite{van2014numerical} without fluid flow.



\section{The prediction-correction based algorithm}
\label{sec:prediction-correction}
In this section, we present a prediction-correction based iterative convolution-thresholding method for solving the topology optimization problem~(\ref{o1}). Fundamentally, this is a free interface related optimization problem with the decision variable being an interface between two different materials. Motivated by the new derivation of diffusion generated motion by mean curvature flow \cite{esedog2015threshold}, the iterative convolution-thresholding method (ICTM) for image segmentation \cite{wang2017efficient,wang2022iterative} and fluid-solid topology optimization \cite{chen2022efficient,hu2022efficient}, we consider to use the indicator function of a domain to implicitly represent the interface between two different materials, approximate the problem using heat kernel convolutions with indicator functions, and derive the prediction-correction-based iterative convolution-thresholding method to approximately solve the problem~(\ref{o1}).


\subsection{Approximate objective functional}
Define an admissible set $\mathcal{B}$ as follows:
\begin{equation}\label{B}
	\mathcal{B}:=\big\{\chi\in BV(\Omega)\ | \  \chi(\textbf{x})=\{0,1\},  \int_{\Omega} \chi \ d\textbf{x}=V_0\big\},
\end{equation}
where $BV(\Omega)$ is the vector space of functions with bounded variation in $\Omega$ and $V_0 = \beta |\Omega|$ is the fixed volume of the high conduction domain. We denote the indicator function $\chi$ of the high conduction domain $\Omega_1$, i.e.,
\begin{align}\label{chia}
	\chi(\textbf{x}):=\left\{\begin{aligned}
		&1,\ \ \ \ \textrm{if}\  \textbf{x}\in\  \Omega_1,\\
		&0, \ \ \ \ \textrm{otherwise}.
	\end{aligned}\right.\end{align}

The boundary $\Gamma$ of $\Omega_1$ is then implicitly represented by $\chi$ and the perimeter of $\Gamma$ is then approximated by
\begin{equation}\label{r}
	|\Gamma|\approx \sqrt\frac{\pi}{\tau}\int_\Omega \chi G_{\tau}* (1-\chi) \ d\textbf{x},
\end{equation}
where $$G_\tau = \frac{1}{(4\pi\tau)^{d/2}}\exp\left(-\frac{|\textbf{x}|^2}{4\tau}\right)$$ is the Gaussian kernel. The rigorous proof of the convergence as $\tau \rightarrow 0$ can be found in Miranda et al. \cite{miranda2007short}.  

From the above notation of $\mathbf{\chi}$, we approximate the conductivity and heat generation rate as,
\begin{align}\label{kq}
	&\kappa(\chi)=\kappa_1 G_\tau \ast \chi +\kappa_2 G_\tau \ast (1-\chi),\\
	&q(\chi) = q_1 G_\tau \ast \chi +q_2 G_\tau \ast (1-\chi).
\end{align}

To make the problem be more regularized, we propose to consider the following objective functional with an additional regularization term  $\int_{\Omega} \kappa(\chi)\nabla T \cdot \nabla T \ d\textbf{x}$ in an approximate form:
\begin{equation}\label{o12}
	 J^\tau(\chi,T)=\int_{\Omega}q(\chi)T \ d\textbf{x} +  \frac{\xi}{2}\int_{\Omega} \kappa(\chi)\nabla T \cdot \nabla T \ d\textbf{x} + \gamma\sqrt\frac{\pi}{\tau}\int_\Omega \chi G_{\tau}\ast (1-\chi) \ d\textbf{x},
\end{equation}
where $\xi>0$ is a parameter. 

Finally, we have the following approximate and regularized formula of problem~(\ref{o1}) as
\begin{equation}\label{o111}
\begin{cases}
	&\min\limits_{\mathbf{\chi}\in \mathcal{B}}J^\tau(\mathbf{\chi},T), \\
	&\textrm{s.t.}  \ T \in Q(\chi),\ \ 
\end{cases}
\end{equation}
where $Q(\chi)$ is the system \eqref{t1} defined by a given $\chi$.

\subsection{The iterative convolution-thresholding method (ICTM) for \eqref{o111} similar to \cite{wang2017efficient,wang2022iterative}} \label{sec:ICTM}

In this section, we first follow the iterative convolution-thresholding method (ICTM) introduced in \cite{wang2017efficient,wang2022iterative} to derive an iterative algorithm to solve \eqref{o111}. We will show that the original ICTM will fail for this type of problems through both theoretical and numerical observations.

To be specific, given an initial guess $\chi^{0}$, in the original ICTM framework, one may seek a sequence
\begin{equation*}
	\chi^{1},\chi^{2},\cdots,\chi^{k+1},\cdots
\end{equation*}
to approximately find the optimizer of the problem. 

One can first relax the constraint $T\in Q(\chi)$ via introducing a Lagrange multiplier $T^*$ to have 
\begin{equation}\label{ro1}
	\tilde{J}^\tau(\chi,T,T^*)=J^\tau(\chi,T)+\int_{\Omega}(-\nabla\cdot(\kappa(\chi)\nabla T)-q(\chi))\cdot T^* d\textbf{x}.
\end{equation}
Based on the first order necessary condition of the optimization problem, one obtains a solution for a given $\chi$ with $T$ and $T^*$ satisfying
\begin{equation*}
	\frac{\delta\tilde{J}^\tau}{\delta T}=0,\ \ \ \frac{\delta\tilde{J}^\tau}{\delta T^*}=0.
\end{equation*}
That is
\begin{align}\label{and}\left\{\begin{aligned}
		-\nabla\cdot(\kappa(\chi)\nabla T)-q(\chi)&=0,\ \ &&\rm in\ \ \Omega,\\
		T&=0,\ \ &&\rm on\ \ \Gamma_D,\\
		\kappa(\chi)\nabla T\cdot \mathbf{n}&=0,\ \ &&\rm on\ \ \Gamma_N,
	\end{aligned}\right.\end{align}
and 
\begin{align}\label{ad}\left\{\begin{aligned}
		\nabla\cdot(\kappa(\chi)\nabla T^*)&=q(\chi)-\xi (\nabla\cdot(\kappa(\chi)\nabla T)),\ \ &&\rm in\ \ \Omega,\\
		T^*&=0,\ \ &&\rm on\ \ \Gamma_D,\\
		\kappa(\chi)\nabla T^*\cdot \mathbf{n}&=0,\ \ &&\rm on\ \ \Gamma_N.
	\end{aligned}\right.\end{align}
The ICTM framework \cite{wang2017efficient,wang2022iterative} then proposes to generate a sequence as follows:
\begin{equation*}
	\chi^{1}, T^1, T^{*1}, \chi^{2}, T^2, T^{*2},\cdots,\chi^{k}, T^k, T^{*k},\cdots
\end{equation*}
where $T^k$ and $T^{*k}$ are solved from \eqref{and} and \eqref{ad} with $\chi^k$.

Given $T^k, T^{*k}$, we then write the objective functional $\tilde{J}^\tau(\chi,T^k,T^{*k})$  into $\tilde{J}^{\tau,k}(\chi)$ as follows:
\begin{align}\label{on}
	\tilde{J}^{\tau,k}(\chi)&=\int_{\Omega}q(\chi)T^k  \ d\textbf{x} + \frac{\xi}{2} \int_{\Omega} \kappa(\chi) \nabla T^k \cdot \nabla T^k \ d\textbf{x}
	+ \gamma\sqrt\frac{\pi}{\tau}\int_\Omega \chi G_{\tau}\ast (1-\chi) \  d\textbf{x} \nonumber\\
	&\ \ \ \ +\int_{\Omega}\kappa(\chi)\nabla T^k\cdot\nabla T^{*k} \ d\textbf{x}-\int_{\Omega}q(\chi) T^{*k} \ d\textbf{x}.
\end{align}
Based on the definition of $\kappa(\chi)$ and $q(\chi)$,  it is straightforward to show that $\tilde{J}^{\tau,k}(\chi)$ is concave with respect to $\chi$ in the following lemma.

\begin{lem}
$\tilde{J}^{\tau,k}(\chi)$ is concave with respect to $\chi$.
\end{lem}
\begin{proof}
Because $q(\chi)$ and $\kappa(\chi)$ are linear functions in $\chi$, we only need to show the concavity of 
$$\int_\Omega \chi G_{\tau}\ast (1-\chi) \  d\textbf{x}. $$
From the semi-group property of the operator $G_{\tau}\ast $, we can write 
\begin{align*}
\int_\Omega \chi G_{\tau}\ast (1-\chi) \  d\textbf{x}  = & \int_\Omega \chi G_{\tau/2}\ast \left[G_{\tau/2}\ast (1-\chi)\right] \  d\textbf{x}  \\ 
=&  \int_\Omega \left[G_{\tau/2}\ast\chi\right]  \left[G_{\tau/2}\ast (1-\chi)\right] \  d\textbf{x}  \\
= & \int_\Omega \left[G_{\tau/2}\ast\chi\right]  \left[1-G_{\tau/2}\ast \chi \right] \  d\textbf{x} ,
\end{align*}
which can directly imply the concavity of the term combining with the fact the $G_{\tau/2}\ast $ is a positive definite operator.
\end{proof}

The next step is to find $\chi^{k+1}$ such that
\begin{equation}\label{chi}
	\chi^{k+1}=\arg\min_{\chi\in B}\tilde{J}^{\tau,k}(\chi).
\end{equation}
It is the minimization of a concave function on a nonconvex admissible set $\mathcal{B}$. However, we can relax it to a problem defined on a convex admissible set by finding $\chi^{k+1}$
such that
\begin{equation}\label{chi1}
	\chi^{k+1}=\arg\min_{\chi \in \mathcal{\mathcal{H}}}\tilde{J}^{\tau,k}(\chi),
\end{equation}
where $\mathcal{H}$ is the convex hull of $\mathcal{B}$ defined as follows:
\begin{equation}\label{H}
	\mathcal{H}:=\big\{\chi \in BV(\Omega)\ | \  \chi(\textbf{x}) \in [0,1],  \int_{\Omega} \chi \ d\textbf{x}=V_0\big\}. 
\end{equation}
It is not difficult to see that the relaxed problem (\ref{chi1}) is equivalent to the original problem (\ref{chi}) which we summarize in the following lemma.
\begin{lem}\label{lem6}
	Let $T^k$ and $T^{*k}$ be given functions. Then
	\begin{equation*}
		\arg\min_{\chi\in \mathcal{B}}\tilde{J}^{\tau,k}(\chi) \in \arg\min_{\chi\in \mathcal{H}}\tilde{J}^{\tau,k}(\chi).
	\end{equation*}
\end{lem}
\begin{proof}
It is a direct consequence of the following two facts:\\
1. The minimizer of a concave functional on a convex set can always be attained at the boundary of the convex set.\\
2. $\mathcal {B}$ is the boundary of its convex hull $\mathcal{H}$. 
\end{proof}

Now we give a simple iterative convolution-thresholding method to solve (\ref{chi1}) by a sequential linear programming. Firstly, we linearize the objective functional  $\tilde{J}^{\tau,k}(\chi)$ at $\chi^k$,
\begin{equation}\label{l1}
	\tilde{J}^{\tau,k}(\chi)\approx \tilde{J}^{\tau,k}(\chi^k)+\mathcal{L}^{\tau,k}_{\chi^k}(\chi-\chi^k),
\end{equation}
where
\begin{align*}\label{l2}
	\mathcal{L}^{\tau,k}_{\chi^k}(\chi)&=\int_{\Omega}\chi \Phi^k \ d\textbf{x},
\end{align*}
and \[\Phi^k = (q_1-q_2)G_\tau\ast(T^k-T^{*k})+\gamma\sqrt{\frac{\pi}{\tau}}G_\tau\ast (1-2\chi^k) +(k_1-k_2) G_\tau\ast(\frac{\xi}{2}\nabla T^k\cdot\nabla T^k+\nabla T^k\cdot\nabla T^{\ast k}) .\]
Then (\ref{chi1}) can be approximated into
\begin{equation}\label{chi2}
	\chi^{k+1}=\arg\min_{\chi \in \mathcal{H}}\mathcal{L}^{\tau,k}_{r^k}(\chi)=\arg\min_{\chi \in \mathcal{H}}\int_{\Omega}\chi \Phi^k \ d\textbf{x},
\end{equation}
where we dropped the constant terms determined by $\chi^k,~ T^k,~ T^{*k}$.

By introducing a Lagrange multiplier $\sigma$ for the constraint $\int_\Omega \chi \ d\textbf{x}  = V_0$, because the constraint for $\chi$ is bounded in $[0,1]$, one can write the first order necessary condition of problem \eqref{chi2} with the volume preserving constraint by
\[
\begin{cases}
\dfrac{\delta (\mathcal{L}^{\tau,k}_{\chi^k}(\chi) - \sigma(\int_\Omega \chi \ d\textbf{x} -V_0))}{\delta \chi} (x) = \Phi^k-\sigma \leq 0  \ \ \ \textrm{if} \ \  \chi(x)=1 , \\
\dfrac{\delta (\mathcal{L}^{\tau,k}_{\chi^k}(\chi) - \sigma (\int_\Omega \chi \ d\textbf{x} -V_0))}{\delta \chi}(x)  = \Phi^k-\sigma \geq 0  \ \ \ \textrm{if} \ \ \chi(x)=0, \\
\dfrac{d(\mathcal{L}^{\tau,k}_{\chi^k}(\chi) - \sigma (\int_\Omega \chi \ d\textbf{x} -V_0))}{d \sigma} = \int_\Omega \chi \ d\textbf{x} -V_0 = 0. 
\end{cases}
\]
Then, based on the linearity of the functional, one can easily arrive that (\ref{chi2}) can be solved in a pointwise manner by
\begin{align}\label{chi3}
		\chi^{k+1} = \begin{cases} 1 & \ \textrm{if} \ \Phi^k\leq \sigma, \\
0 & \ \textrm{otherwise}.
		\end{cases}	
		\end{align}
where $\sigma$ is chosen as a constant such that $\int_{\Omega}\chi^{k+1}(\textbf{x})d\textbf{x}=V_0$.

The algorithm for problem \eqref{o111} under the ICTM framework is summarized in Algorithm~\ref{a:MBO}.
\begin{algorithm}[ht!]
\DontPrintSemicolon
 \KwIn{$\chi^0$: Initial guess, $\tau>0$, $\xi >0$, $\kappa_1,\kappa_2, q_1,q_2$.}
 \KwOut{$\chi^\ast \in \mathcal{B}$.}
 \While{$\|\chi^{k+1}-\chi^k\|_2>tol$}{
{\bf 1.} For the fixed $\chi^k$, solve
\begin{align*}\left\{\begin{aligned}
		-\nabla\cdot(\kappa(\chi^k)\nabla T)-q(\chi^k)&=0,\ \ &&\rm in\ \ \Omega,\\
		T&=0,\ \ &&\rm on\ \ \Gamma_D,\\
		\kappa(\chi^k)\nabla T\cdot \mathbf{n}&=0,\ \ &&\rm on\ \ \Gamma_N,
	\end{aligned}\right.\end{align*}
and 
\begin{align*}\left\{\begin{aligned}
		\nabla\cdot(\kappa(\chi^k)\nabla T^*)&=q(\chi^k)-\xi (\nabla\cdot(\kappa(\chi^k)\nabla T)),\ \ &&\rm in\ \ \Omega,\\
		T^*&=0,\ \ &&\rm on\ \ \Gamma_D,\\
		\kappa(\chi^k)\nabla T^*\cdot \mathbf{n}&=0,\ \ &&\rm on\ \ \Gamma_N
	\end{aligned}\right.\end{align*}
	to have $T^k$ and $T^{*k}$.\\
{\bf 2.} Use $T^k$ and $T^{*k}$ to evaluate 
 \[\Phi^k = (q_1-q_2)G_\tau\ast(T^k-T^{*k})+\gamma\sqrt{\frac{\pi}{\tau}}G_\tau\ast (1-2\chi^k) +(k_1-k_2) G_\tau\ast(\frac{\xi}{2}\nabla T^k\cdot\nabla T^k+\nabla T^k\cdot\nabla T^{\ast k}) .\]
{\bf 3.} Find a proper $\sigma$ such that $\chi^{k+1}$
\[\chi^{k+1}(x) = \begin{cases}1 \ \  \textrm{if}  \ \Phi^k(x)\leq \sigma,   \\
0 \ \  \textrm{otherwise,} \end{cases}\]
with $\int_\Omega \chi^{k+1}  \ d\textbf{x} = V_0$.
 }
\caption{The ICTM for approximating minimizers of \eqref{o111}. }
\label{a:MBO}
\end{algorithm}

\subsection{The failure of the ICTM}\label{sec:failureICTM}

From the derivation of the ICTM in Section~\ref{sec:ICTM}, one can have
   \begin{equation*}
		\tilde{J}^\tau(\mathbf{\chi}^{k+1},T^k)\leq \tilde{J}^\tau(\mathbf{\chi}^{k},T^k),
	\end{equation*}
such that
\begin{equation}\label{ed}
J^\tau(\mathbf{\chi}^{k+1},T^k)\leq J^\tau(\mathbf{\chi}^{k},T^k)
\end{equation}
for any $ \ \tau>0$. Using the method of Lagrange multiplier and the fact that $J^\tau(\mathbf{\chi}^{k},T)$ is convex with respect to $T$, one can obtain that $T^{k+1}$ is the minimizer of $J^\tau(\mathbf{\chi}^{k+1},T)$ satisfying the constraint defined in \eqref{and}. However, $T^k$ does not satisfy the constraint in \eqref{and} with $\chi^{k+1}$, that is, one can not have 
\[J^\tau(\chi^{k+1},T^{k+1})\leq J^\tau(\chi^{k+1},T^k),\]
and thus one can not have 
\[J^\tau(\chi^{k+1},T^{k+1})\leq J^\tau(\chi^k,T^k).\]
We note that this type of oscillation in the objective functional is also observed numerically. In Figure~\ref{Example1_fail}, we list an example implemented by Algorithm~\ref{a:MBO}. In this example, we impose the Dirichlet boundary condition on the left boundary ${0}\times[0.45,0.55]$ and a Neumann boundary condition on the remaining boundary. We set $\tau = 1\times10^{-4}$, $\gamma = 30$ and $\xi=1\times10^{-5}$. We observe that the approximate solution of $\chi$ is not stable and exhibits variability, switching between multiple solutions as displayed in Figure~\ref{Example1_fail}. Moreover,  the values of the objective functional oscillate during the iterations. 

Indeed, this is a common issue in the classical ICTM for topology optimization with multi-physics constraints. The classical ICTM works very well and converges fast both numerically \cite{wang2017efficient,wang2022iterative} and theoretically \cite{wang2022modularized}, provided that one can have the monotonically decay of the objective functional during iterations. However, for many optimal design problems with multi-physics constraints, the classical ICTM usually fails without monotonically decay from the numerical observation. In next section, we will discuss a novel prediction-correction based ICTM for the cases where the classical ICTM fails.

\begin{figure}[htbp]
	\centering
        \includegraphics[width=0.3\linewidth,trim=5cm 10cm 3cm 10cm,clip]{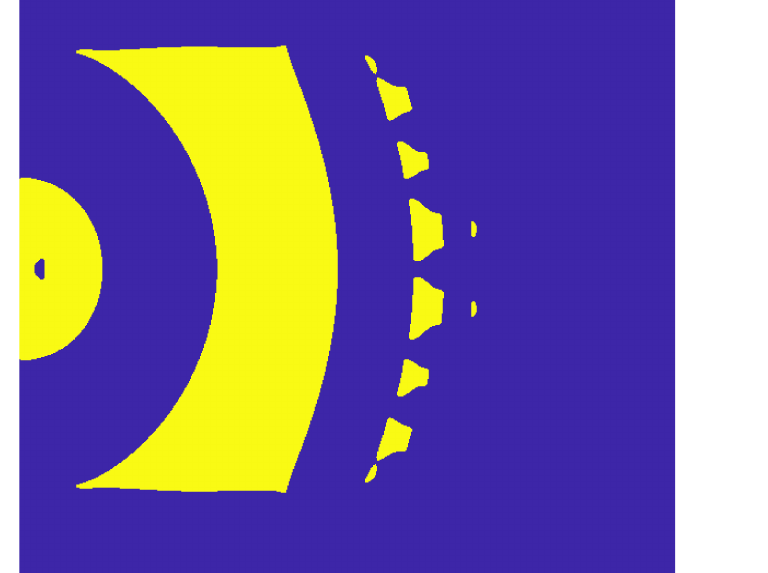}
	\includegraphics[width=0.3\linewidth,trim=5cm 10cm 3cm 10cm,clip]{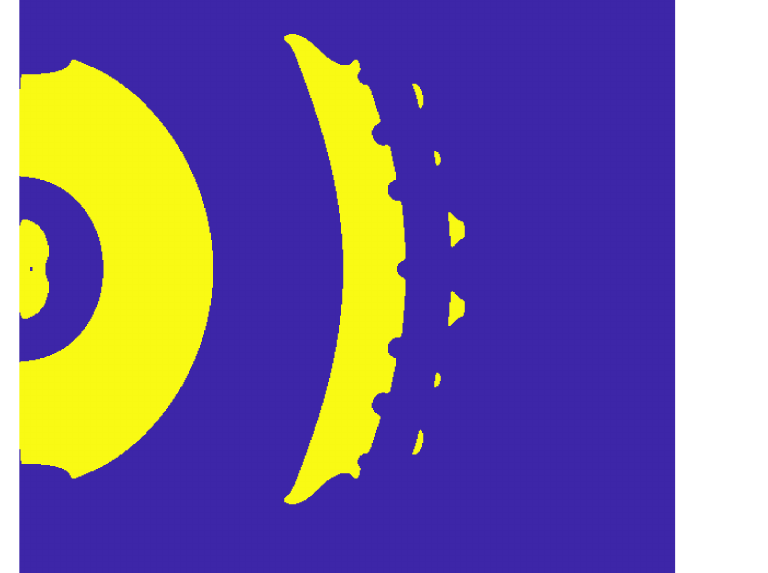}
	\includegraphics[width=0.3\linewidth,trim=2.5cm 8.5cm 3cm 8.5cm,clip]{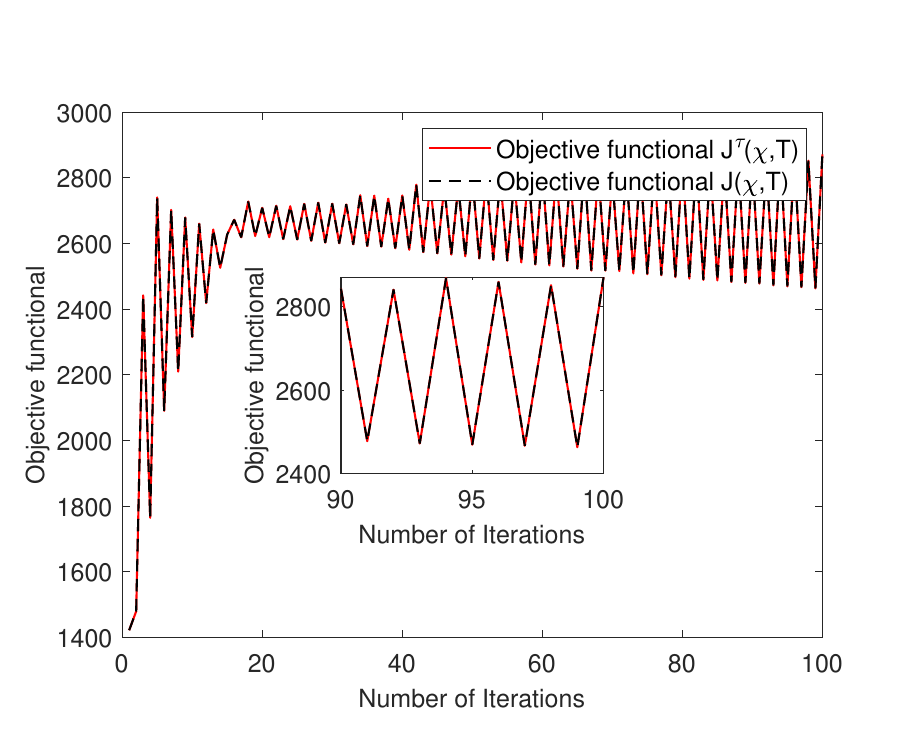}
	\caption{ The results uses ICTM on a $600\times 600$ grid. The parameters are set as $\tau = 1\times10^{-4}$, $\gamma = 30$ and $\xi=1\times10^{-5}$. Left:  The snapshot of $\chi$ at the $99$-th iteration. Middle:  The snapshot of $\chi$ at the $100$-th iteration. Right: The change of objective functional values during iterations. See Section~\ref{sec:failureICTM}.}
	\label{Example1_fail}
\end{figure}

\subsection{A prediction-correction based ICTM}\label{sec:pcICTM}

Based on the above argument on the failure of the ICTM, we propose a prediction-correction based ICTM for (\ref{o111}). From the above formal analysis, one can see the issue on the failure of the method is that the update of $T$ can not give a decrease of the objective functional. However, $T^k$ and $T^{*k}$ are determined by $\chi^k$ and they are usually uniquely solved from the systems \eqref{and} and \eqref{ad} with a given $\chi^k$. 

Therefore, to obtain a minimizing sequence, we propose to generate the following sequence
\begin{equation*}
	\chi^{1}, T^1, T^{*1}, \chi^{2}, T^2, T^{*2},\cdots,\chi^{k}, T^k, T^{*k},\cdots
\end{equation*}
such that $J^\tau(\mathbf{\chi}^{k+1},T^{k+1})\leq J^\tau(\mathbf{\chi}^{k},T^{k})$. That is, we explicitly impose the monotonically decay of the objective functional as a constraint in the iteration. 

In the previous section, we discussed that $\chi^{k+1}$ is updated as follows
\[\chi^{k+1} = \arg\min_{\chi \in \mathcal{B}} \tilde{J}^{\tau,k}(\chi)\]
which is independent of $T^{k+1}$. 

In this section, we propose to solve the following problem instead:
\begin{align*}
\chi^{k+1} = & \arg\min_{\chi \in \mathcal{B}} \tilde{J}^{\tau,k}(\chi) \\
 s.t. & \ J^\tau(\mathbf{\chi}^{k+1},T^{k+1})\leq J^\tau(\mathbf{\chi}^{k},T^{k}).
\end{align*}
Because $T$ and $T^{*}$ are implicitly determined by $\chi$ from systems \eqref{and} and \eqref{ad}, one can not directly determine a $\chi^{k+1}$ to have the decay of the objective functional. We thus consider a prediction-correction based method by firstly predicting a candidate for $\chi^{k+1}$ and then doing a correction to have the decay of the objective functional based on the prediction.

Based on the derivation of the original ICTM and the first order necessary condition of the approximate problem \eqref{o111}, to update $\chi^{k+1}$, we define the following {\bf prediction} set according to $T^k$ and $T^{*k}$:
\begin{align*}
A&:=\big\{p\in \mathcal{N},\Phi^k \leq \sigma, \chi^{k}(p)=0\big\},\\
B&:=\big\{p\in \mathcal{N},\Phi^k > \sigma, \chi^{k}(p)=1\big\},
\end{align*}
where $\mathcal{N}$ is the set of all vertices of the grid. Note that $A$ defines the set that $\chi$ tends to change from $0$ to $1$ and $B$ defines the set that $\chi$  tends to change from $1$ to $0$ in the $(k+1)$-th iteration. Because $\int_\Omega \chi \ d\textbf{x}$ is preserved, the size of $A$ should be equal to the size of $B$ ({\it i.e.}, $|A| = |B|$). 
In Algorithm~\ref{a:MBO},  $\chi^{k+1}$ was updated by 
\begin{align*}
\chi^{k+1} = \chi_A+\chi^k-\chi_B.
\end{align*}

From above numerical and theoretical observation, we observe that this update may not guarantee the decay of the objective functional value. To be consistent with the first order necessary condition, we need to do a correction to find $\tilde A \subset A$ where $\chi$ changes from $0$ to $1$ and $\tilde B \subset B$ where $\chi$ changes from $1$ to $0$ during the $(k+1)$-th iteration, subject to the principle that the objective functional decays in the iteration. Based on the principle of the decay of the objective functional value, we define the following  {\bf correction} set according to  $T^k$, $T^{*k}$, $\sigma$, $A$, and $B$:
\begin{align*}
\tilde A_1&:=\big\{p\in \mathcal{N},\Phi^k \leq \sigma_1, \chi^{k}(p)=0\big\},\\
\tilde A_2&:=\big\{p\in \mathcal{N},\Phi^k \in (\sigma_1,\sigma), \chi^{k}(p)=0\big\},\\
\tilde B_1&:=\big\{p\in \mathcal{N},\Phi^k > \sigma_2, \chi^{k}(p)=1\big\},\\
\tilde B_2&:=\big\{p\in \mathcal{N},\Phi^k \in (\sigma, \sigma_2), \chi^{k}(p)=1\big\},
\end{align*}
where $\sigma_1$ and $\sigma_2$ are determined by the decay of the objective functional value. Then, the $(k+1)$-th iteration is defined by 
\begin{align*}
\chi^{k+1} = \chi_{\tilde A_1}+\chi^k-\chi_{\tilde B_1}.
\end{align*}
A diagram of the process of correction is given in Figure~\ref{fig:process_correction}. 
\begin{figure}
    \centering{
    \includegraphics[width=0.3\linewidth,trim=5cm 13.5cm 5cm 4cm,clip]{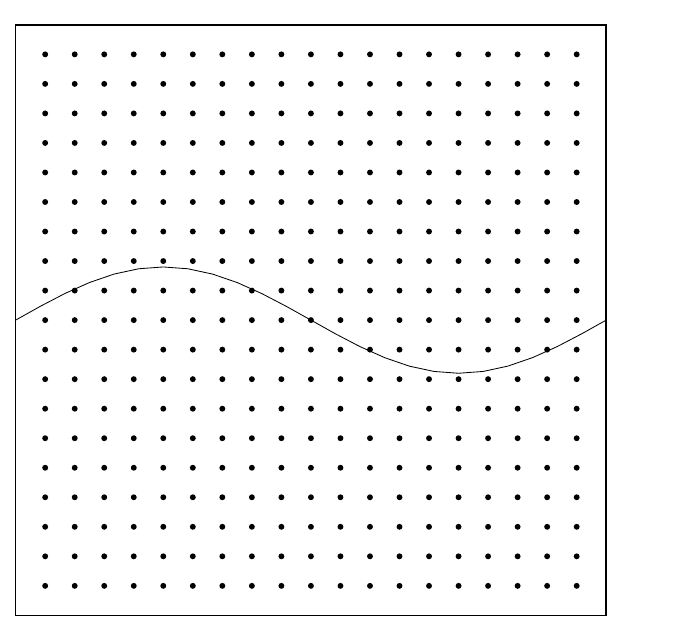}
    \includegraphics[width=0.3\linewidth,trim=5cm 13.5cm 5cm 4cm,clip]{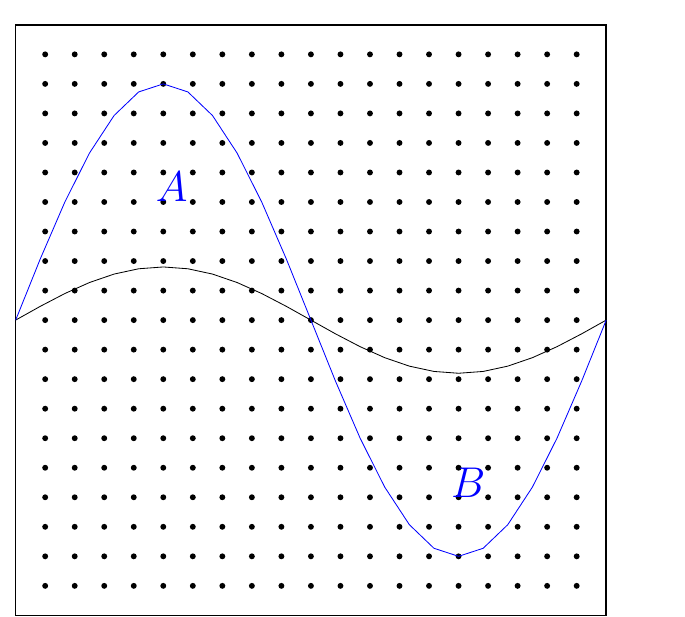}
    \includegraphics[width=0.3\linewidth,trim=5cm 13.5cm 5cm 4cm,clip]{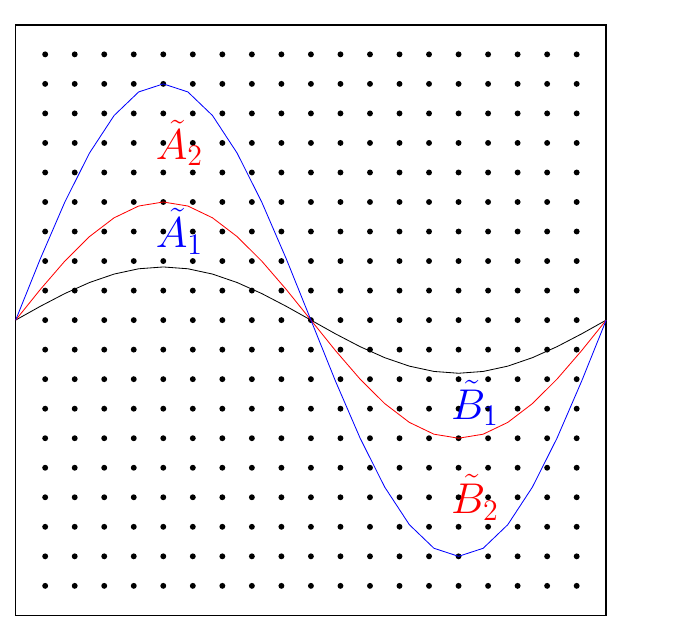}
    \caption{The process of prediction-correction. Left: the $k$-th iteration.  Middle: the prediction using ICTM ({\it i.e.}, sets $A$ and $B$). Right: the correction for the $(k+1)$-th iteration ({\it i.e.}, sets $A_1$ and $B_1$). See Section~\ref{sec:pcICTM}.}
    \label{fig:process_correction}}
\end{figure}

To find $\sigma_1$ and $\sigma_2$ to have the decay of the objective functional, we consider the discrete problem where $\tilde A_2$ and $\tilde B_2$ contain discrete points. Again, because the preservation of $\int_\Omega \chi \ d \textbf{x}$, the size of $\tilde A_2$ should be the same as the size of $\tilde B_2$ in terms of number of discrete points. Hence, the values of $\sigma_1$ and $\sigma_2$ can be implicitly determined by the size of $\tilde A_2$ and $\tilde B_2$. Denote $N$ by the number of discrete points in $A$, we first sort the values of $\phi^k$ in $A$ in a descending order and the values of  $\phi^k$ in $B$ in an ascending order. We propose to iteratively find the proper $\tilde A_2$ and $\tilde B_2$ to have the decay of the objective functional. To be specific, we start to set a small number of points in $\tilde A_2$ and $\tilde B_2$, and gradually increase the number of point until the objective functional value decays. The algorithm is summarized in Algorithm~\ref{a:correction} and the algorithm for the whole prediction-correction based ICTM is then summarized in Algorithm~\ref{a:prediction-correction}.

\begin{algorithm}[h!]
\DontPrintSemicolon
 \KwIn{$\chi^k$, $\tau>0$, $\xi >0$, $\kappa_1,\kappa_2, q_1,q_2$, $T^k$, $T^{*k}$, $A$, $B$, $N$: number of points in $A$, $J^{\tau}(\mathbf{\chi}^{k},T^{k})$, $\Phi^k$, $\theta\in (0,1)$, $s = 1$.}
 \KwOut{$\chi^{k+1} \in \mathcal{B}$.}
\If{$J^{\tau}(\tilde \chi_s,\tilde T_s)>J^{\tau}(\chi^{k},T^{k})$ }{
{\bf 1.} Sort the values of $\phi^k$ in $A$ in a descending order (denoted by $\tilde \phi_A^k$) and the values of  $\phi^k$ in $B$ in an ascending order (denoted by $\tilde \phi_B^k$),  \\
{\bf 2.} Compute $n_s = N-\lfloor N*\theta^s\rfloor$, \\
{\bf 3.} Set $\sigma_1 = \tilde \phi_A^k(n_s)$ and $\sigma_2 =  \tilde \phi_B^k(n_s)$, \\
{\bf 4.} Set	 \[A_1 = \{p\in A: \phi_A(p)\leq \sigma_1\} \ \ \ \ B_1 = \{p\in B: \phi_B(p)\leq \sigma_2\}\]
and set $\tilde \chi_s = \chi_{A_1}+\chi^k-\chi_{B_1}$,\\
{\bf 5.} Calculate $\tilde T_s$ based on $\tilde \chi_s$ and compute $J^{\tau}(\tilde \chi_s,\tilde T_s)$,\\
{\bf 6.} Set $s = s+1$.}
Set $\chi^{k+1} = \tilde \chi_s$.
\caption{The correction step in the ICTM. }
\label{a:correction}
\end{algorithm}

\begin{algorithm}[h!]
\DontPrintSemicolon
 \KwIn{$\chi^0$, $\tau>0$, $\xi >0$, $\kappa_1,\kappa_2, q_1,q_2$, $\theta\in (0,1)$, $k = 0$, the tolerance $tol>0$.}
 \KwOut{$\chi^{\ast} \in \mathcal{B}$ that approximates the optimal solution.}
 \While{$\|\chi^{k+1}-\chi^k\|_2>tol$}{
{\bf 1.} Prediction.\\
{\bf 1.1} Compute 
 \[\Phi^k = (q_1-q_2)G_\tau\ast(T^k-T^{*k})+\gamma\sqrt{\frac{\pi}{\tau}}G_\tau\ast (1-2\chi^k) +(k_1-k_2) G_\tau\ast(\frac{\xi}{2}\nabla T^k\cdot\nabla T^k+\nabla T^k\cdot\nabla T^{\ast k}) \]
 and 
 \begin{align*}
J^{\tau}(\mathbf{\chi}^{k},T^{k}) = &\int_\Omega(q(\chi^{k})T^{k})\ d\textbf{x} + \frac{\xi}{2}\int_\Omega\kappa(\chi^{k})\nabla T^{k}\cdot \nabla T^{k} \ d\textbf{x}+\gamma\sqrt{\frac{\pi}{\tau}}\int_\Omega \chi^{k} G_\tau\ast (1-\chi^{k})\ d\textbf{x}.
\end{align*}
{\bf 1.2} Find a proper $\sigma$ such that 
\begin{align*}
A&:=\big\{p\in \mathcal{N},\Phi^k \leq \sigma, \chi^{k}(p)=0\big\},\\
B&:=\big\{p\in \mathcal{N},\Phi^k > \sigma, \chi^{k}(p)=1\big\}.
\end{align*}
{\bf 1.3} Compute $N$, the number of discrete points in $A$.\\
{\bf 2.} Correction.\\
Run the correction step defined in Algorithm~\ref{a:correction}.
 }
\caption{The prediction-correction-based ICTM approximating optimizers of \eqref{o111}. }
\label{a:prediction-correction}
\end{algorithm}

\begin{remark}
Note that in this algorithm, the worst case in some iteration is that one can not find two sets $A_1$ and $B_1$ to have an objective functional decay. This indeed usually indicates the convergence of the algorithm numerically. With the exception of this case, it is generally possible to find a new update that minimizes the objective function subject to the constraints. This novel prediction-correction based ICTM should also work for general optimal design problems with multi-physics constraints. In this paper, we firstly focus on the steady heat-transfer optimization problem to introduce the idea of the prediction-correction-based ICTM and extend it to more complicated problems elsewhere.
\end{remark}

\begin{remark}
Another intuition of the prediction-correction based ICTM is to consider the proximal map or implicit update of the next step. In the continuous level of the original ICTM, the variation of the objective functional with respect to the indicator function is correct. However, one then simply changes the values from 1 to 0 when the variation is less than 0, and vice versa. This can be interpreted as a gradient descent method with a sufficiently large step together with projection. That is, with a sufficiently large step size $s$
\begin{equation}
\chi^{k+1} = Proj_{[0,1]}  \left(\chi^{k}- s \left. \frac{\delta J^\tau}{ \delta \chi}\right|_{\chi^k}\right),
\end{equation}
where 
\begin{equation*}
Proj_{[0,1]}(v) =
\begin{cases}
v  & {\rm if} \ v\in [0,1], \\
0 & {\rm if} \ v<0, \\
1 & {\rm if} \ v>1.
\end{cases}
\end{equation*} 
Obviously, for a sufficiently large step size $s$, this is equivalent to the original ICTM. This is an explicit way to find the approximate solution to minimize the objective functional and this only works when the objective functional has some properties (for example, concavity or unimodality) with respect to the indicator function. Thus, based on the gradient information, one can interpret the prediction-correction based ICTM is an implicit way to update the next step by penalizing the change between two steps.
\end{remark}

\section{Numerical experiments}
\label{sec:numerical}

In this section, we first show the implementation of Algorithm~\ref{a:prediction-correction}. We use the continuous piecewise linear finite element space to approximate the problems (\ref{and}) and (\ref{ad}). Let $\mathcal{T}_h$ be a uniform triangulation of the domain $\Omega$, and $\mathcal{N}_h$ is the set of all vertices of $\mathcal{T}_h$. For a given $\chi_h \in \mathcal{B}_h$ where $\mathcal{B}_h$ is the discrete version of $\mathcal{B}$ defined on $\mathcal{N}_h$. We use the fast Fourier transform (FFT) for the computation of convolutions.

We introduce the continuous piecewise linear finite element spaces as follows:
\begin{align*}
	V_h = \{v\in H^1(\Omega) \ | \ v \in P_1(K), \forall K\in \mathcal{T}_h \},\quad
	V_h^{0} = V_h\cap H^1_{\Gamma_D}(\Omega),
\end{align*}
where $P_1(K)$ is the linear polynomial on $K$ and $H^1_{\Gamma_D}(\Omega) = \{ v\in H^1(\Omega) \ | \ v |_{\Gamma_D}=0 \}$. For the solution of $(\ref{and})$ and $(\ref{ad})$, we find $T_h\in V_h^0$ and $T^\ast_h\in V_h^0$ such that
\begin{align*}
(\kappa(\chi_h)\nabla T_h,\nabla \varphi_h) &= (q(\chi_h),\varphi_h),\ \ \forall \varphi_h\in V_h^0,\\
 -(\kappa(\chi_h)\nabla T^\ast_h,\nabla \varphi_h) &= (q(\chi_h),\varphi_h) + \xi(\kappa(\chi_h)\nabla T_h,\nabla \varphi_h),\ \ \forall \varphi_h\in V_h^0.
\end{align*}
When $T_h$ and $T^\ast_h$ are obtained, we use the FFT to evaluate $\Phi_h$ on each node of $\mathcal{N}_h$ as follows:
 \begin{equation*}				
 \Phi_h=(q_1-q_2)G_\tau\ast(T_h -T^\ast_h)+\gamma\sqrt{\frac{\pi}{\tau}}G_\tau\ast (1-2\chi_{h}) +(\kappa_1-\kappa_2)G_\tau\ast(\frac{\xi}{2}\nabla T_h\cdot\nabla T_h+\nabla T_h\cdot\nabla T_h^{\ast }).
\end{equation*}
Next we can update the indicator function $\mathbf{\chi}_h$ using $\Phi_h$ as the approach stated in Algorithm~\ref{a:prediction-correction}.

In the remainder of this section, we test the following three examples to demonstrate the performance of the proposed method:
\begin{itemize}
\item[1.] An area-to-point problem as considered in previous works \cite{boichot2016genetic,ikonen2018topology,ullah2021parametrized} where the ``point" indicates a short edge on the boundary. The Dirichlet boundary condition is imposed on a short edge of the left-hand side of the boundary while Neumann boundary conditions are imposed on the remaining boundary, as depicted in Figure~\ref{m1}. 
\item[2.] An area-to-sides problem with only Dirichlet boundary as in \cite{guo2016optimization}. 
\item[3.] A three dimensional case as in \cite{burger2013three} which is a cubic domain for the volume-to-surface problem as shown in Figure~\ref{fig:model3}. The Dirichlet boundary is located at a small surface on the middle of bottom face while the Neumann boundary condition is imposed on the other part of the boundary.
\end{itemize}

 In order to do comparisons among optimal distributions with different parameters $\kappa_1,\ \kappa_2,\ q_1,\ q_2,\ \tau,\ \gamma,\ \xi$, we use non-dimensional thermal conductivity and heat generation. In each iteration of the correction step, we choose half of the points of the remaining prediction sets that transfer from $\Omega_1$ to $\Omega_2$ or from $\Omega_2$ to $\Omega_1$ during the iteration in the correction step ({\it i.e.}, $\theta = 1/2$).

\subsection{Area-to-point problem}\label{Area-to-point_ex}
In this example as shown in Figure~\ref{m1}, the length of side of the square domain is $l=1$, and the Dirichlet boundary is located in the middle of the left boundary with $d = 0.1 l$. The domain, enclosed by adiabatic boundaries, allows heat conduction only through the Dirichlet boundary. The heat generation rate is different in $\Omega_1$ and $\Omega_2$. In order to verify the effectiveness of Algorithm~\ref{a:prediction-correction}, we test the initial distribution of $\chi$ shown in Figure~\ref{X1} with Algorithm~\ref{a:MBO} (original ICTM) and Algorithm~\ref{a:prediction-correction} (prediction-correction based ICTM) individually. Furthermore, we investigate the influences of $\frac{\kappa_1}{\kappa_2}$, $\frac{q_1}{q_2}$, mesh size, volume fraction $\beta$, $\tau$, $\gamma$, and the choice of initial distribution of $\chi$. The regularization parameter $\xi$ is set to be $1\times10^{-5}$ in all cases for this example.

\begin{figure}[ht!]
   \centering
   \includegraphics[width=0.3\linewidth,trim=6cm 10cm 5cm 10cm,clip]{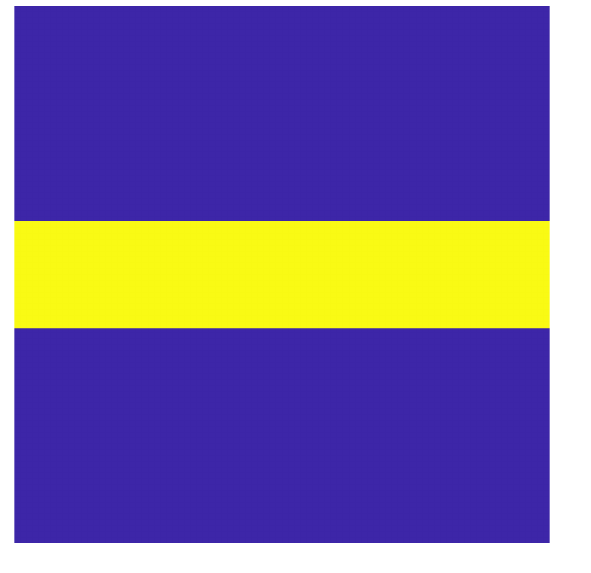}
   \caption{The initial distribution of $\chi_1.$ See Section~\ref{Area-to-point_ex}.}
   \label{X1}
\end{figure}

\subsubsection{Comparison between the classical ICTM and the prediction-correction based ICTM}\label{sec:comparison}
 We first compare Algorithm~\ref{a:MBO} and Algorithm~\ref{a:prediction-correction} by the choices of $\gamma=50$ and $30$ to show the performance of two algorithms. We set $\kappa_1=10, \ \kappa_2=1, \ q_1=1, \ q_2=100, \ \tau=1\times10^{-4}$, and the volume fraction $\beta=0.2$ on a $600\times600$ grid. For the initial distribution of $\chi$ in Figure~\ref{X1}, we get the optimal distribution of $\chi$ and objective functional curves as shown in Figures~\ref{Example1_fail}, \ref{Example1.11}, \ref{Example1.1}, and \ref{Example1.3}. The yellow region indicates the high conductivity subdomain with low heat generation rate, and the blue region is the low conductivity subdomain with high heat generation rate. The objective functional $J(\chi,T)$ and $J^\tau(\chi,T)$ are defined in (\ref{o1}) and (\ref{o12}) respectively, where the $|\Gamma|$ in $J(\chi,T)$ is also approximated by the Gaussian convolution as that in $J^\tau(\chi,T)$. The difference between $J(\chi,T)$ and $J^\tau(\chi,T)$ is only the term $\frac{\xi}{2}\int_{\Omega} \kappa(\chi)\nabla T \cdot \nabla T \ d\textbf{x}$.
 
 For the case of $\gamma=50$, the optimal distributions of $\chi$ shown in Figure~\ref{Example1.11} and Figure~\ref{Example1.1} are similar, but the objective functional is not always stable by Algorithm~\ref{a:MBO}, while the objective functional value strictly decays by Algorithm~\ref{a:prediction-correction}. 
 
For the case of $\gamma=30$, from Figure~\ref{Example1_fail}, we can see that the objective functional always oscillates during the iteration using Algorithm~\ref{a:MBO} and the distribution of $\chi$ can not converge. However, if we apply Algorithm \ref{a:prediction-correction}, the optimal distribution of $\chi$ can be obtained and the objective functional stably decays as shown in Figure~\ref{Example1.3}. From Figures~\ref{Example1.1} and \ref{Example1.3}, we observe that the objective functional $J(\chi,T)$ consistently yields smaller values compared to the objective functional $J^\tau(\chi,T)$, indicating that $J(\chi,T)$ possesses the decaying property observed in $J^\tau(\chi,T)$.

\begin{figure}[ht!]
    \centering
	\includegraphics[width=0.3\linewidth,trim=6cm 10cm 5cm 10cm,clip]{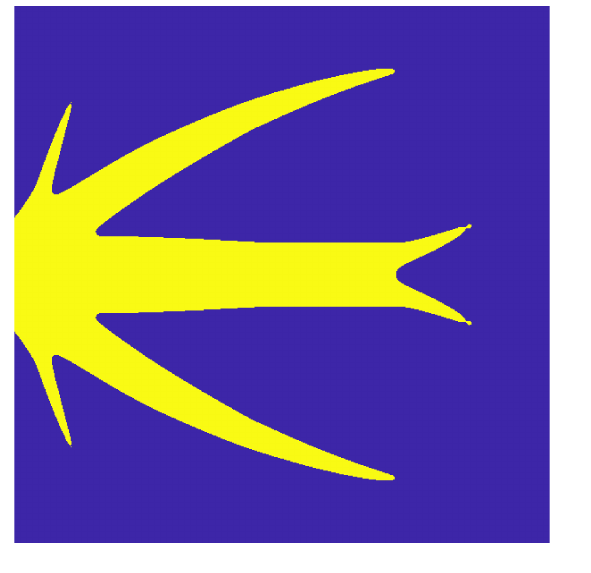}
	\includegraphics[width=0.4\linewidth,trim=2cm 8.5cm 3cm 10cm,clip]{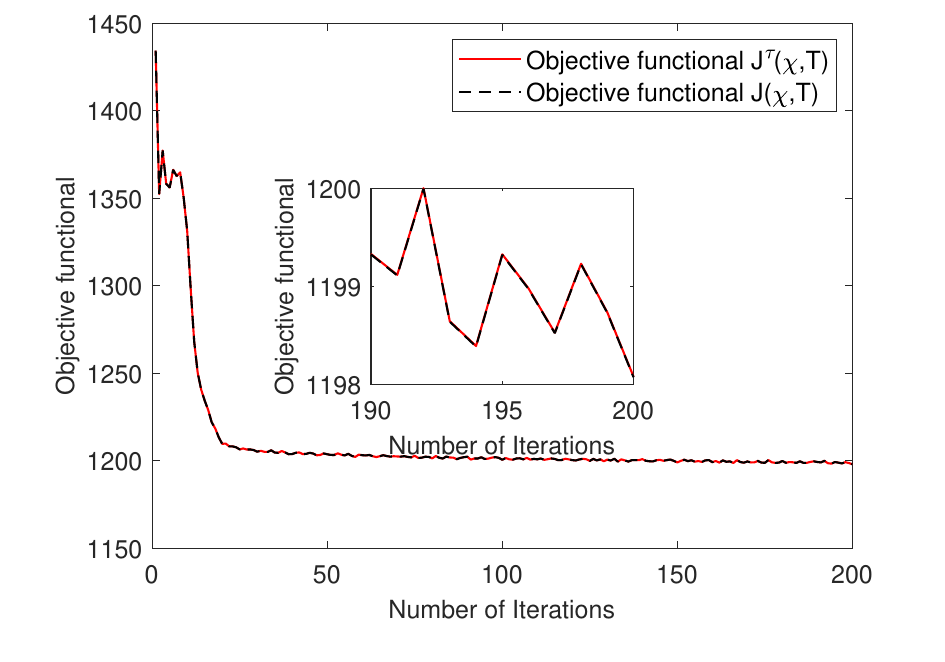}
	\caption{ The optimal results with $\gamma = 50$, $\kappa_1=10,\ \kappa_2=1,\ q_1=1,\ q_2=100$, volume fraction $\beta = 0.2$ on a $600\times 600$ grid using ICTM. Left: Approximate optimal solution. Right: Objective functional curves. See Section~\ref{sec:comparison}.}
	\label{Example1.11}
\end{figure}

\begin{figure}[ht!]
	\centering
	\includegraphics[width=0.3\linewidth,trim=6cm 10cm 5cm 10cm,clip]{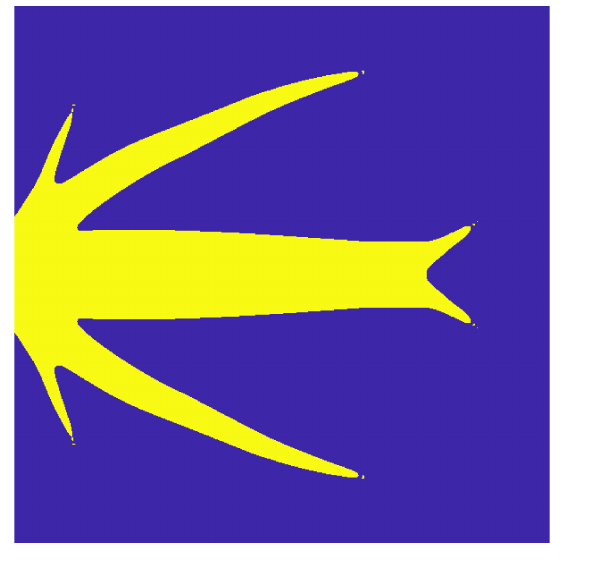}
	\includegraphics[width=0.4\linewidth,trim=2.5cm 8.5cm 3cm 10cm,clip]{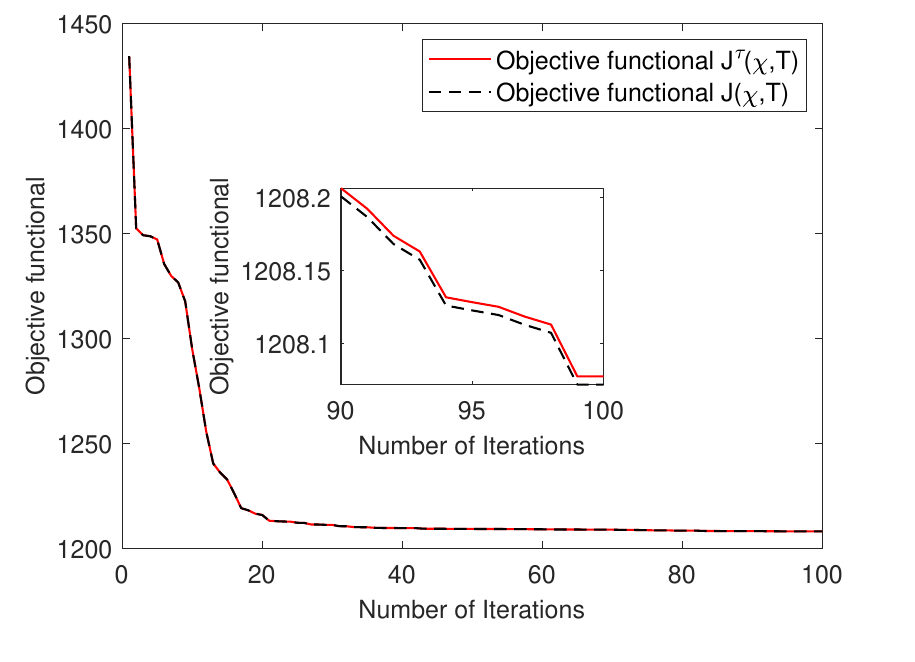}
	\caption{The optimal results with $\gamma = 50$, $\kappa_1=10,\ \kappa_2=1,\ q_1=1,\ q_2=100$, volume fraction $\beta = 0.2$ on a $600\times 600$ grid using the prediction-correction-based ICTM. Left: Approximate optimal solution. Right: Objective functional curves. See Section~\ref{sec:comparison}.}
	\label{Example1.1}
\end{figure}

\begin{figure}[ht!]
	\centering
    \includegraphics[width=0.3\linewidth,trim=6cm 10cm 5cm 10cm,clip]{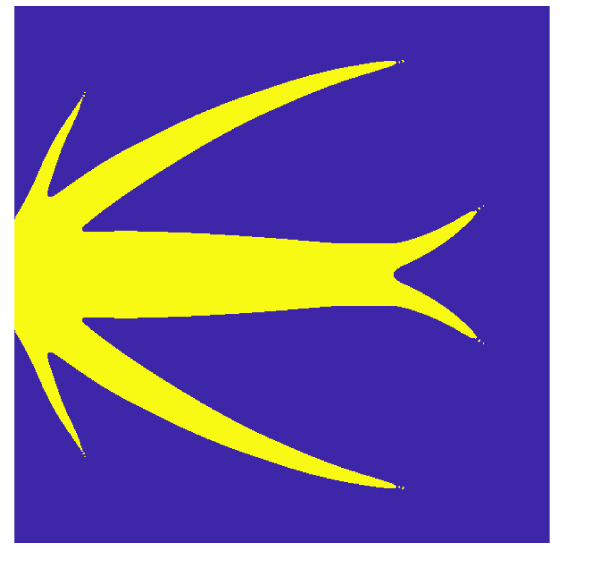}
    \includegraphics[width=0.4\linewidth,trim=2.5cm 8.5cm 3cm 10cm,clip]{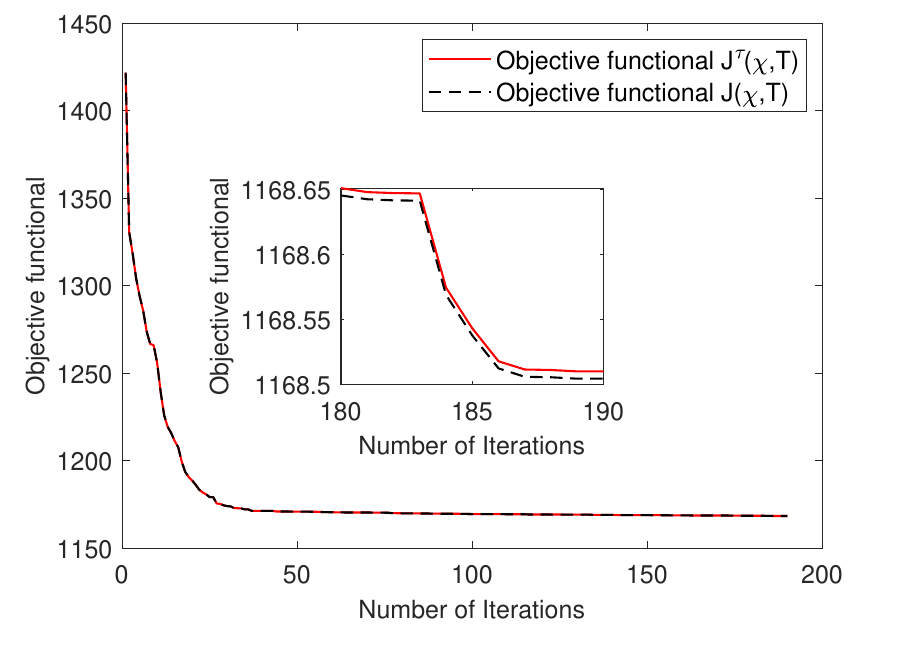}
	\caption{ The optimal results with $\gamma = 30$, $\kappa_1=10,\ \kappa_2=1,\ q_1=1,\ q_2=100$, volume fraction $\beta = 0.2$ on a $600\times 600$ grid using the prediction-correction-based ICTM. Left: Approximate optimal solution. Right: Objective functional curves. See Section~\ref{sec:comparison}.}
	\label{Example1.3}
\end{figure}

\subsubsection{The evolution of the profile during iterations}\label{sec:evolution}
Next, we investigate the change of $\chi$ together with the objective functional decay, with $\kappa_1=10,\ \kappa_2=1$, $q_1=1,\ q_2=100$, $\gamma = 15$, $\tau = 1\times10^{-4}$, and volume fraction $\beta = 0.2$ on a $600\times 600$ grid. The change of $\chi$ are displayed in Figure~\ref{E_tol} during the iteration. Figure~\ref{E_tol} demonstrates the decay of both the objective functional $J(\chi,T)$ and $J^\tau(\chi,T)$ throughout the iteration process. One can see that the algorithm converges fast in about 100-200 steps to obtain a stationary solution and the profiles at the 100-th and 200-th iterations are quite similar. 

\begin{figure}[ht!]
   \centering
   \includegraphics[width=0.9\linewidth]{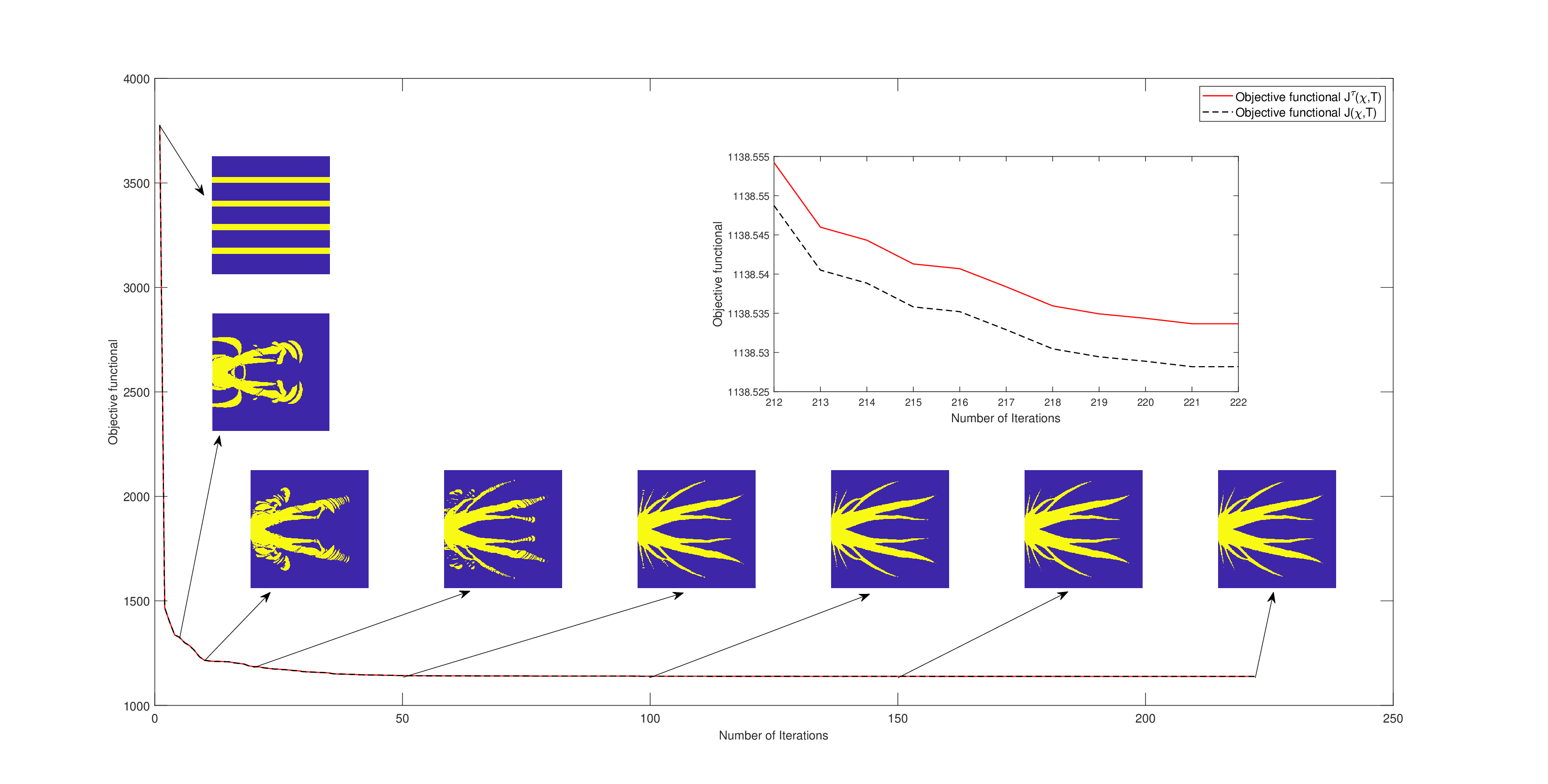}
   \caption{Change of $\chi$ and objective functional value during iterations. See Section~\ref{sec:evolution}.}
   \label{E_tol}
\end{figure}

\subsubsection{Profile dependency on parameters}\label{sec:dependonpara}
We further investigate the dependence of the approximate optimal solutions of $\chi$ on $\frac{\kappa_1}{\kappa_2}$, $\frac{q_1}{q_2}$, mesh size, volume fraction $\beta$, $\tau$, $\gamma$ and the initial distribution. 

{\bf Dependency on $\kappa_1/\kappa_2$:} In order to investigate the dependence of $\kappa_1/\kappa_2$ for the approximate optimal solution of $\chi$ with the same initial guess as that in Figure~\ref{E_tol}, we fix $q_1 = 1$, $q_2 = 100$, $\gamma = 15$, $\beta = 0.2$, $\tau = 10^{-4}$ and change $\kappa_1/\kappa_2$ from $5$, $10$ to $20$. We observe from Figure~\ref{Example8.2} that the solution contains finer branches for larger $\frac{\kappa_1}{\kappa_2}$, which is consistent with the results reported in \cite{boichot2016genetic}. 
\begin{figure}[ht!]
      \centering
	 \includegraphics[width=0.28\linewidth,trim=4cm 10cm 3cm 10cm,clip]{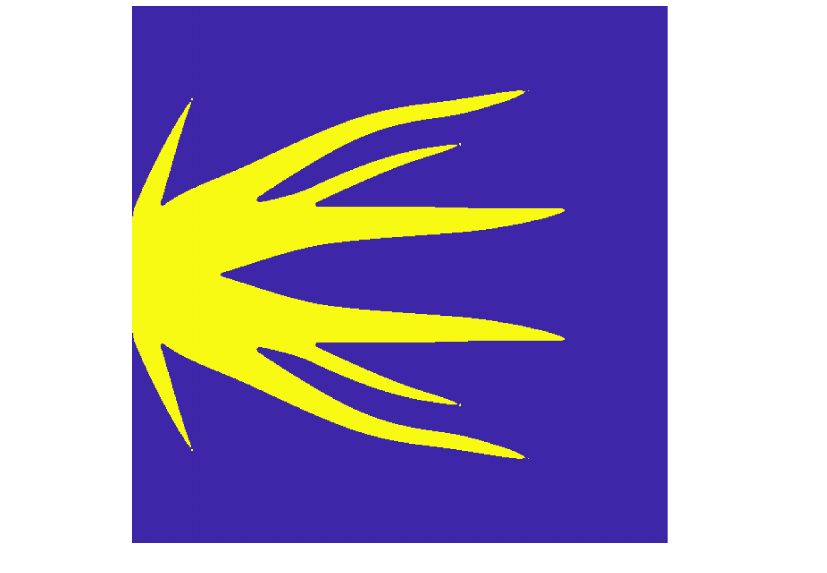}
     \includegraphics[width=0.28\linewidth,trim=4cm 10cm 3cm 10cm,clip]{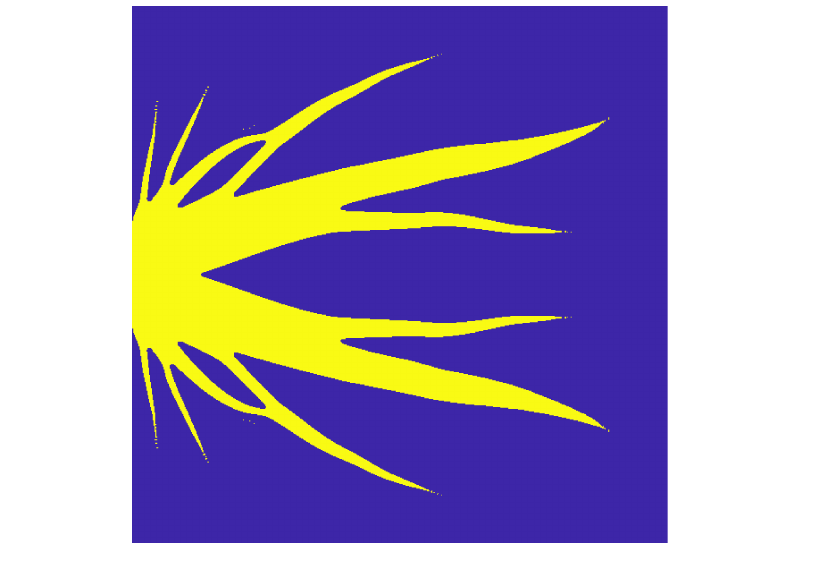}
	 \includegraphics[width=0.28\linewidth,trim=4cm 10cm 3cm 10cm,clip]{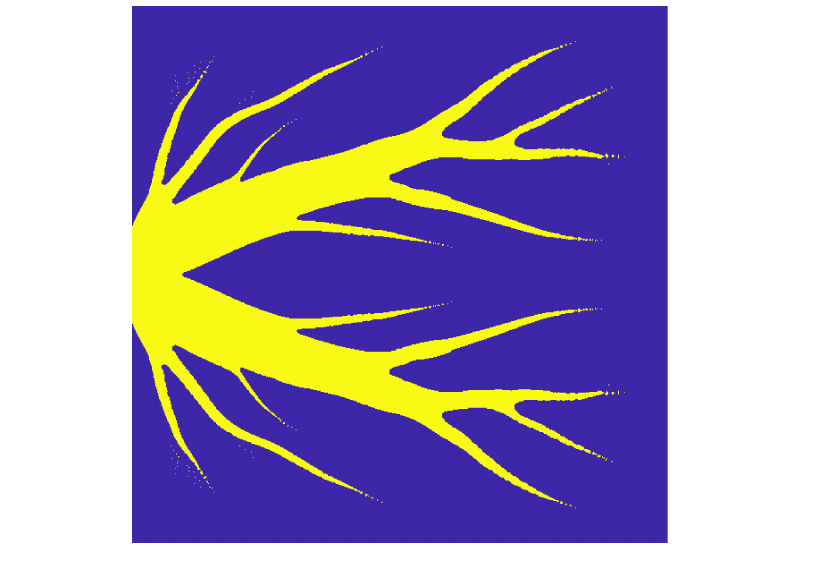}\\
     \includegraphics[width=0.28\linewidth,trim=2.5cm 9cm 3cm 10cm,clip]{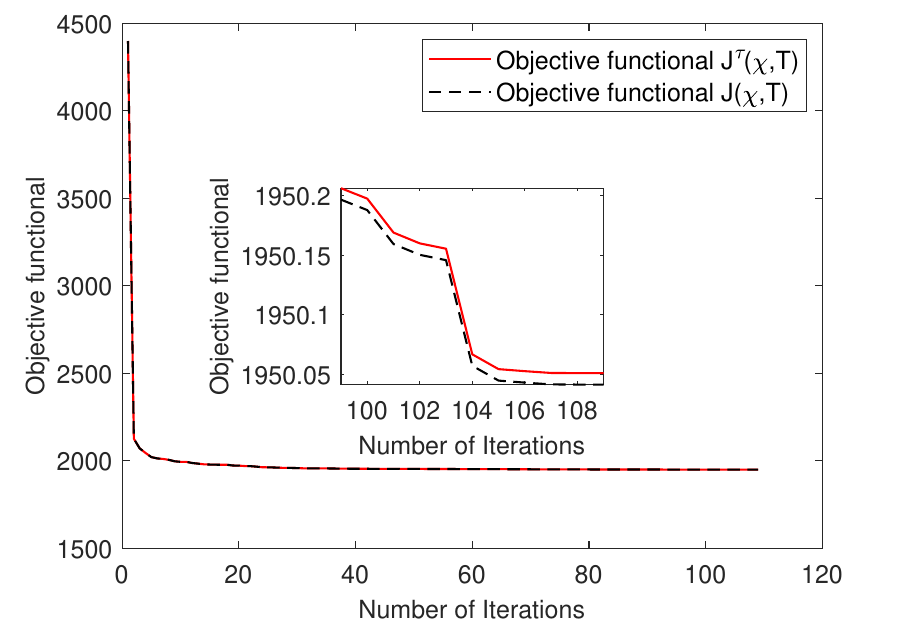}
     \includegraphics[width=0.28\linewidth,trim=2.5cm 9cm 3cm 10cm,clip]{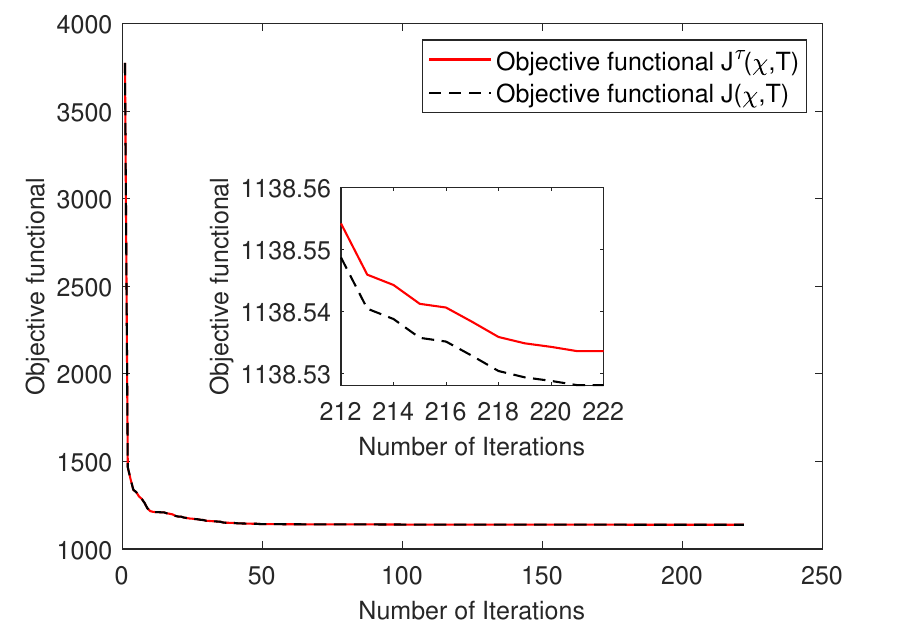}
	 \includegraphics[width=0.28\linewidth,trim=2.5cm 9cm 3cm 10cm,clip]{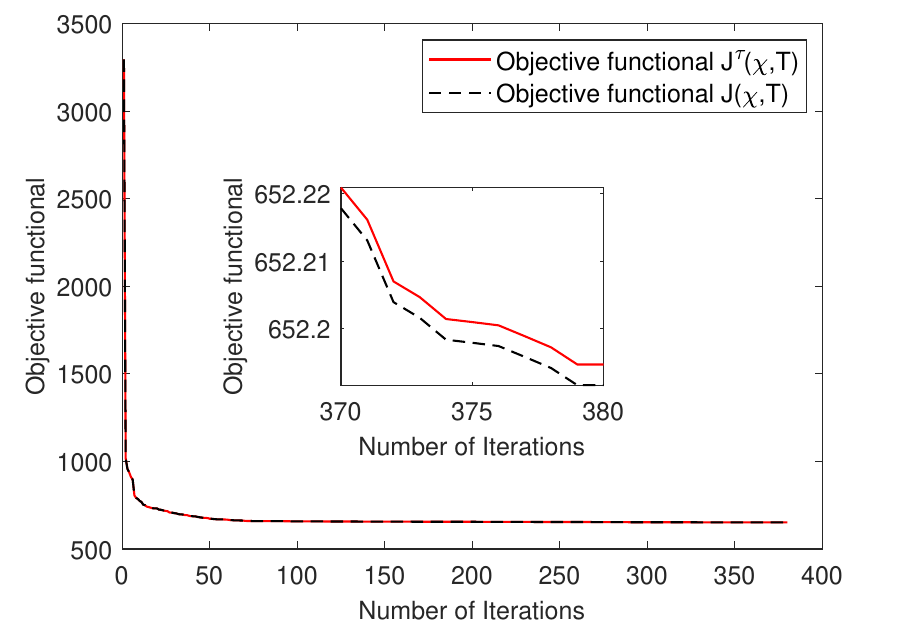}
	\caption{Comparing the impact of $\kappa_1/\kappa_2$ on the optimal distribution of $\chi$ and energy on a $600\times 600$ grid with $q_1=1,\ q_2=100$, $\gamma = 15$, volume fraction $\beta=0.2$ $\tau = 1\times 10^{-4}$. From left to right: the approximate optimal solution and objective functional curves for $\kappa_1 = 5,\ 10,\ 20$ and $\kappa_2=1$. See Section~\ref{sec:dependonpara}.}
	\label{Example8.2}
\end{figure}

{\bf Dependency on $q_1/q_2$:} Similarly, for the dependence of $q_1/q_2$ on the optimal distribution of $\chi$, we fix $\kappa_1=10,\ \kappa_2=1$, $\gamma = 15$, volume fraction $\beta = 0.2$, $\tau = 1\times10^{-4}$ on a mesh size $600 \times 600$ and change $q_1/q_2$  from $40$, $80$, to $100$. Different from the ratio $k_1/k_2$, one can observe that $\chi$ contains finer branches with a smaller ratio of $q_1/q_2$, as shown in Figure~\ref{Example8.1}.

\begin{figure}[ht!]
	\centering
        \includegraphics[width=0.28\linewidth,trim=4cm 10cm 3cm 10cm,clip]{fig/2D/model1/k/kX10.pdf}
	\includegraphics[width=0.28\linewidth,trim=4cm 10cm 3cm 10cm,clip]{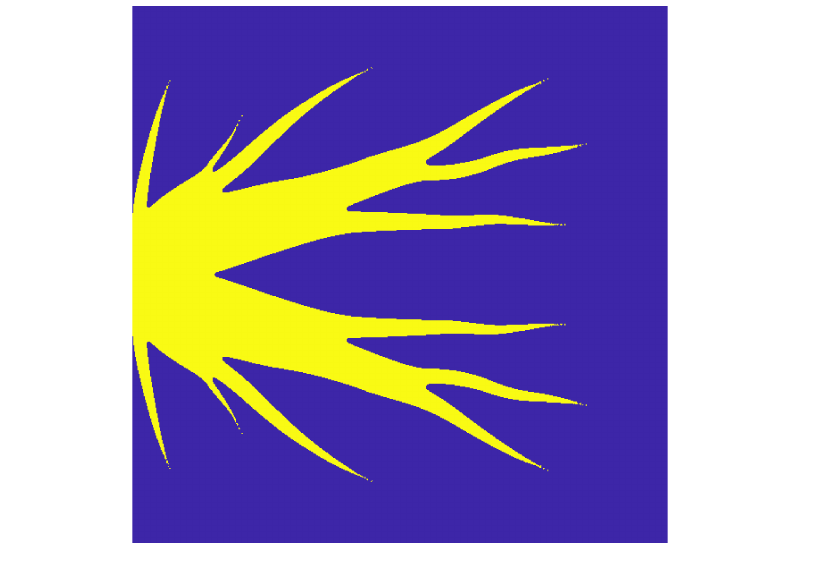}
	\includegraphics[width=0.28\linewidth,trim=4cm 10cm 3cm 10cm,clip]{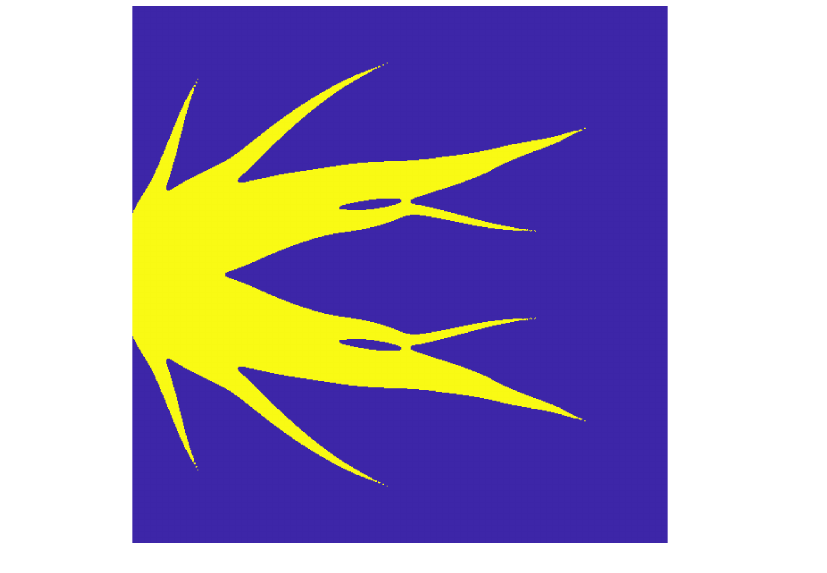}\\
        \includegraphics[width=0.28\linewidth,trim=2.5cm 9cm 3cm 10cm,clip]{fig/2D/model1/k/kE10.pdf}
        \includegraphics[width=0.28\linewidth,trim=2.5cm 9cm 3cm 10cm,clip]{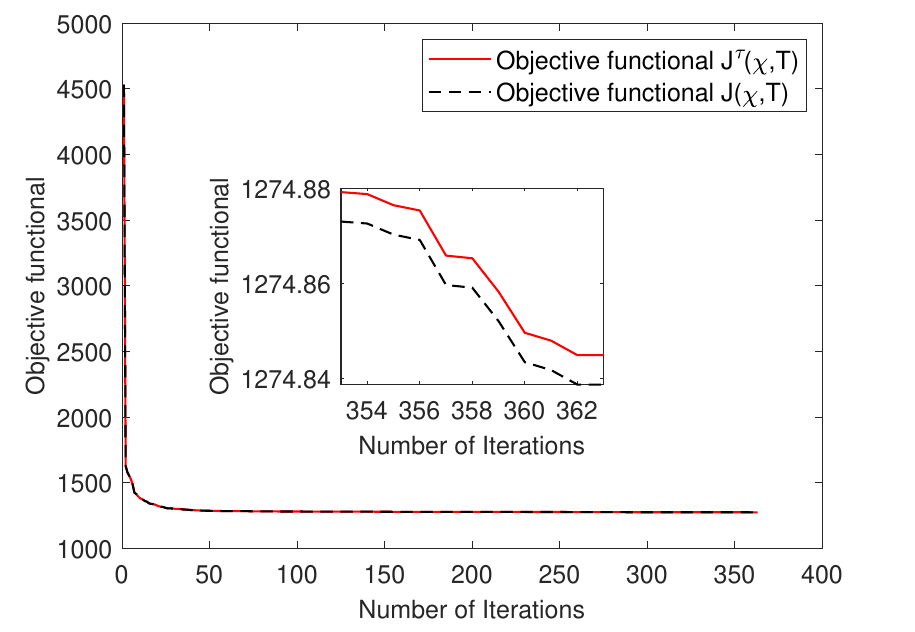}
	\includegraphics[width=0.28\linewidth,trim=2.5cm 9cm 3cm 10cm,clip]{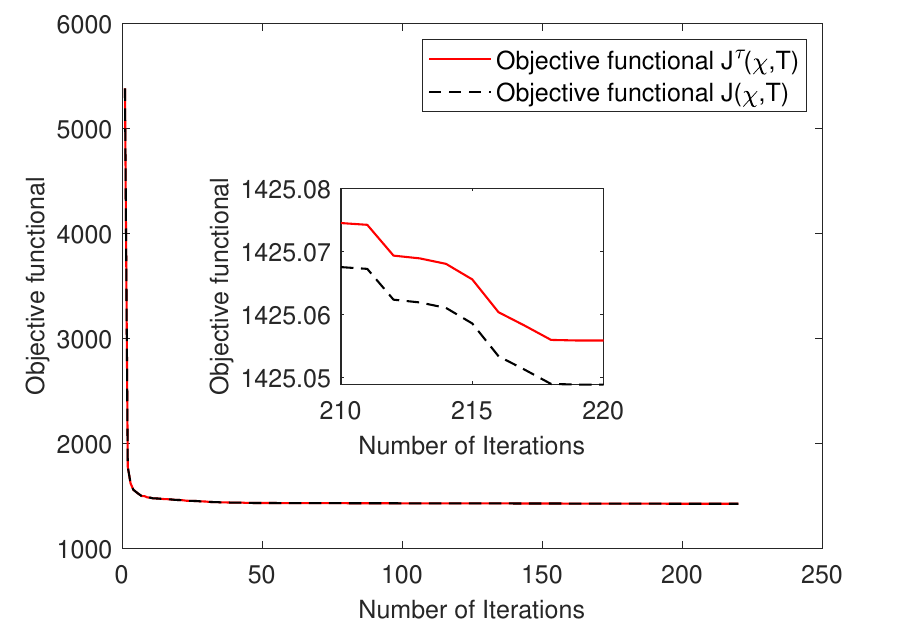}
	 \caption{Comparing the impact of $q_1/q_2$ on the approximate optimal solutions and objective functional values on a $600\times 600$ grid with $\kappa_1=10,\ \kappa_2=1$, $\gamma = 15$, volume fraction $\beta = 0.2$ and $\tau = 1\times10^{-4}$. Left to right: Approximate solutions and objective functional value curves for for $q_1=1,\ 40,\ 80$, and $q_2=100$. See Section~\ref{sec:dependonpara}.}
	\label{Example8.1}
\end{figure}


{\bf Dependency on mesh size:} As for the impact on the mesh size, Figure~\ref{mesh} displays the solution and the objective functional decaying property on the different grids. We observe that the solution remains stable when the mesh is refined and the objective functional decays with a similar profile. 

\begin{figure}[ht!]
	\centering
        \includegraphics[width=0.28\linewidth,trim=4cm 10cm 3cm 10cm,clip]{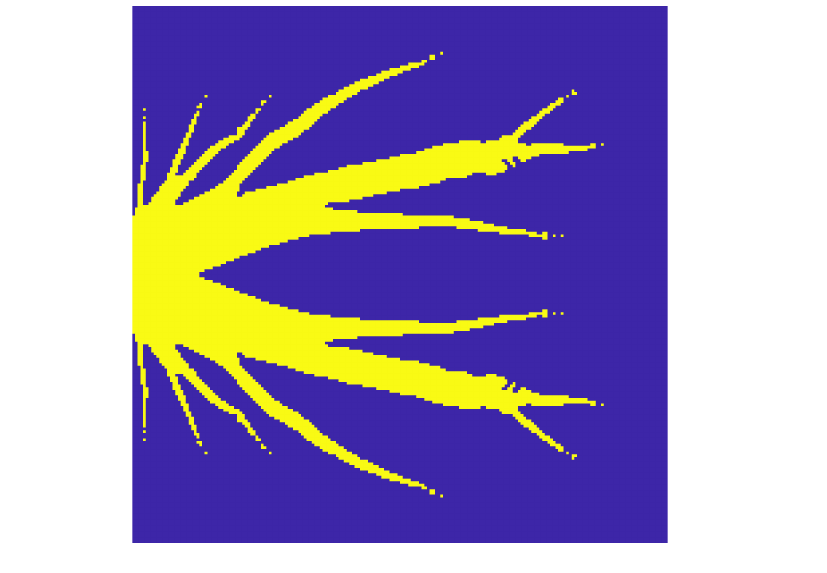}
	\includegraphics[width=0.28\linewidth,trim=4cm 10cm 3cm 10cm,clip]{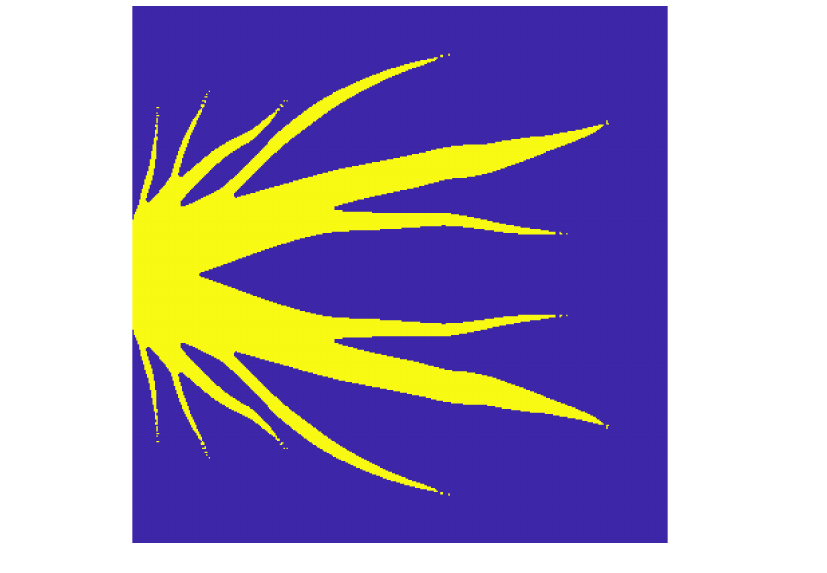}
         \includegraphics[width=0.28\linewidth,trim=4cm 10cm 3cm 10cm,clip]{fig/2D/model1/k/kX10.pdf}\\
        \includegraphics[width=0.28\linewidth,trim=2.5cm 9cm 3cm 10cm,clip]{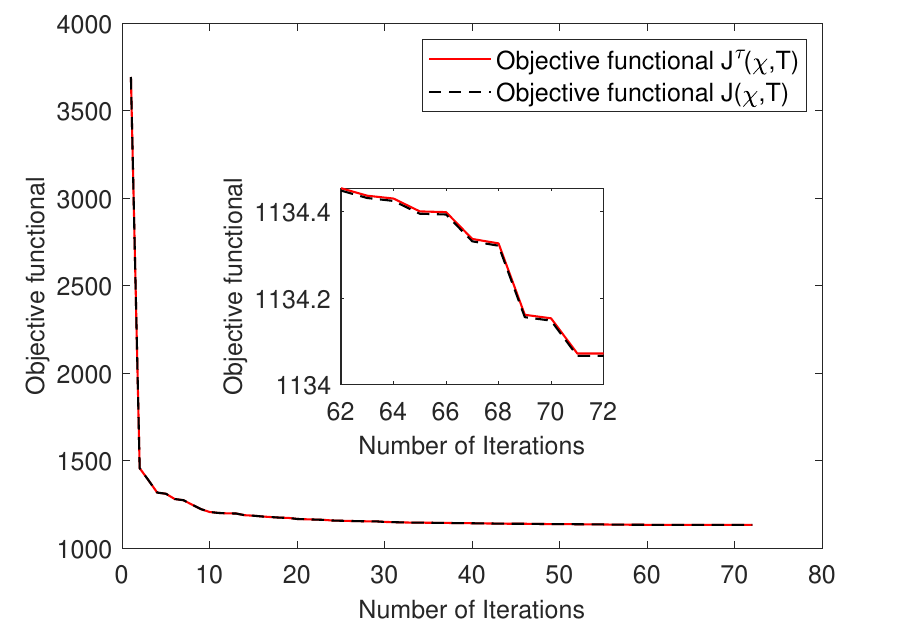}
        \includegraphics[width=0.28\linewidth,trim=2.5cm 9cm 3cm 10cm,clip]{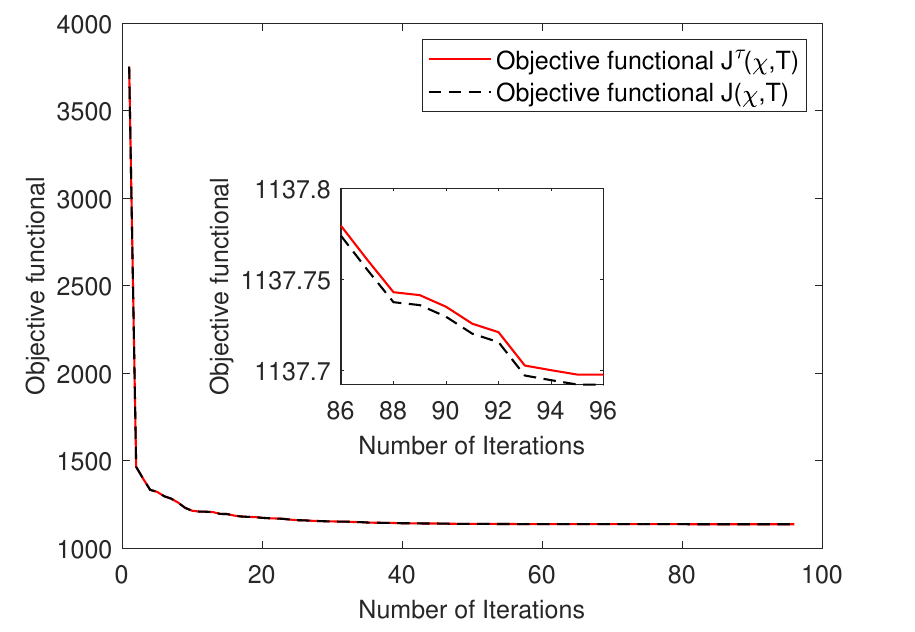}
	\includegraphics[width=0.28\linewidth,trim=2.5cm 9cm 3cm 10cm,clip]{fig/2D/model1/k/kE10.pdf}
	\caption{Comparing the impact of mesh size on the approximate solutions and objective functional values with $\kappa_1=10,\ \kappa_2=1$, $\gamma = 15$, $q_1=1,\ q_2=100$, volume fraction $\beta = 0.2$, $\tau = 1\times 10^{-4}$. Left to right: approximate solutions and objective functional values on the grids $200\times 200$, $400\times 400$, $600\times 600$. See Section~\ref{sec:dependonpara}.}
	\label{mesh}
\end{figure}

{\bf Dependency on volume fraction:} For the impact on volume fraction $\beta$, as shown in Figure~\ref{volume}, we see that the approximate solution becomes ``fat'' as volume increases but with the similar profiles.

\begin{figure}[ht!]
	\centering
        \includegraphics[width=0.28\linewidth,trim=4cm 10cm 3cm 10cm,clip]{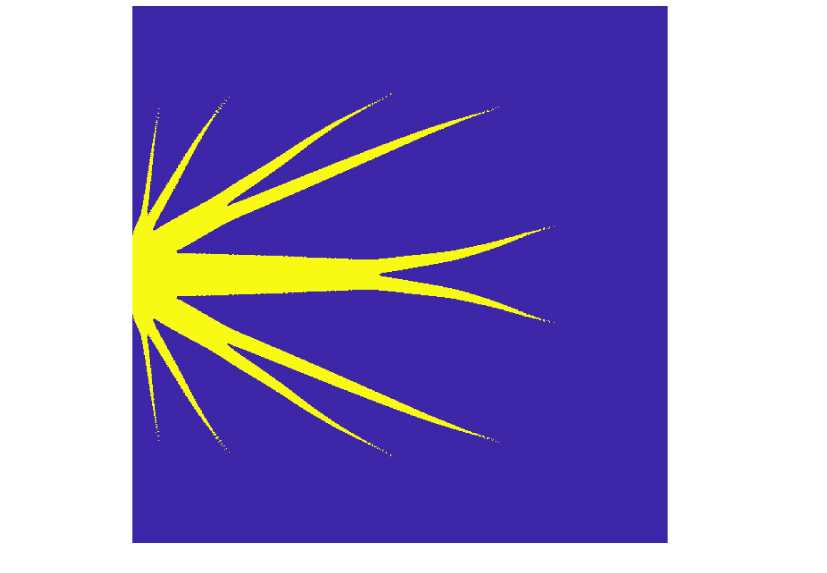}
	\includegraphics[width=0.28\linewidth,trim=4cm 10cm 3cm 10cm,clip]{fig/2D/model1/k/kX10.pdf}
         \includegraphics[width=0.28\linewidth,trim=4cm 10cm 3cm 10cm,clip]{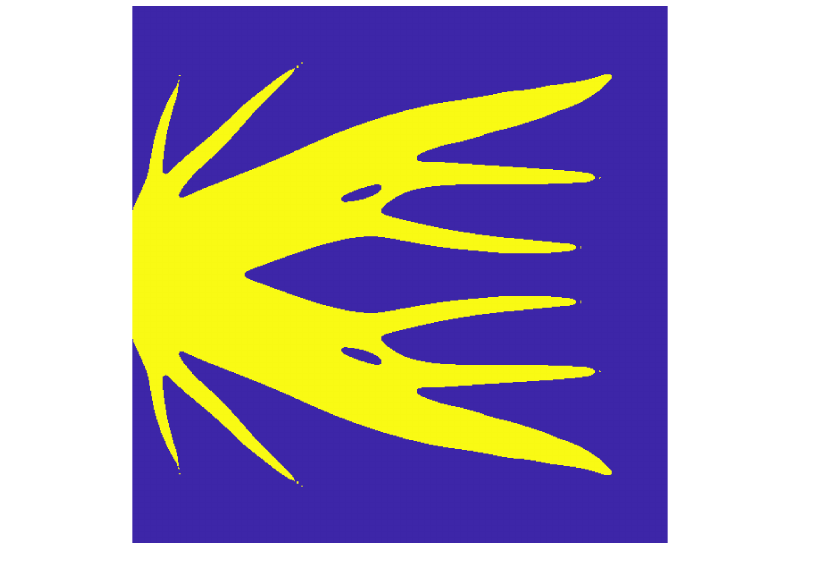}\\
        \includegraphics[width=0.28\linewidth,trim=2.5cm 9cm 3cm 10cm,clip]{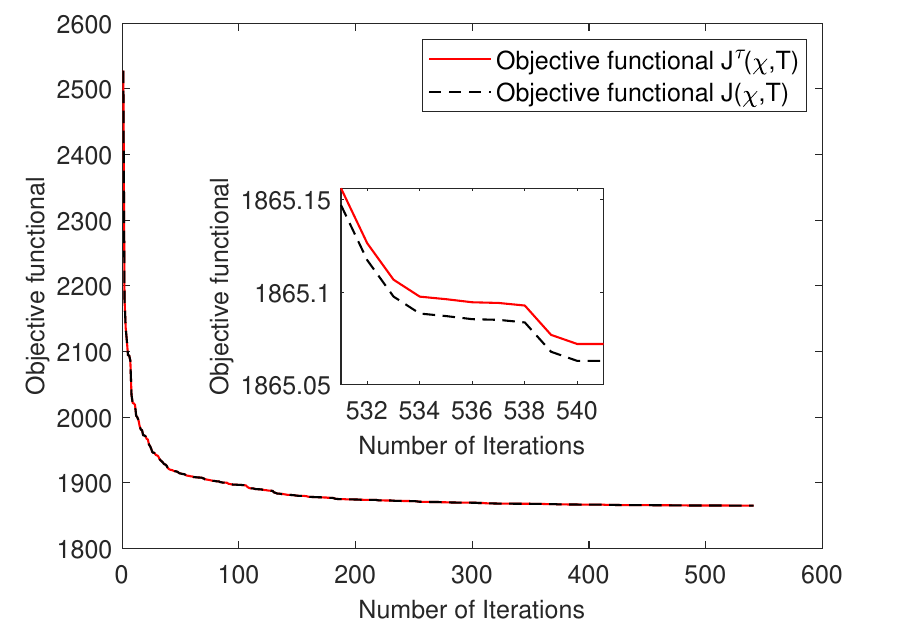}
        \includegraphics[width=0.28\linewidth,trim=2.5cm 9cm 3cm 10cm,clip]{fig/2D/model1/k/kE10.pdf}
	\includegraphics[width=0.28\linewidth,trim=2.5cm 9cm 3cm 10cm,clip]{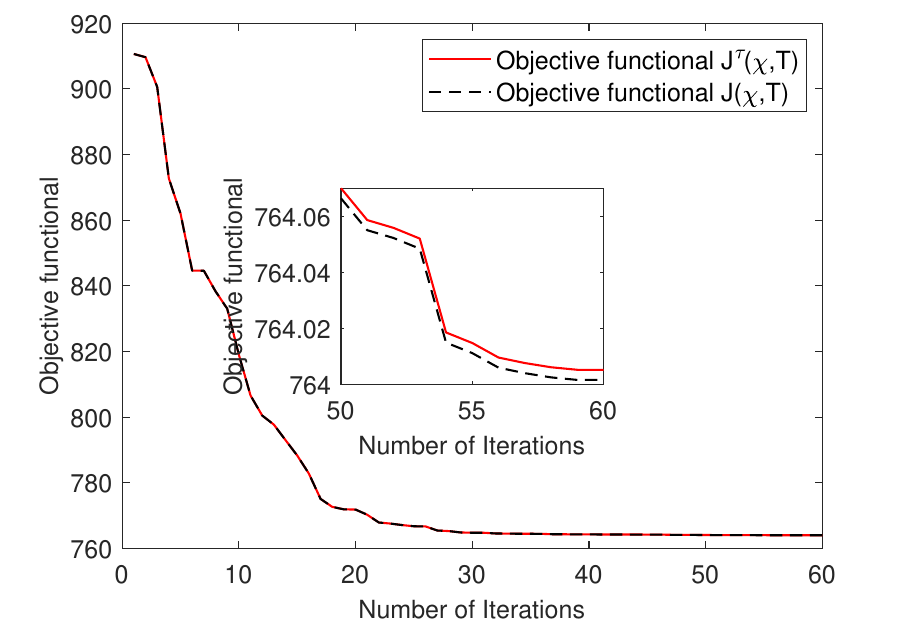}
	\caption{Comparing the impact of volume fraction on the optimal distribution of $\chi$ and energy value with $\kappa_1=10,\ \kappa_2=1$, $\gamma = 15$, $q_1=1,\ q_2=100$, $\tau$ on a $600\times 600$ grid. Left to right: Approximate solutions and objective functional values for volume fraction $\beta = 0.1,\ 0.2,\ 0.3$. See Section~\ref{sec:dependonpara}.}
	\label{volume}
\end{figure}

{\bf Dependency on $\gamma$:} We then consider different values of $\gamma$ with other values fixed. In Figure~\ref{fig:gamma}, we observe that as $\gamma$ decreases, the branches of the solution become finer with more details.  When the ratio of $\kappa_1/\kappa_2$ increases, larger values of $\gamma$ can regularize the solution. Figure~\ref{fig:k40} shows the optimal distribution of $\chi$ and the objective functional curve for the case of $\kappa_1=40,\ \kappa_2=1, \gamma=20$ and a random initial distribution of $\chi$. These are consistent with that larger $\gamma$ penalizes the length of the boundary. 

\begin{figure}[ht!]
    \centering
    \includegraphics[width=0.28\linewidth,trim=4cm 10cm 3cm 10cm,clip]{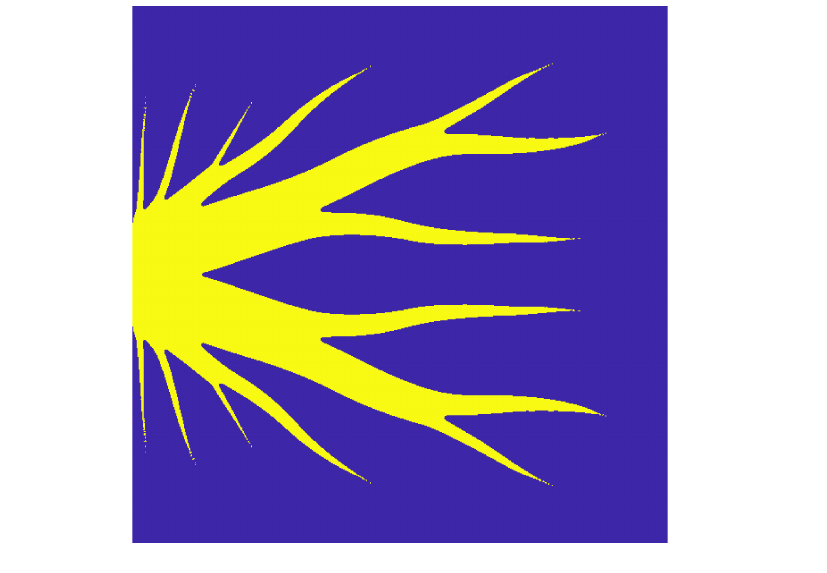}\includegraphics[width=0.28\linewidth,trim=4cm 10cm 3cm 10cm,clip]{fig/2D/model1/k/kX10.pdf}\includegraphics[width=0.28\linewidth,trim=4cm 10cm 3cm 10cm,clip]{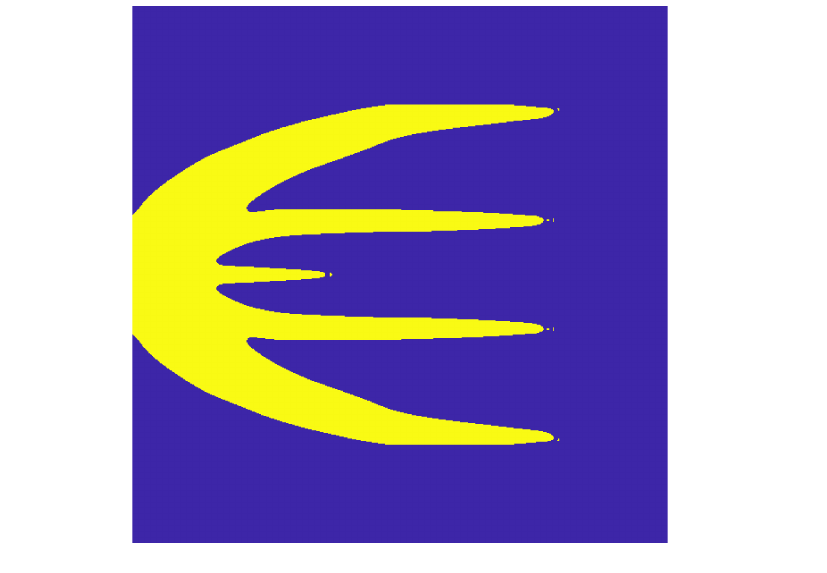}\\
    \includegraphics[width=0.28\linewidth,trim=2.5cm 9cm 3cm 10cm,clip]{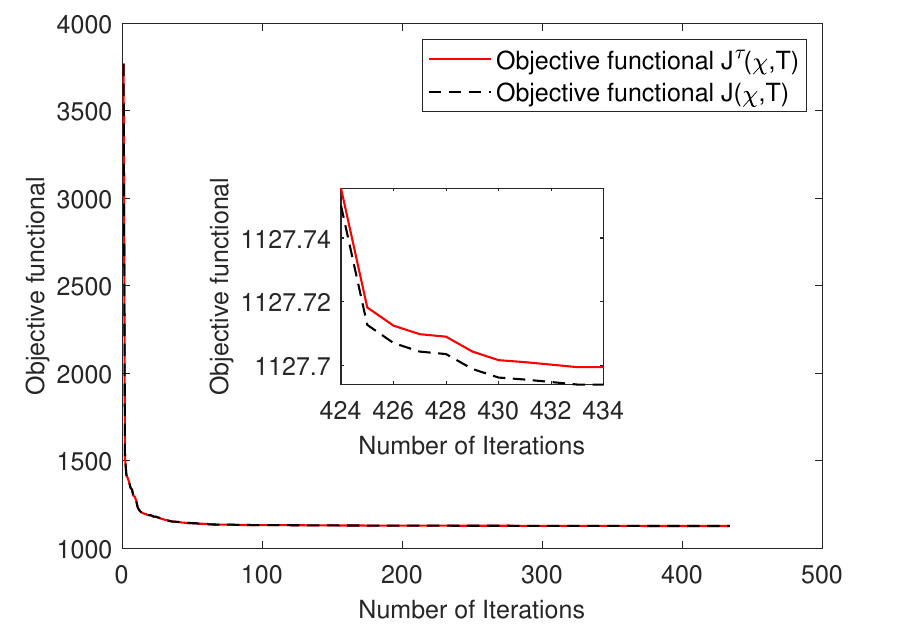}\includegraphics[width=0.28\linewidth,trim=2.5cm 9cm 3cm 10cm,clip]{fig/2D/model1/k/kE10.pdf}\includegraphics[width=0.28\linewidth,trim=2.5cm 9cm 3cm 10cm,clip]{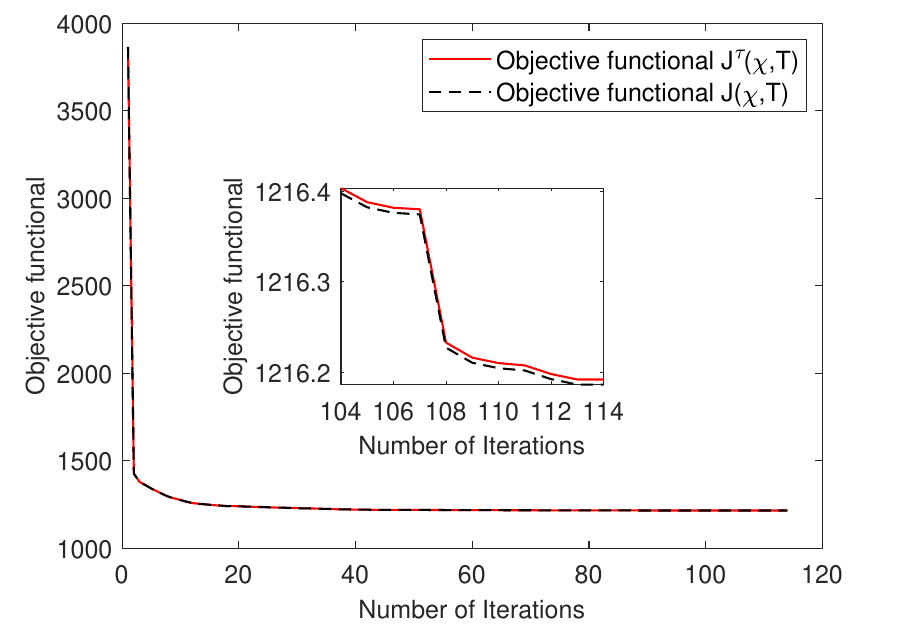}
    \caption{Comparing the impact of $\gamma$ on the optimal distribution of $\chi$ and energy value with $\kappa_1=10,\ \kappa_2=1$, $\tau = 1\times10^{-4}$, $q_1=1,\ q_2=100$ and volume fraction $ \beta = 0.2$ on a $600\times 600$ grid. Left to right: Approximate solutions and objective functional values for $\gamma = 12,\ 15,\ 50$. See Section~\ref{sec:dependonpara}.}
    \label{fig:gamma}
\end{figure}

\begin{figure}[ht!]
    \centering
    \includegraphics[width=0.28\linewidth,trim=6cm 10cm 3cm 10cm,clip]{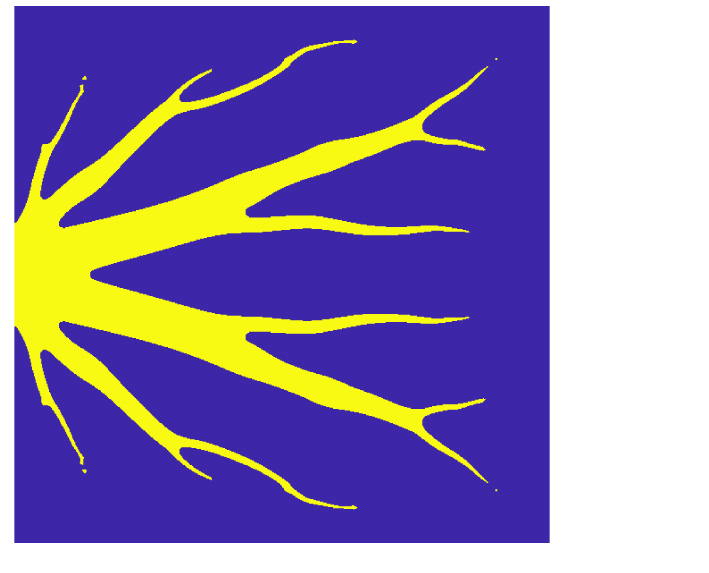}
    \includegraphics[width=0.28\linewidth,trim=2.5cm 8.5cm 3cm 10cm,clip]{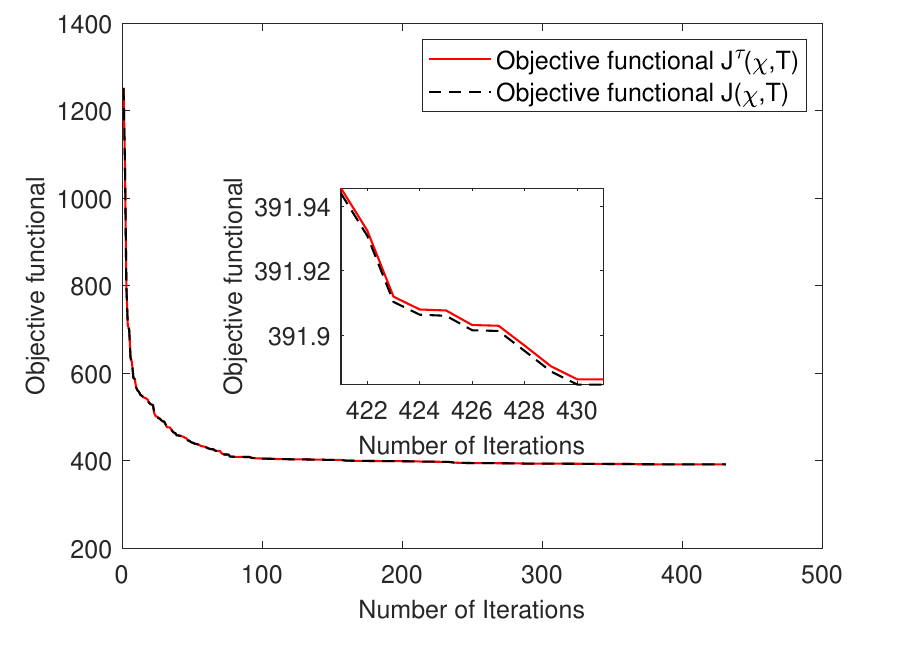}
    \caption{The parameters are set as $q_1=1,\ q_2=100$, $\kappa_1 = 40$, $\kappa_2=1$, $\gamma = 20$, volume fraction $\beta = 0.2$, mesh size = $600\times 600$, and $\tau = 1\times 10^{-4}$. Left: The optimal distribution of $\chi$. Right: Curve of the objective functional values. See Section~\ref{sec:dependonpara}.}
    \label{fig:k40}
\end{figure}

{\bf Dependency on $\tau$:}  The influence of $\tau$ on the optimal distribution of $\chi$ is also shown in Figure~\ref{fig:tau}. As $\tau$ increases, the distribution of $\chi$ has fewer fine branches.

\begin{figure}[ht!]
    \centering
    \includegraphics[width=0.28\linewidth,trim=4cm 10cm 3cm 10cm,clip]{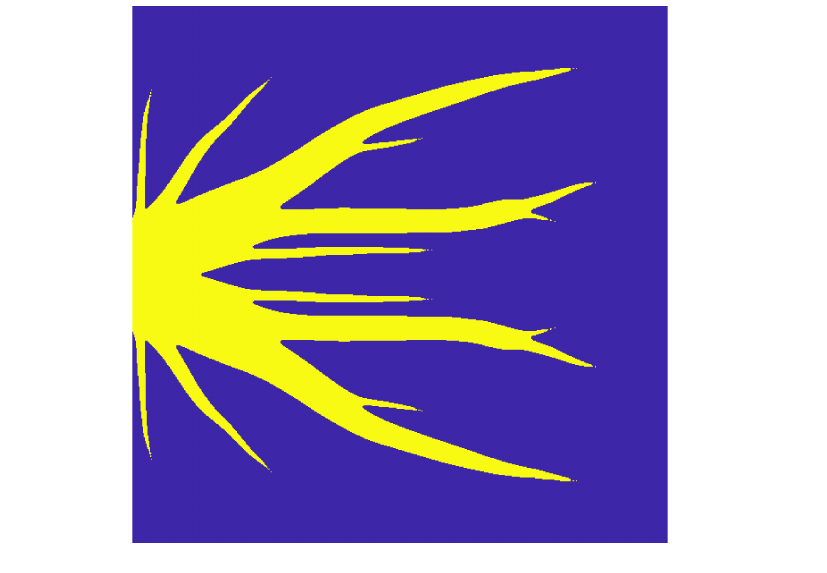}
    \includegraphics[width=0.28\linewidth,trim=4cm 10cm 3cm 10cm,clip]{fig/2D/model1/k/kX10.pdf}
    \includegraphics[width=0.28\linewidth,trim=4cm 10cm 3cm 10cm,clip]{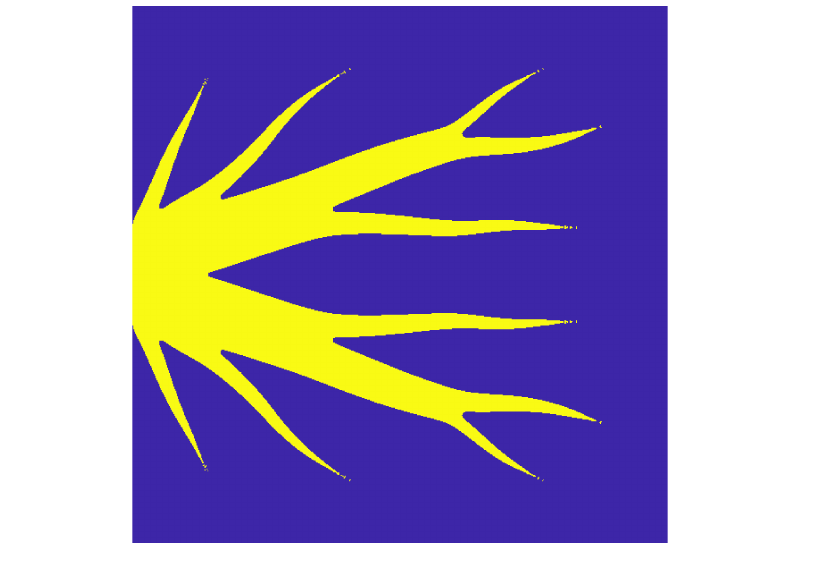}\\
    \includegraphics[width=0.28\linewidth,trim=2.5cm 9cm 3cm 10cm,clip]{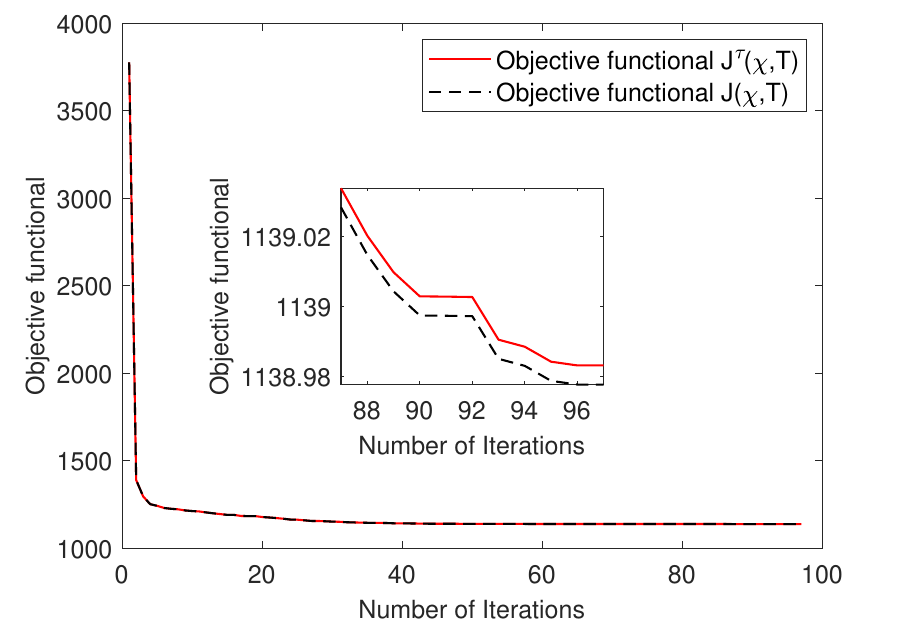}
    \includegraphics[width=0.28\linewidth,trim=2.5cm 9cm 3cm 10cm,clip]{fig/2D/model1/k/kE10.pdf}
    \includegraphics[width=0.28\linewidth,trim=2.5cm 9cm 3cm 10cm,clip]{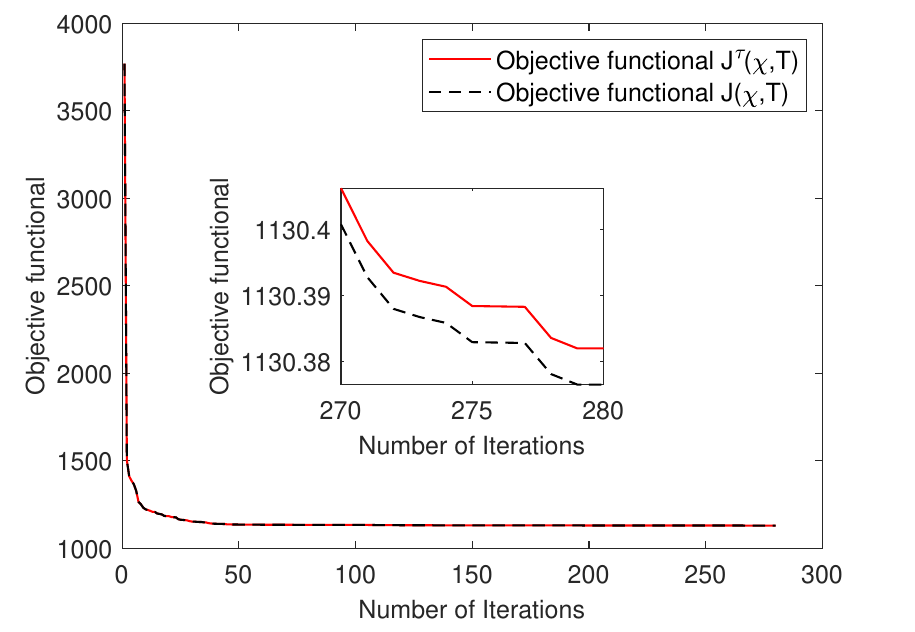}
    \caption{Comparing the impact of $\tau$ on the optimal distribution of $\chi$ and energy value with $\kappa_1=10,\ \kappa_2=1$, $\gamma = 15$, $q_1=1,\ q_2=100$ and volume fraction $ \beta = 0.2$ on a $600\times 600$ grid. Left to right: Approximate solutions and objective functional values for $\tau = 0.35\times10^{-4}$, $1\times10^{-4}$, $1.5\times10^{-4}$. See Section~\ref{sec:dependonpara}.}
    \label{fig:tau}
\end{figure}

{\bf Dependency on initial guesses:}
Figure~\ref{random} shows the impact of the initial distribution of $\chi$ on the solution. We choose three different random initial guesses of $\chi$ for the cases in Figure~\ref{random}. We can see from Figure~\ref{random} that the optimal distributions of $\chi$ indeed depends on different choices of the initial distributions of $\chi$, but the objective functional always decays similarly. We note that this example shows the key difficulty of topology optimization problems where the solution landscape is very complicated. One usually can only find a local stationary solution starting with a random initial guess.

\begin{figure}[ht!]
	\centering
        \includegraphics[width=0.28\linewidth,trim=4cm 10cm 3cm 10cm,clip]{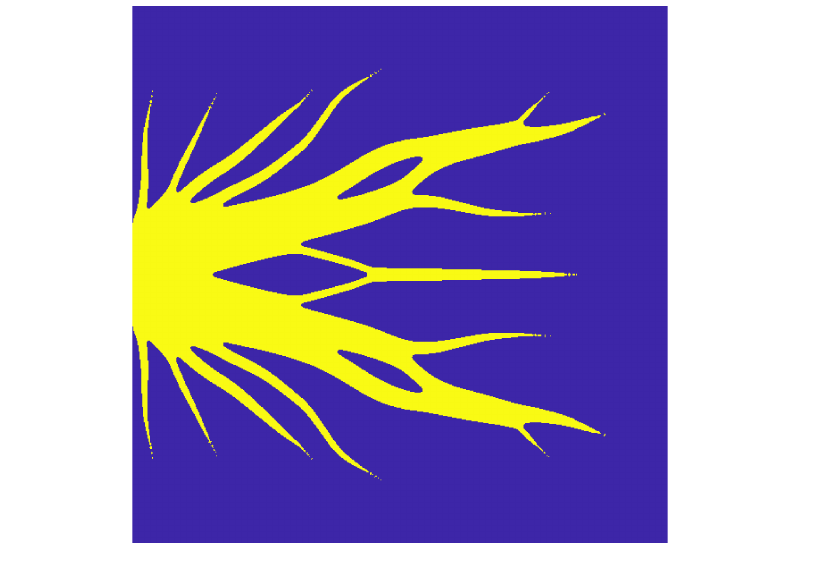}
	\includegraphics[width=0.28\linewidth,trim=4cm 10cm 3cm 10cm,clip]{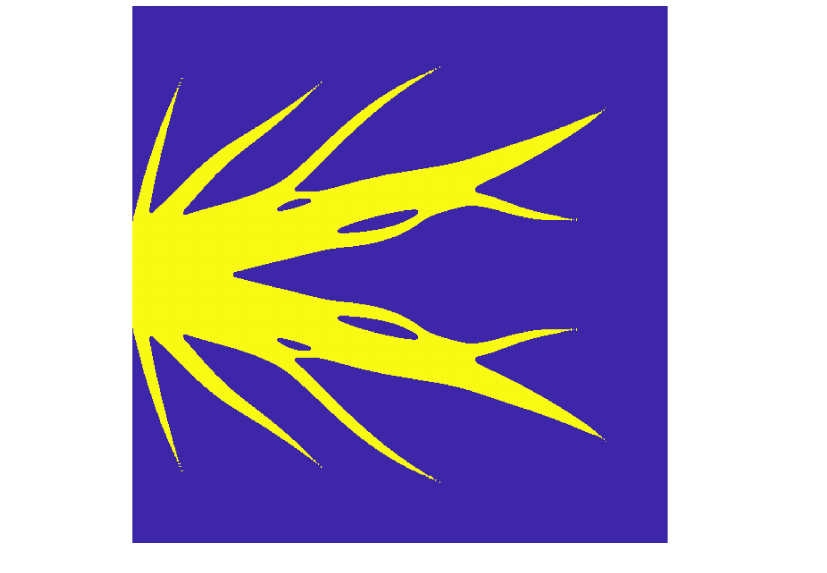}
        \includegraphics[width=0.28\linewidth,trim=4cm 10cm 3cm 10cm,clip]{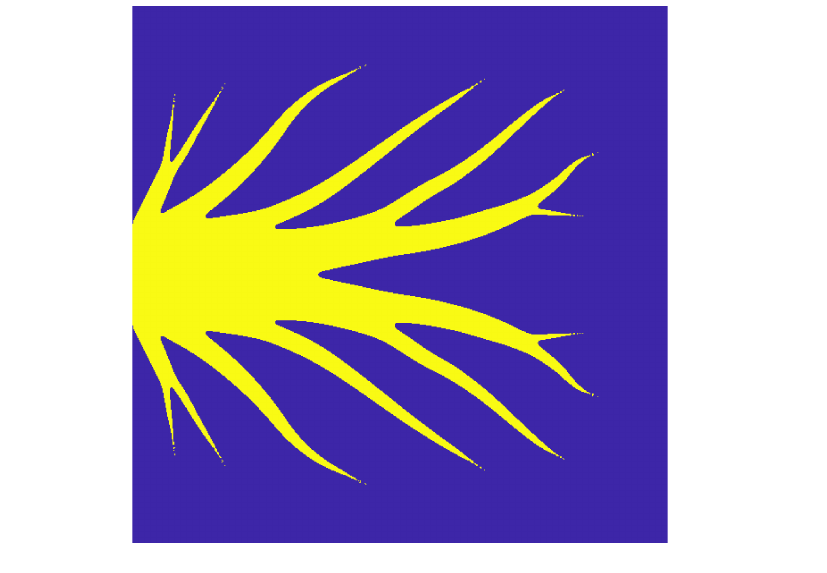}\\
        \includegraphics[width=0.28\linewidth,trim=2.5cm 9cm 3cm 10cm,clip]{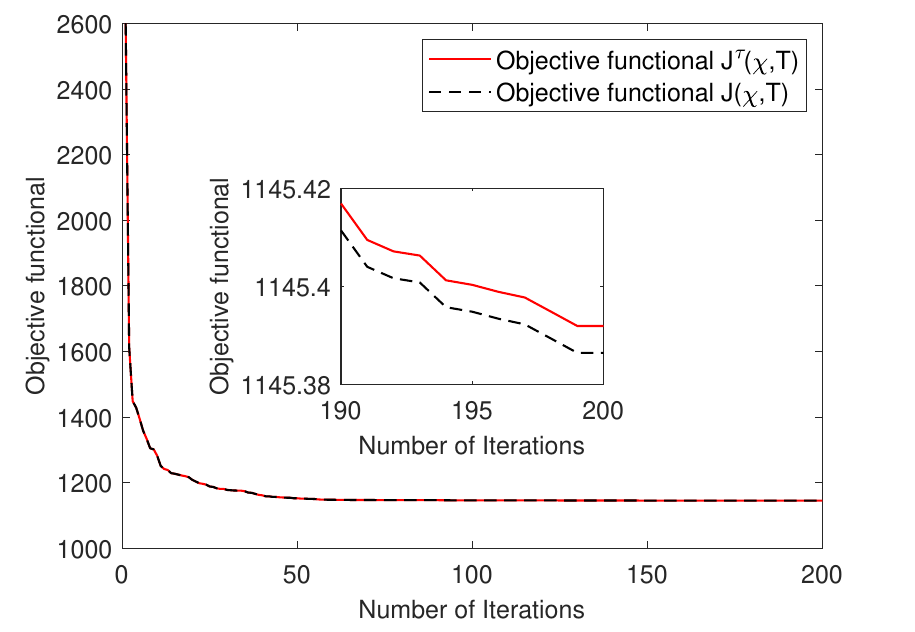}
        \includegraphics[width=0.28\linewidth,trim=2.5cm 9cm 3cm 10cm,clip]{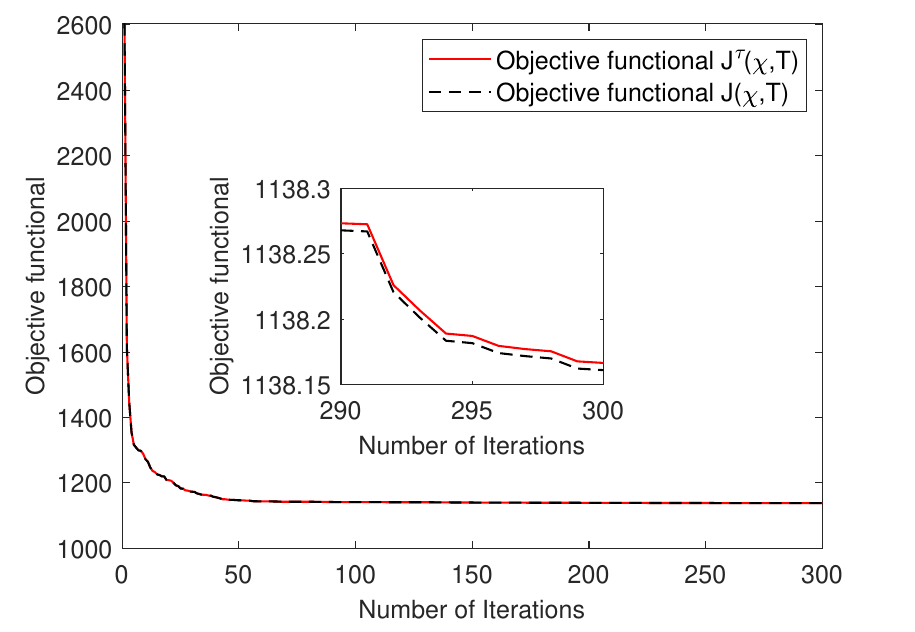}
	\includegraphics[width=0.28\linewidth,trim=2.5cm 9cm 3cm 10cm,clip]{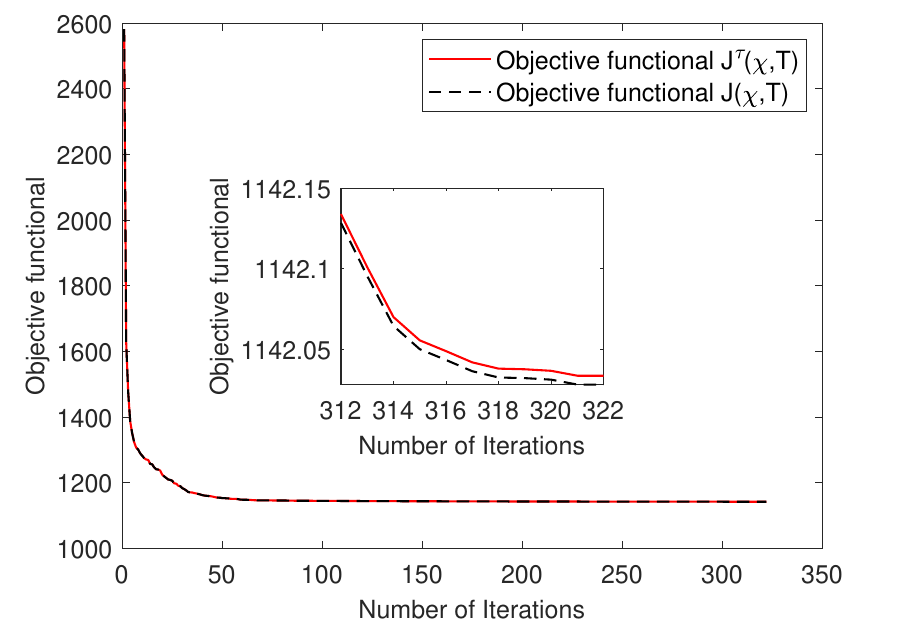}
	\caption{Comparing the impact of the initial distribution of $\chi$ on the optimal distribution of $\chi$ and energy value with $\kappa_1=10,\ \kappa_2=1$, $q_1=1,\ q_2=100$, $\gamma = 15$ and volume fraction $\beta = 0.2$ on a $600\times 600$ grid. Left to right: Approximate solutions and objective functional values for three random initial distribution of $\chi$. See Section~\ref{sec:dependonpara}.}
	\label{random}
\end{figure}

\subsection{Area-to-sides problem}\label{sec:area2sides}
 In this example, we consider the problem where the Dirichlet boundary condition is only imposed as shown in Figure~\ref{fig:Example2} with $l=1$. Due to the symmetry of the problem, we only optimize over the upper left quadrant of the domain. We also set $\xi=1\times 10^{-5}$ in this example. Similar to the above area-to-point example, we test the impact of $\frac{\kappa_1}{\kappa_2}$, $\frac{q_1}{q_2}$ on the optimal distribution of $\chi$ in Figure~\ref{fig:example2-k} and Figure~\ref{fig:example2-q}. In Figure~\ref{fig:example2-k}, the optimal distributions of $\chi$ exhibits finer branches with a higher conductivity ratio. Additionally, the original objective functional is smaller than the modified objective functional and both objective functional always decay in all examples. We also note that the value of objective functional which 
 stably converges during iteration decreases as the ratio of $\kappa_1/\kappa_2$ increases. Figure~\ref{fig:example2-q} shows that the heat-generation rate ratio $\frac{q_1}{q_2}$ influences the optimal distribution of $\chi$ and the values of objective functional. 
 
 \begin{figure}[ht!]
     \centering
     \includegraphics[width=0.4\linewidth,trim=0cm 1.5cm 0cm 10cm,clip]{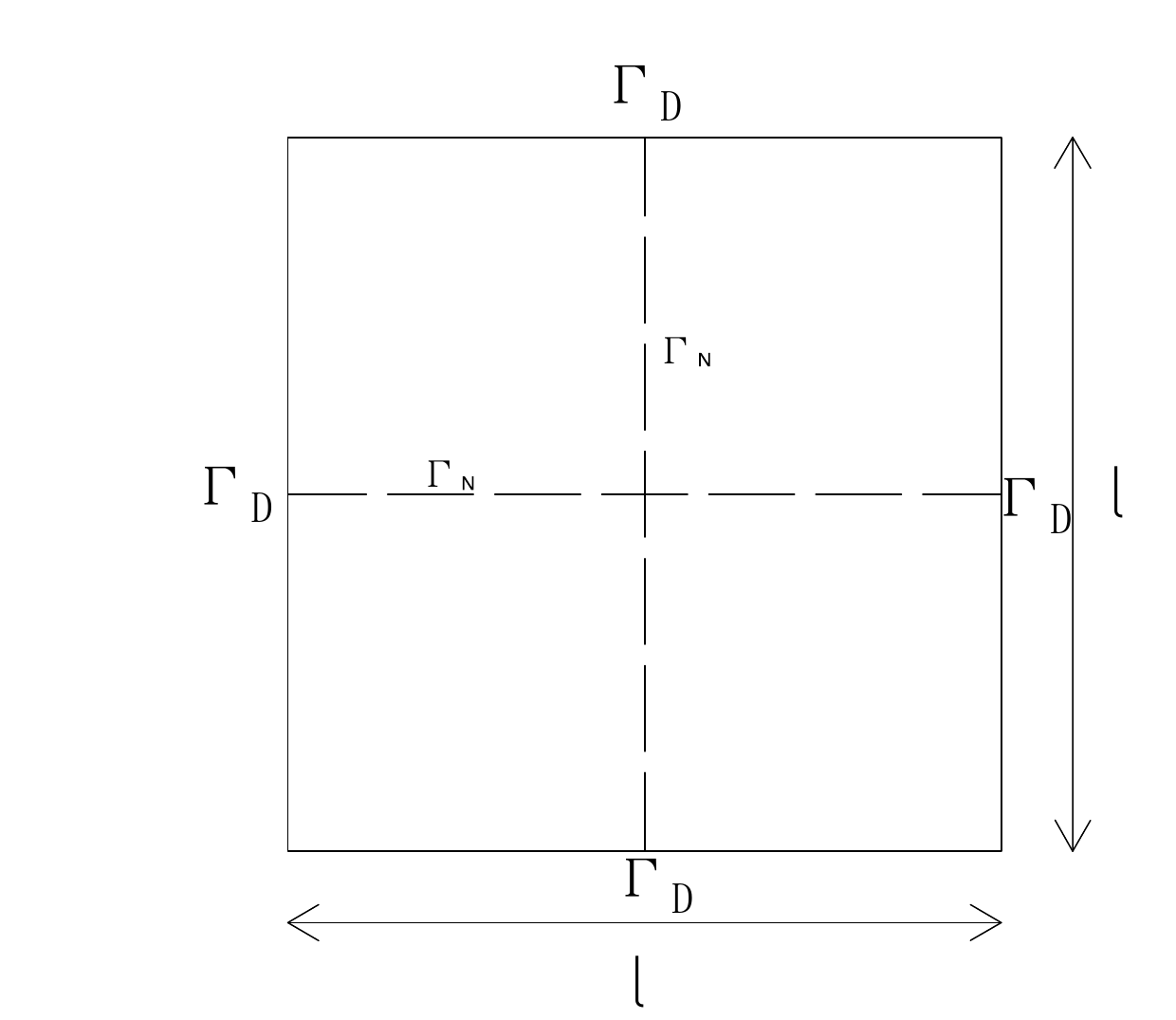}
     \includegraphics[width=0.4\linewidth,trim=5cm 10cm 3cm 10cm,clip]{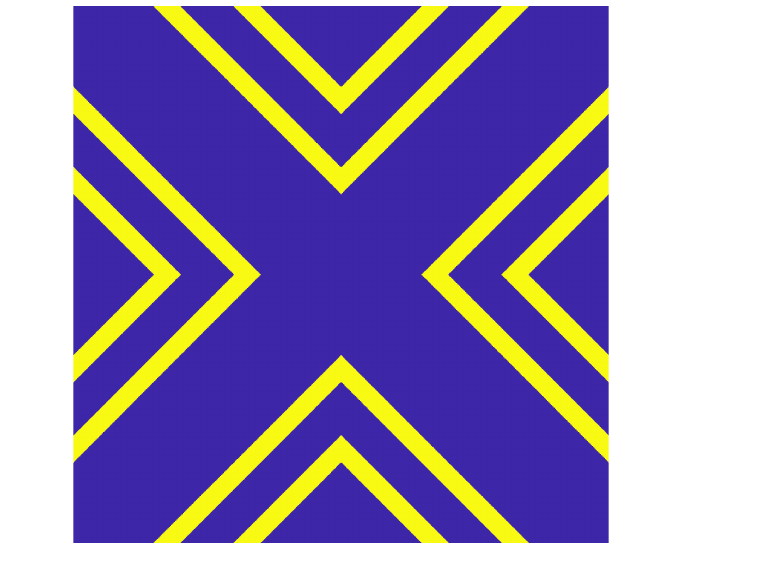}
     \caption{Left: The volume-to-sides problem. The boundary condition is Dirichlet boundary. Right: The initial distribution of $\chi$. See Section~\ref{sec:area2sides}.}
     \label{fig:Example2}
 \end{figure}

\begin{figure}[ht!]
    \centering
    \includegraphics[width=0.28\linewidth,trim=4cm 10cm 3cm 10cm,clip]{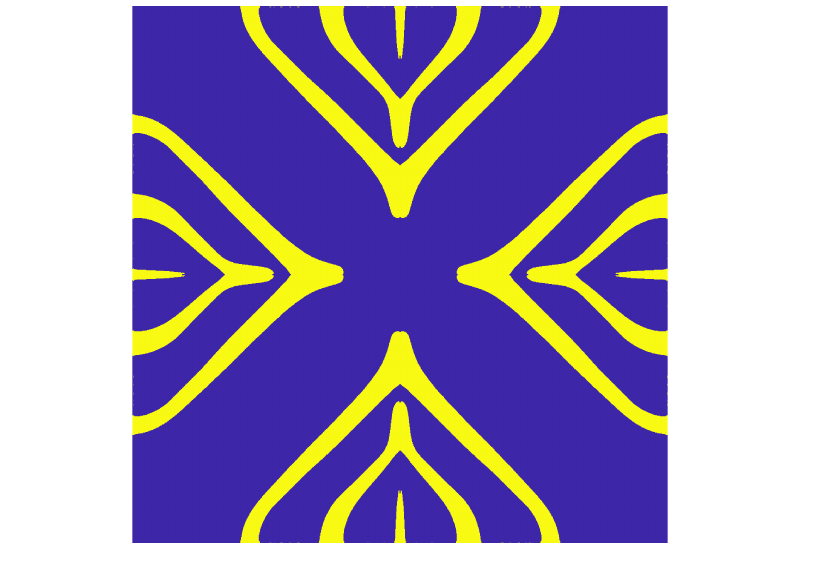}
    \includegraphics[width=0.28\linewidth,trim=4cm 10cm 3cm 10cm,clip]{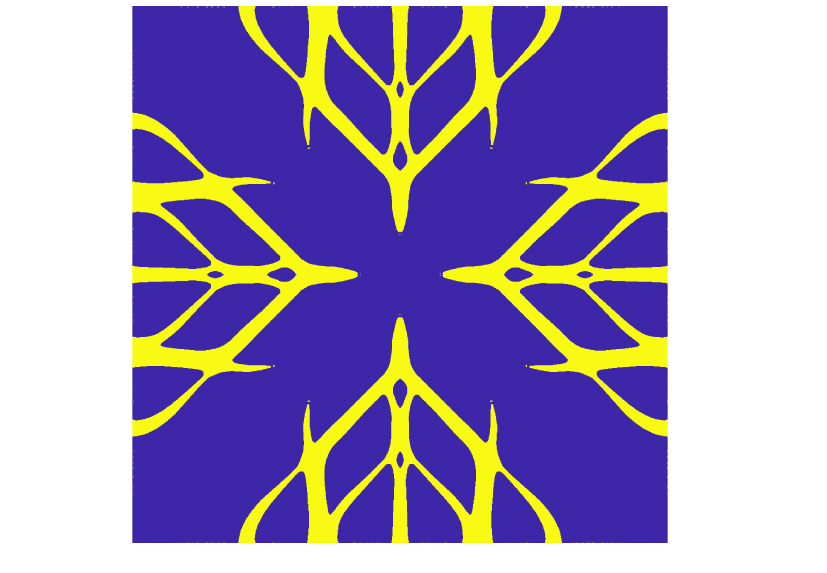}
    \includegraphics[width=0.28\linewidth,trim=4cm 10cm 3cm 10cm,clip]{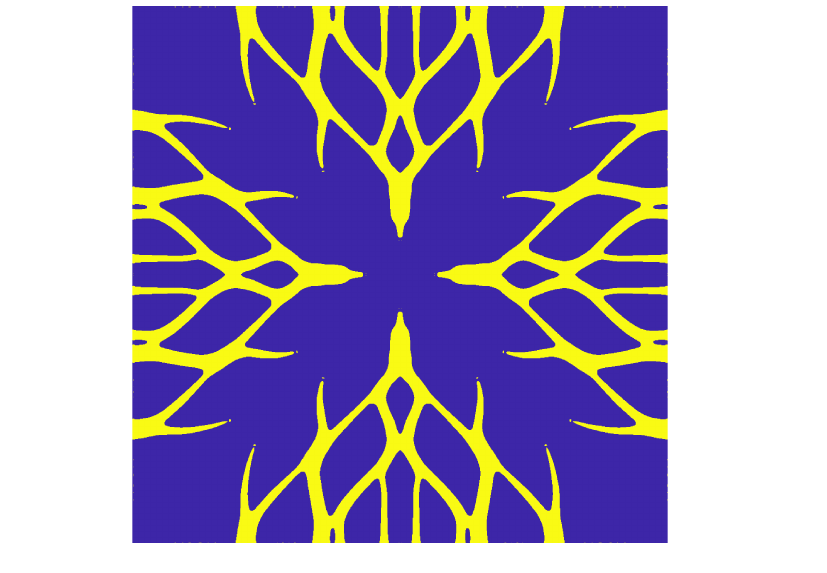}\\
    \includegraphics[width=0.28\linewidth,trim=2.5cm 9cm 3cm 10cm,clip]{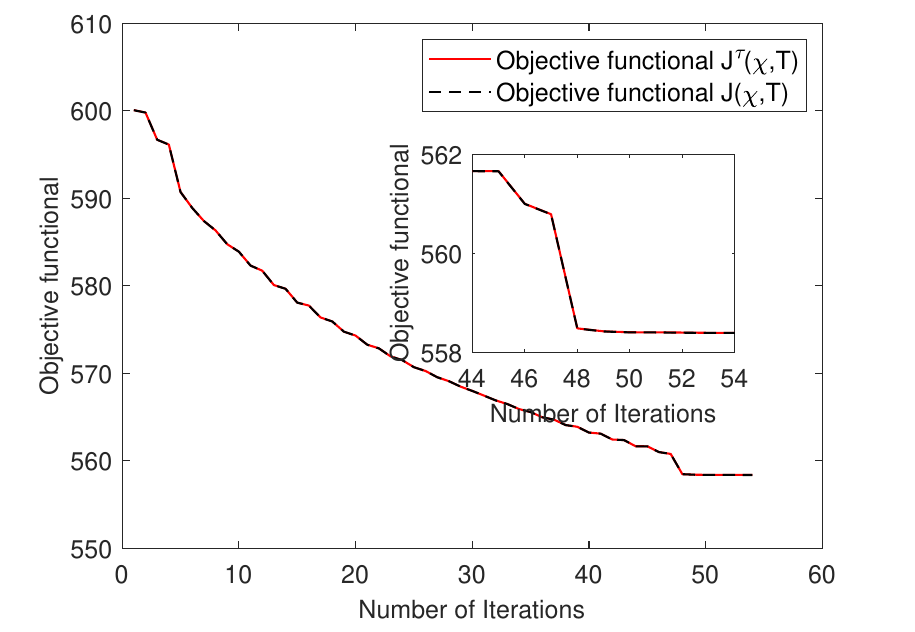}
    \includegraphics[width=0.28\linewidth,trim=2.5cm 9cm 3cm 10cm,clip]{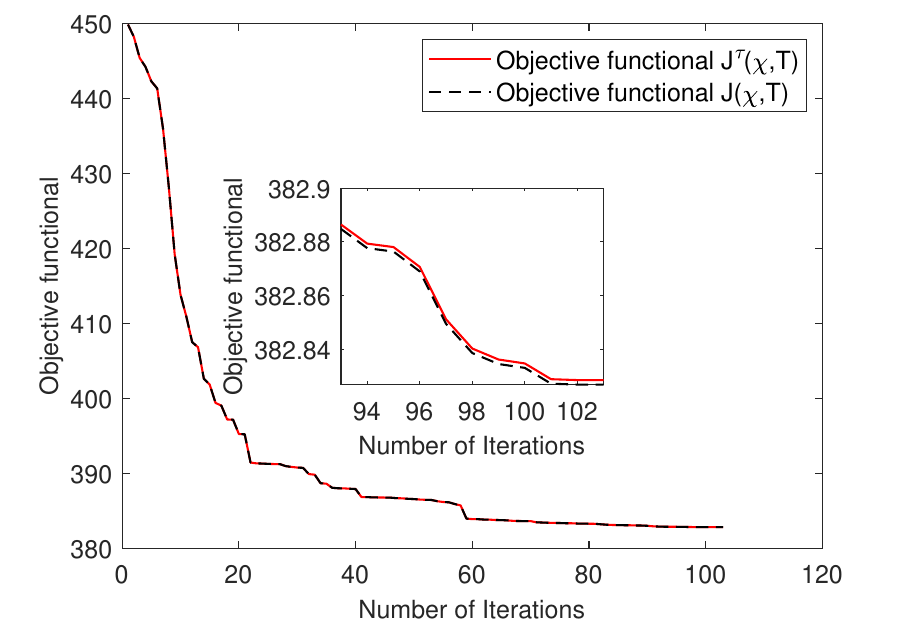}
    \includegraphics[width=0.28\linewidth,trim=2.5cm 9cm 3cm 10cm,clip]{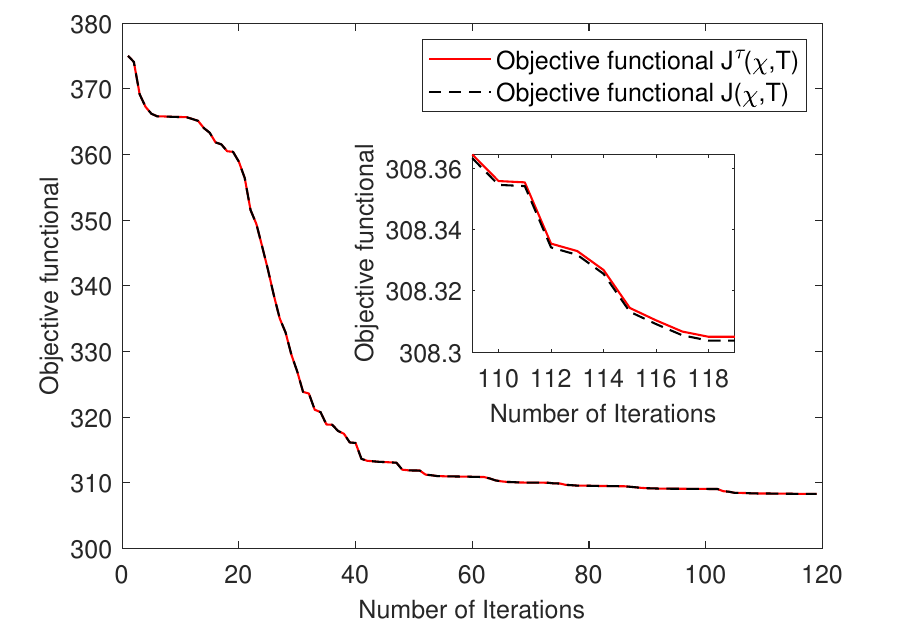}
    \caption{Comparing the impact of $\frac{\kappa_1}{\kappa_2}$ on the optimal distribution of $\chi$ and energy on a $1200\times 1200$ grid with $q_1=1,\ q_2=100$, $\tau = 1\times10^{-4}$, $\gamma = 20$, $\xi=1\times 10^{-5}$. Left to right: Approximate solutions and objective functional values for $\kappa_1=5,\ 10,\ 15$, and $\kappa_2=1$. See Section~\ref{sec:area2sides}.}
    \label{fig:example2-k}
\end{figure}

\begin{figure}[ht!]
    \centering
    \includegraphics[width=0.28\linewidth,trim=4cm 10cm 3cm 10cm,clip]{fig/2D/model2/q/Xq1.pdf}
    \includegraphics[width=0.28\linewidth,trim=4cm 10cm 3cm 10cm,clip]{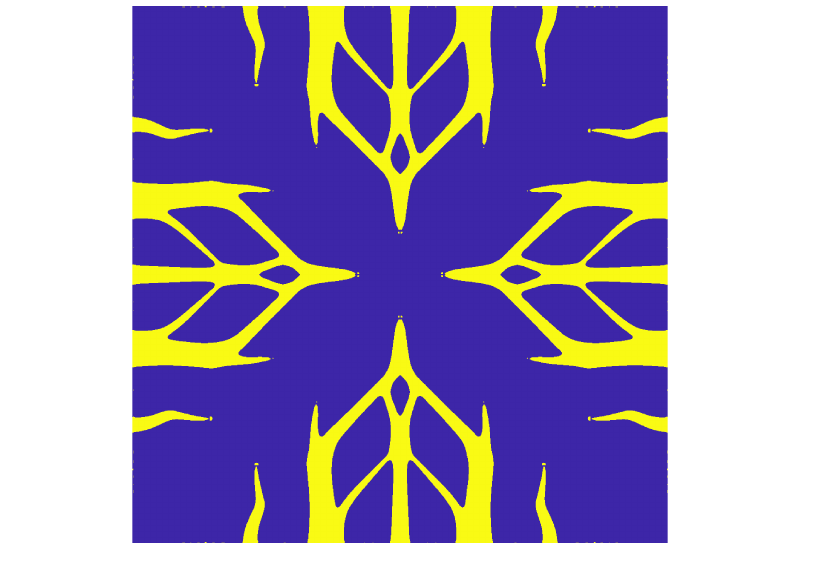}
    \includegraphics[width=0.28\linewidth,trim=4cm 10cm 3cm 10cm,clip]{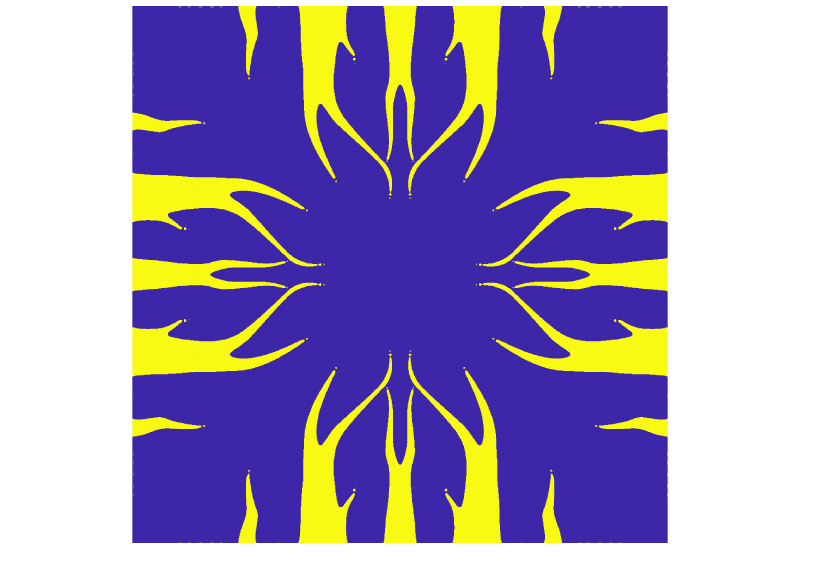}\\
    \includegraphics[width=0.28\linewidth,trim=2.5cm 9cm 3cm 10cm,clip]{fig/2D/model2/q/Eq1.pdf}
    \includegraphics[width=0.28\linewidth,trim=2.5cm 9cm 3cm 10cm,clip]{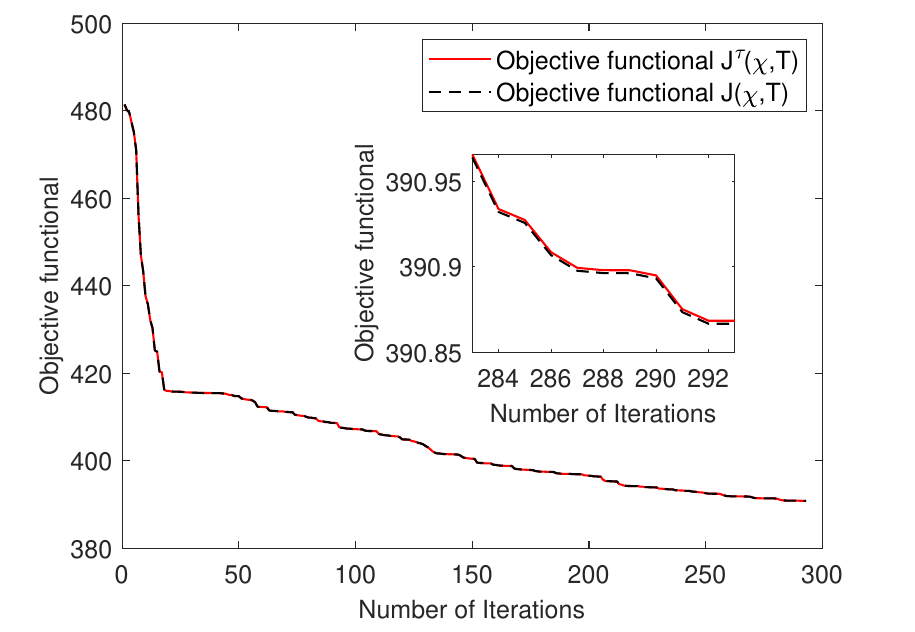}
    \includegraphics[width=0.28\linewidth,trim=2.5cm 9cm 3cm 10cm,clip]{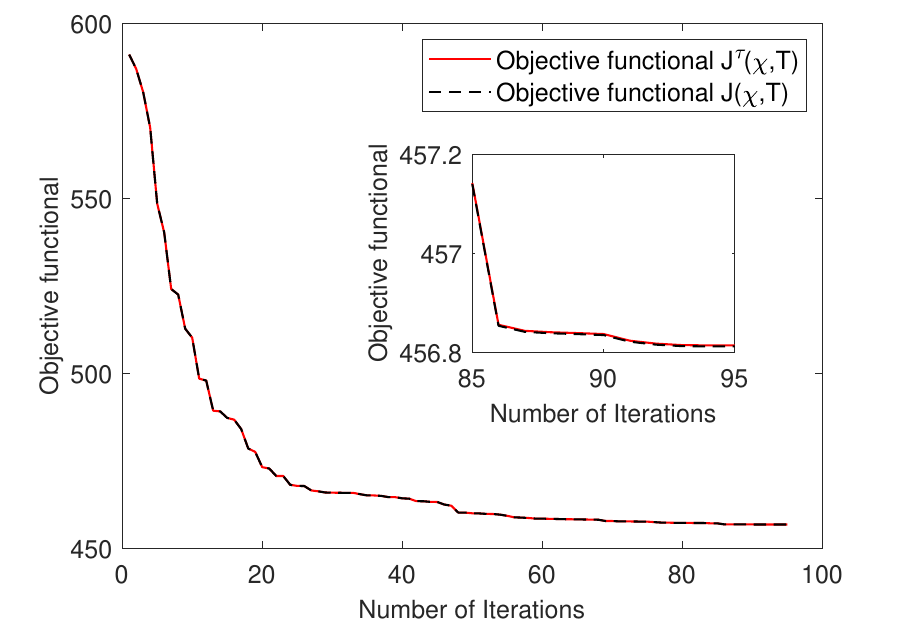}
    \caption{Comparing the impact of $\frac{q_1}{q_2}$ on the optimal distribution of $\chi$ and energy on a $1200\times 1200$ grid with $\kappa_1=10,\ \kappa_2=1$, $\tau = 1\times10^{-4}$, $\gamma = 20$, $\xi=1\times 10^{-5}$. Left to right: Approximate solutions and objective functional values for $q_1=1,\ 20,\ 80$, $q_2=100$ from left to right. See Section~\ref{sec:area2sides}.}
    \label{fig:example2-q}
\end{figure}

 Additionally, we set $\kappa_1=10,\ \kappa_2=1$, $q_1=1,\ q_2=100$, $\tau = 1\times10^{-4}$, $\gamma = 20$ to test the impact of volume fraction $\beta$ and the mesh size on the optimal distribution of $\chi$ in Figure~\ref{fig:example2-v} and Figure~\ref{fig:example2-mesh} respectively. In Figure~~\ref{fig:example2-v}, the results show that the branches become to be concentrating on the main branches and the value of objective functional which stably converges during iteration is gradually decreasing as the volume fraction $\beta$ increases on a $1200\times 1200$ grid. In order to see the impact of mesh sizes, we can also see from Figure~\ref{fig:example2-mesh} that the prediction-correction-based ICTM demonstrates the stability of the Algorithm~\ref{a:prediction-correction} on different grids and the finer distribution of $\chi$ is observed on the finer grids.

\begin{figure}[ht!]
    \centering
        \includegraphics[width=0.28\linewidth,trim=4cm 10cm 3cm 10cm,clip]{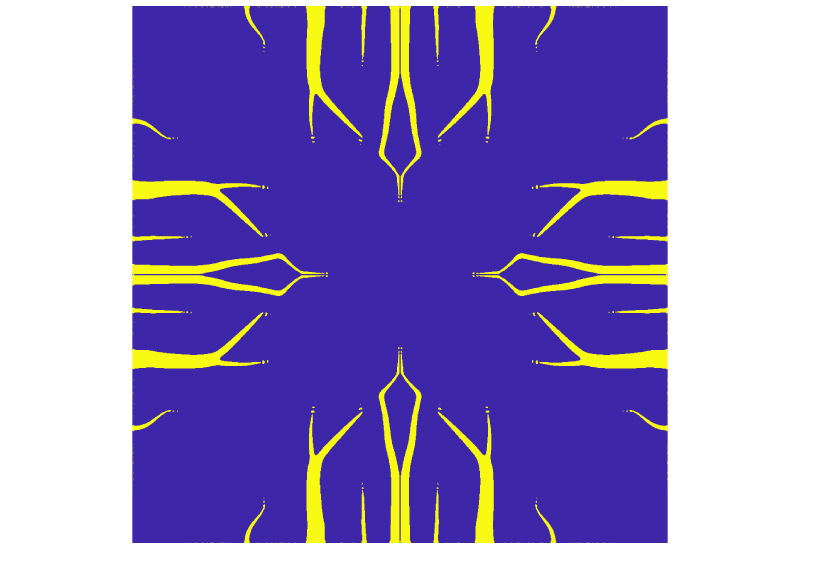}
	\includegraphics[width=0.28\linewidth,trim=4cm 10cm 3cm 10cm,clip]{fig/2D/model2/q/Xq1.pdf}
        \includegraphics[width=0.28\linewidth,trim=4cm 10cm 3cm 10cm,clip]{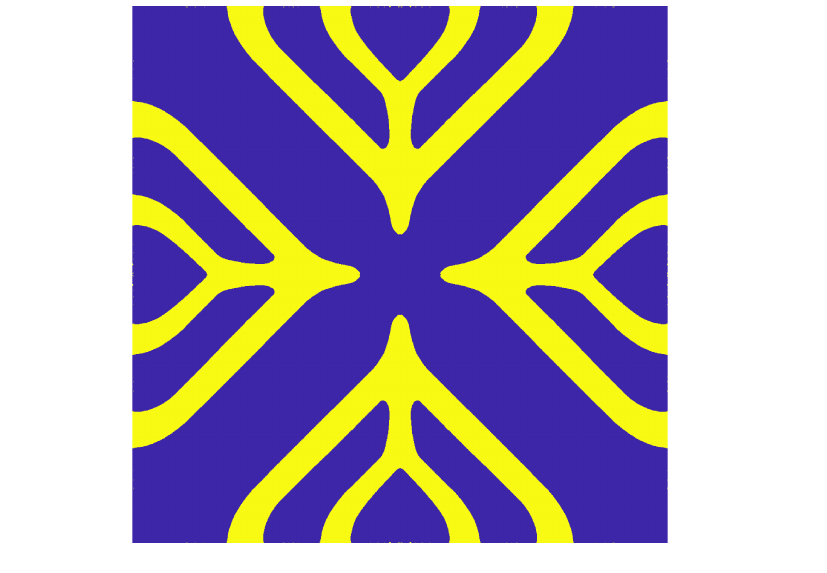}\\
        \includegraphics[width=0.28\linewidth,trim=2.5cm 9cm 3cm 10cm,clip]{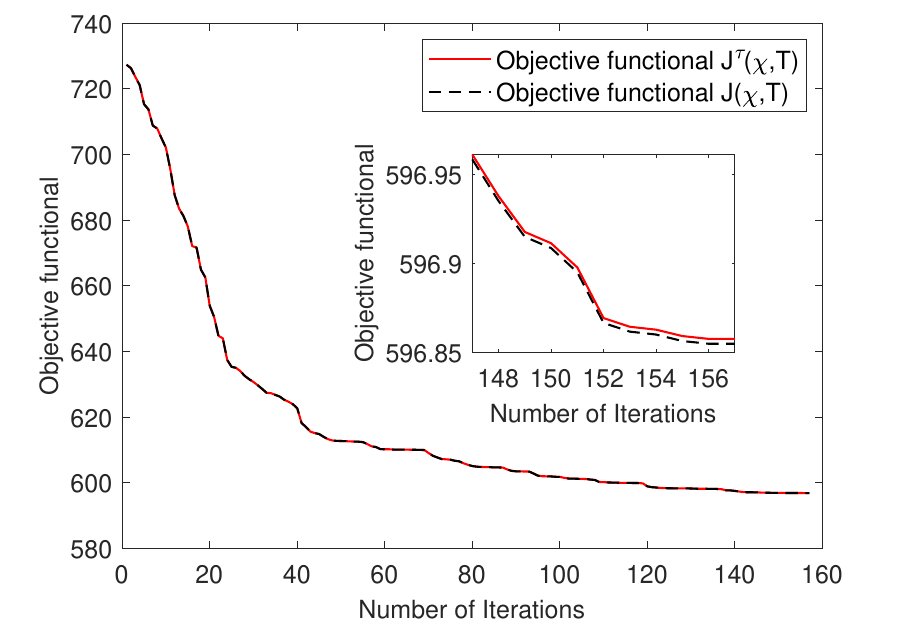}
        \includegraphics[width=0.28\linewidth,trim=2.5cm 9cm 3cm 10cm,clip]{fig/2D/model2/q/Eq1.pdf}
	\includegraphics[width=0.28\linewidth,trim=2.5cm 9cm 3cm 10cm,clip]{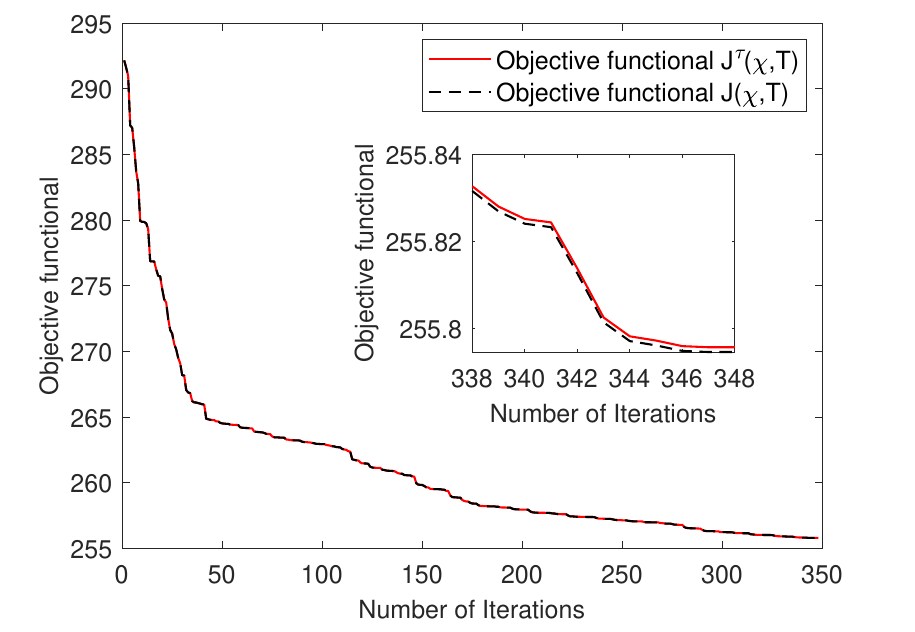}
	\caption{Comparing the influences of volume fraction on the optimal distribution of $\chi$ and energy on a $1200\times 1200$ grid with $\kappa_1=10,\ \kappa_2=1$, $q_1=1,\ q_2=100$, $\tau = 1\times10^{-4}$, $\gamma = 20$ and $\xi=1\times 10^{-5}$. Left to right: Approximate solutions and objective functional values for volume fraction $\beta = 0.1, \ 0.2,\  0.3$. See Section~\ref{sec:area2sides}.}
    \label{fig:example2-v}
\end{figure}

\begin{figure}[ht!]
    \centering
    \includegraphics[width=0.28\linewidth,trim=4cm 10cm 3cm 10cm,clip]{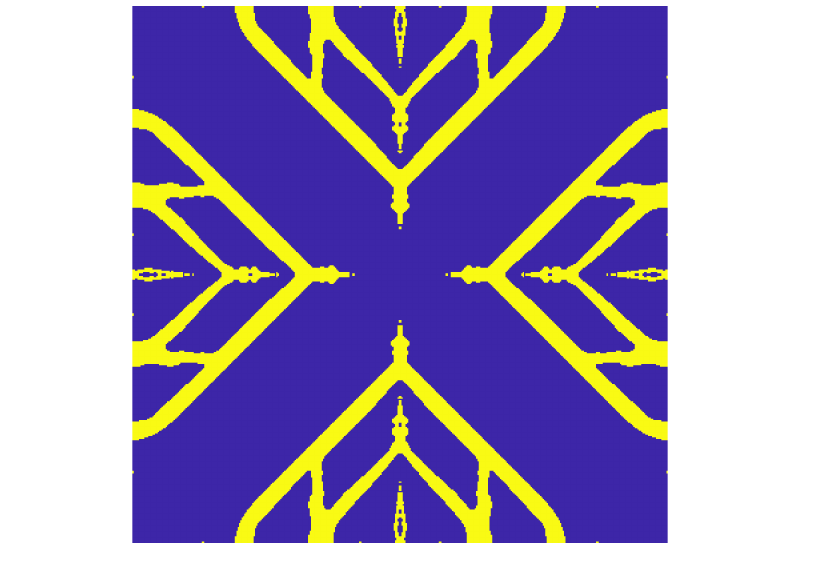}
    \includegraphics[width=0.28\linewidth,trim=4cm 10cm 3cm 10cm,clip]{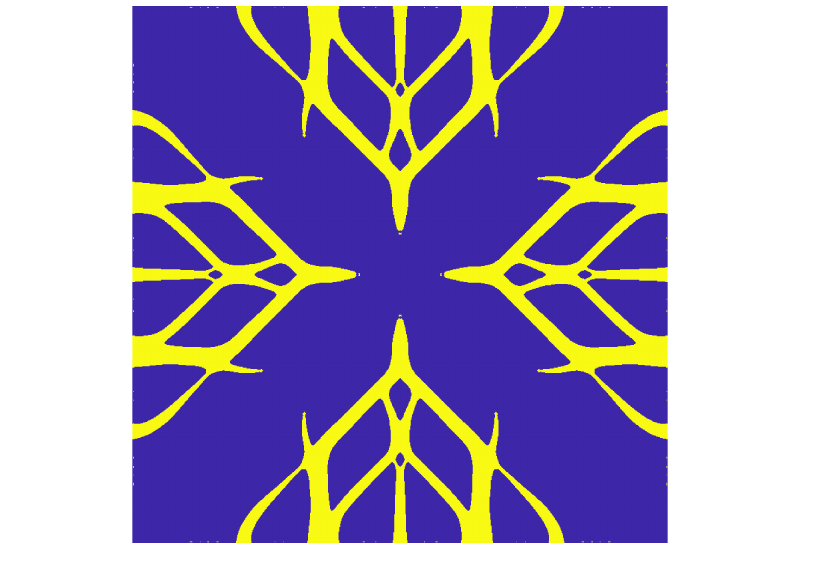}
    \includegraphics[width=0.28\linewidth,trim=4cm 10cm 3cm 10cm,clip]{fig/2D/model2/q/Xq1.pdf}\\
     \includegraphics[width=0.28\linewidth,trim=2.5cm 9cm 3cm 10cm,clip]{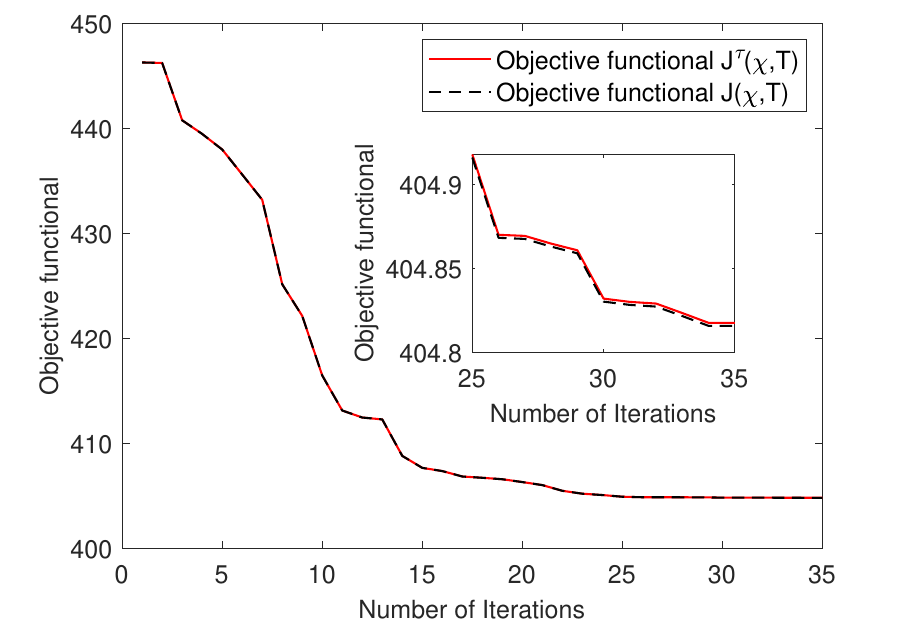}
    \includegraphics[width=0.28\linewidth,trim=2.5cm 9cm 3cm 10cm,clip]{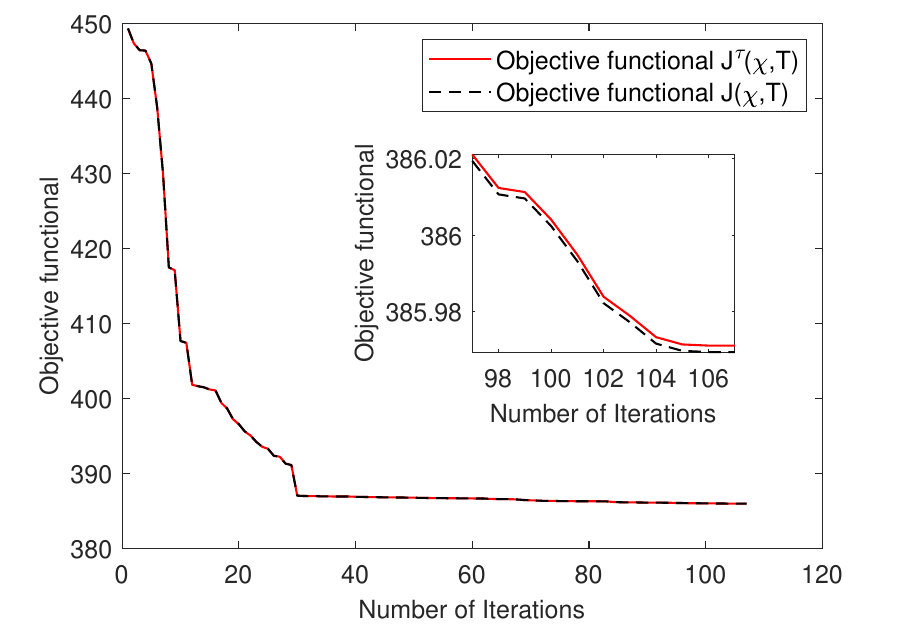}
    \includegraphics[width=0.28\linewidth,trim=2.5cm 9cm 3cm 10cm,clip]{fig/2D/model2/q/Eq1.pdf}
    \caption{Comparing the influences of mesh size on the optimal distribution of $\chi$ and energy. Left to right: Approximate solutions and objective functional values on grids $400\times 400$, $800\times 800$, $1200\times 1200$. See Section~\ref{sec:area2sides}.}
    \label{fig:example2-mesh}
\end{figure}

\subsection{Three dimensional results}\label{sec:3d}
We now present an example in three dimensions as shown in \cite{burger2013three}, focusing on a cubic domain for the volume-to-surface problem depicted in Figure~\ref{fig:model3} with $l=1$. The Neumann boundary condition is imposed on the boundary except for the Dirichlet boundary condition which is imposed on a small square region with a characteristic length of side $0.1l$ located at the middle region of one of the external surfaces. 


We test the following cases with the random initial distribution of $\chi$. Firstly, we evaluate the influence of the conductivity ratio $\frac{\kappa_1}{\kappa_2}$ on the optimal distribution of $\chi$ on a $60\times60\times 60$ grid. We set the parameters as $q_1=1$, $q_2=100$, $\xi = 1\times 10^{-7}$, $\gamma = 20$, $\tau = 3\times 10^{-5}$, volume fraction $\beta= 0.2$. Let $\kappa_1=20,\ \kappa_2=1$ in Figure~\ref{fig:k20v2} and $\kappa_1=40,\ \kappa_2=1$ in Figure~\ref{fig:k40v2}, we observe that the branches become progressively thinner as the value of $\frac{\kappa_1}{\kappa_2}$ increases and the objective functional decaying property is observed. In Figure~\ref{fig:k20v2}, the top three graphs represent the cutting figures of the distribution of $\chi$ on the planes $z=0,\ y=1/2$ and $x=1/2$ respectively. Next, we fix $q_1=1$,\ $q_2=100$, $\kappa_1=40,\ \kappa_2=1$, $\xi = 1e-7$, $\gamma = 20$, $\tau=3\times 10^{-5}$ and volume fraction $\beta= 0.2$ to test the impact of mesh sizes. We show the impacts of mesh size on the optimal distribution of $\chi$ in Figure~\ref{fig:k40v2} and Figure~\ref{fig:N128} respectively. Similar to the two dimensions problem, the optimal distributions of $\chi$ show finer branches in the finer mesh. These results are similar to those obtained in \cite{burger2013three}.

\begin{figure}[ht!]
    \centering
    \includegraphics[width=0.4\linewidth,trim=0cm 1.5cm 0cm 10cm,clip]{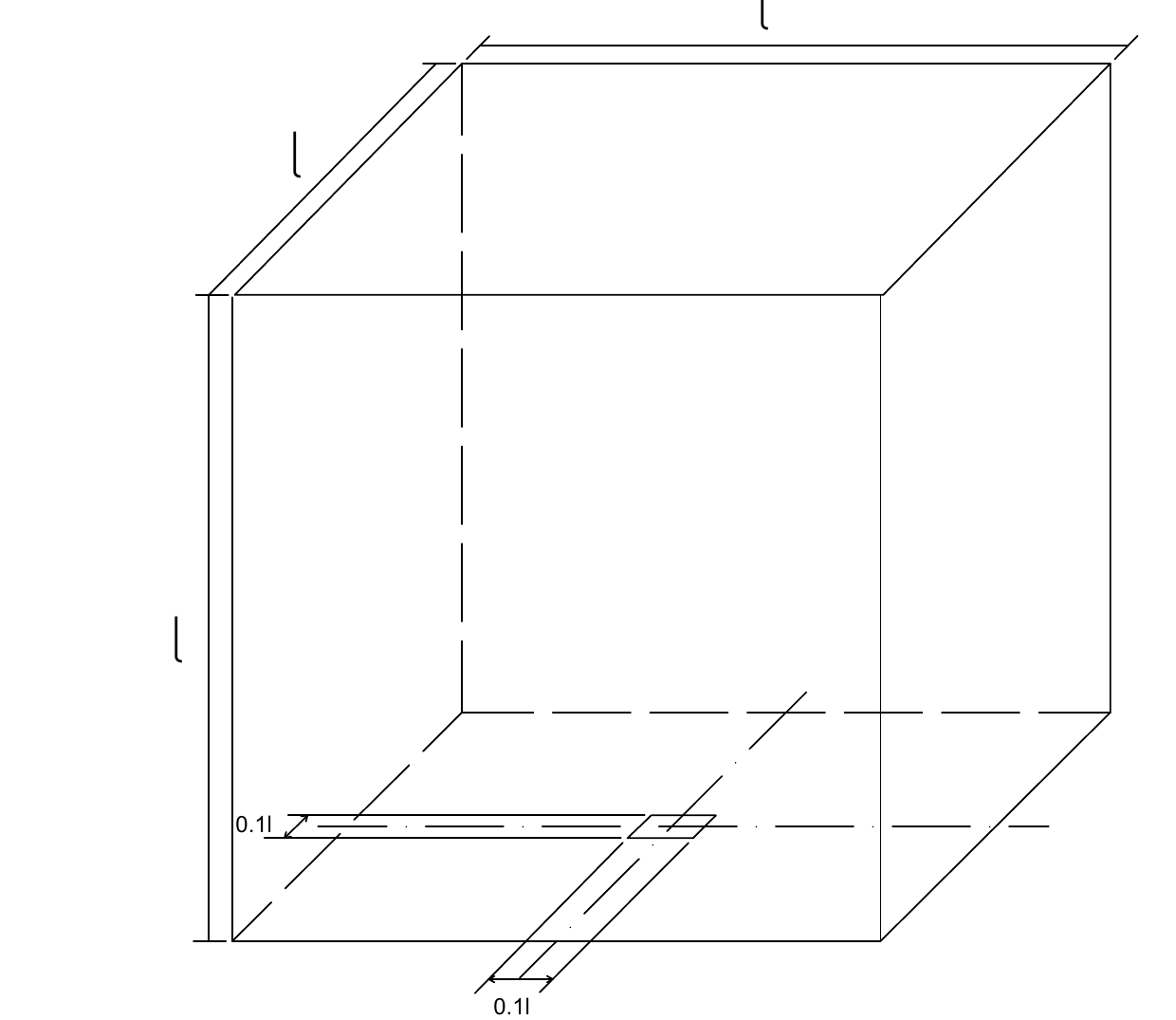}
    \caption{The Dirichlet boundary is on the middle of bottom face(cf. \cite{burger2013three}). See Section~\ref{sec:3d}.}
    \label{fig:model3}
\end{figure}

\begin{figure}[ht!]
    \centering
    \includegraphics[width=0.3\linewidth]{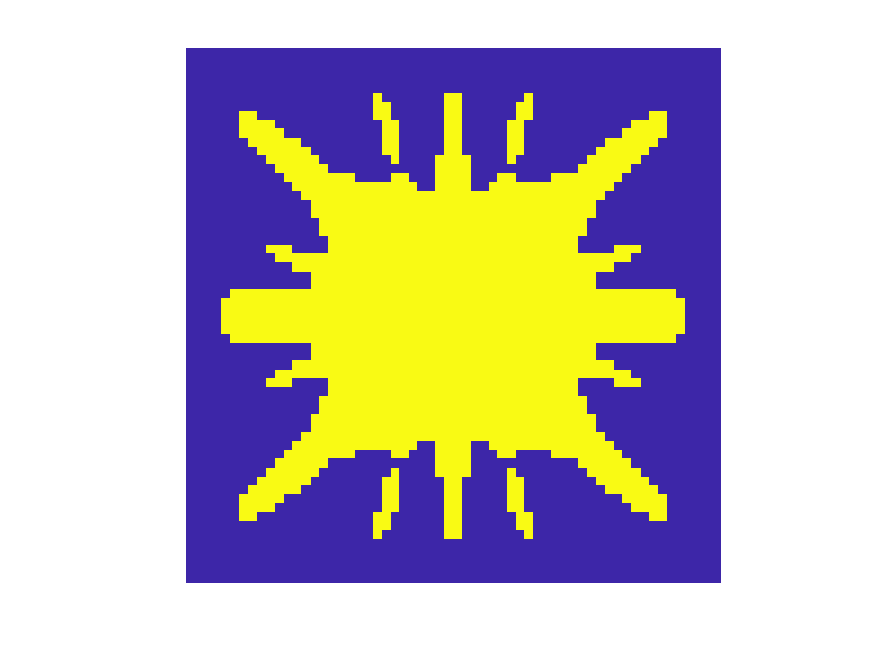}
    \includegraphics[width=0.3\linewidth]{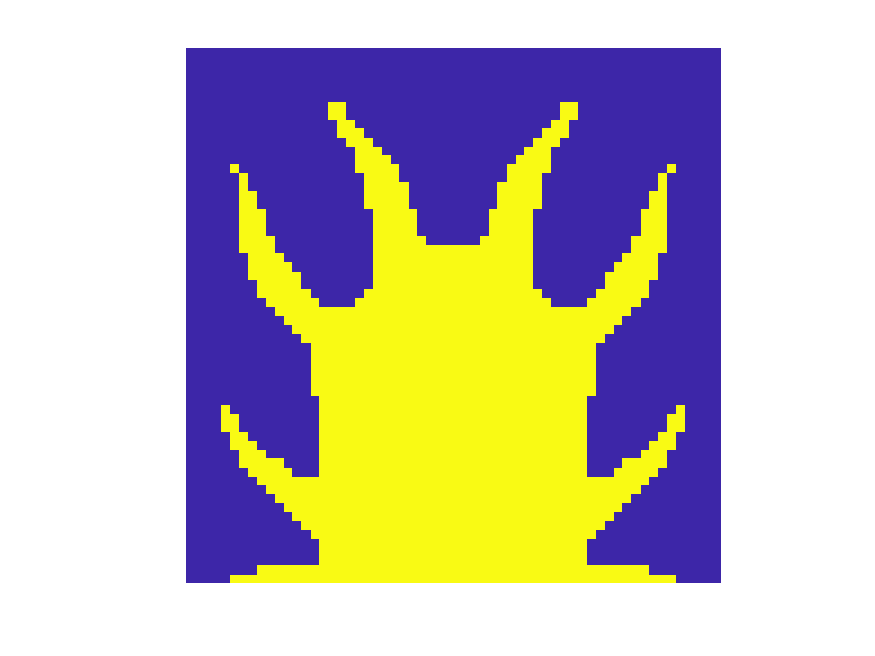}
    \includegraphics[width=0.3\linewidth]{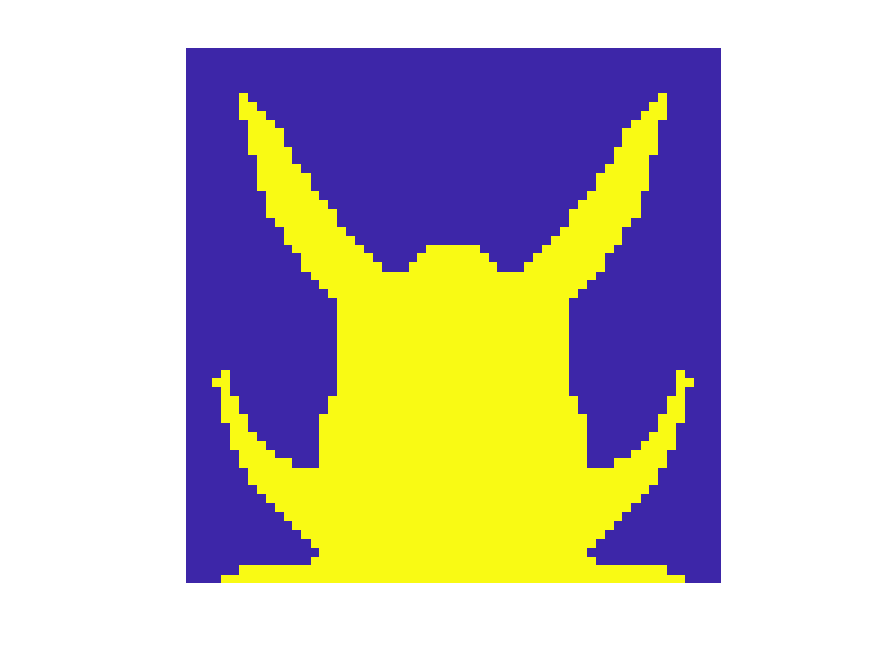}\\
    \includegraphics[width=0.3\linewidth]{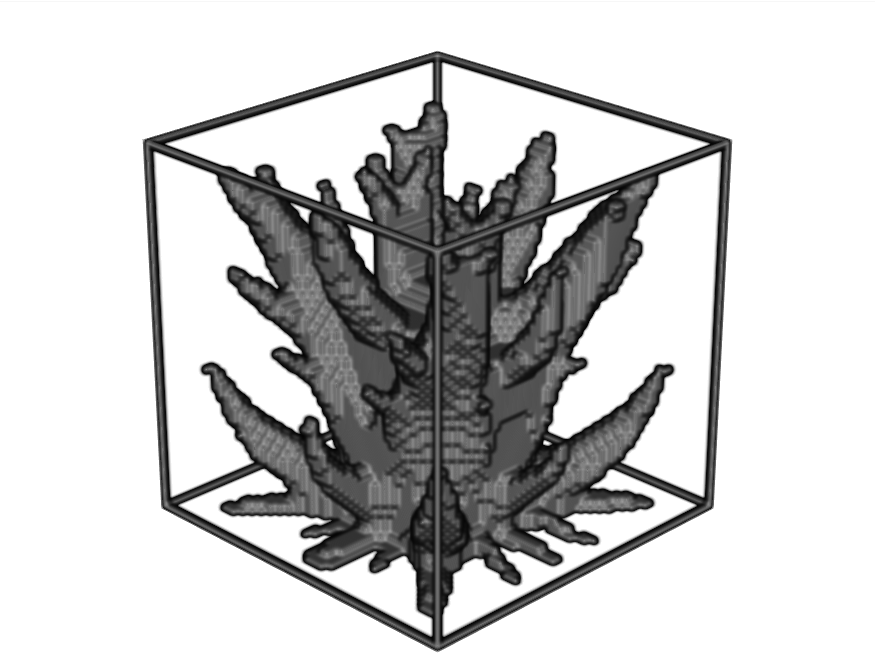}
    \includegraphics[width=0.3\linewidth]{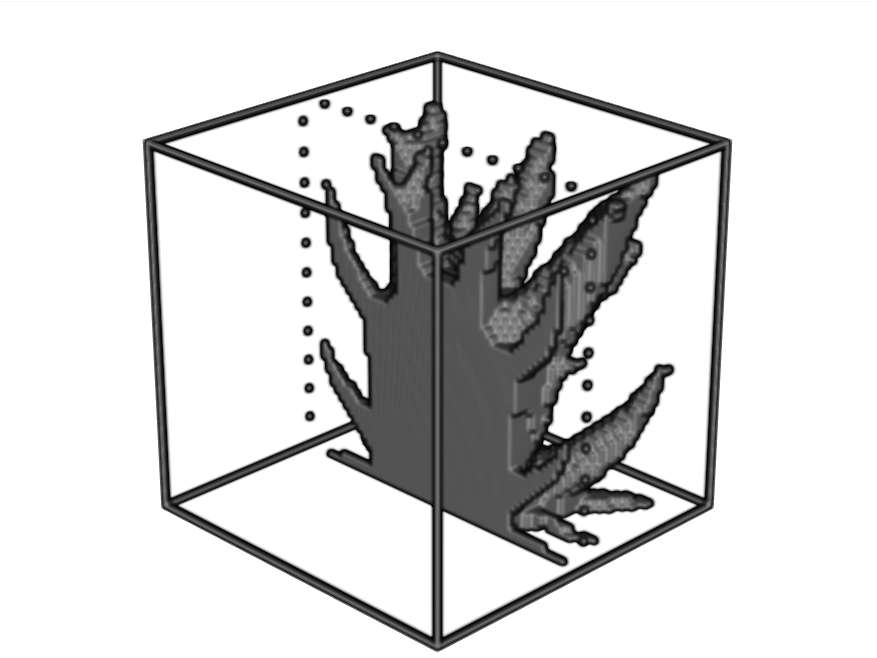}
    \includegraphics[width=0.3\linewidth,trim=2.5cm 9cm 3cm 10cm,clip]{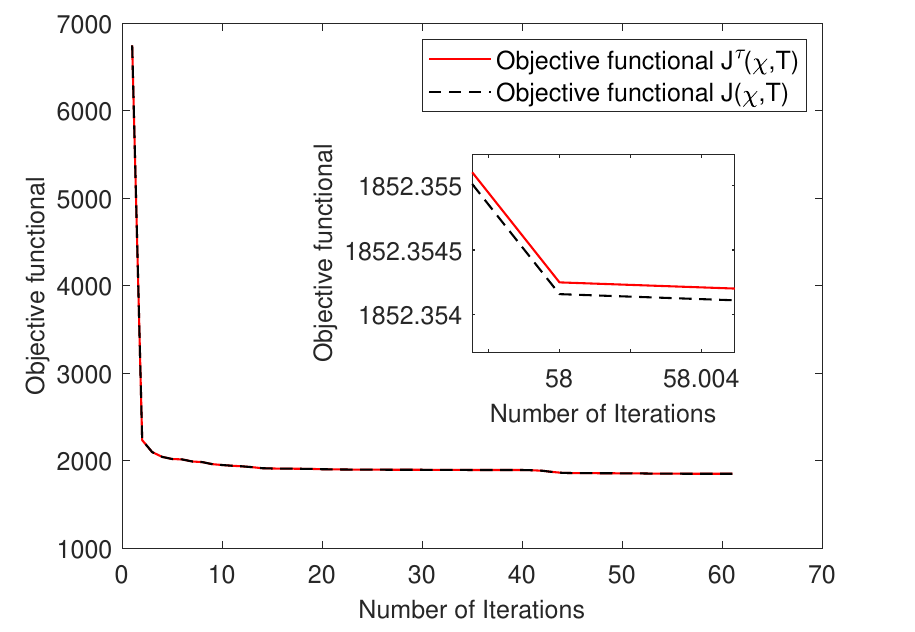}
    \caption{Optimal results for $\kappa_1=20,\ \kappa_2=1$ and volume fraction $\beta = 0.2$ on a $60\times 60 \times 60$ grid. Top-left: A cutting graph of $\chi$ on the plane $z=0$. Top-middle: A cutting graph of $\chi$ on the plane $y=1/2$. Top-right: A cutting graph of $\chi$ on the plane $x=1/2$. The distribution of high-conduction materials $\chi$. Bottom-left: Isometric view. Bottom-middle: Sectional isometric view. Bottom-right: Objective functional values. See Section~\ref{sec:3d}.}
    \label{fig:k20v2}
\end{figure}

\begin{figure}[ht!]
    \centering
    \includegraphics[width=0.3\linewidth]{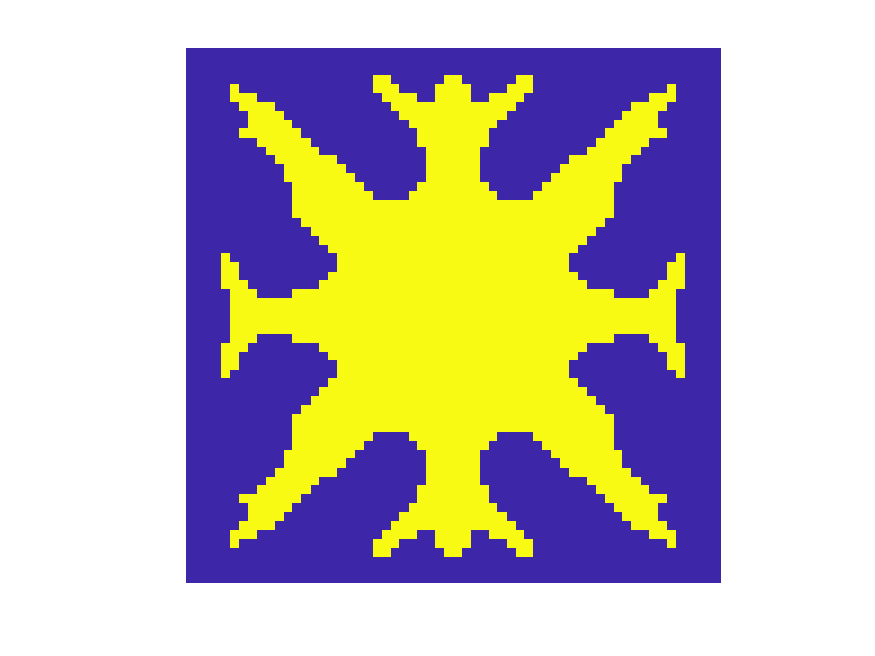}
    \includegraphics[width=0.3\linewidth]{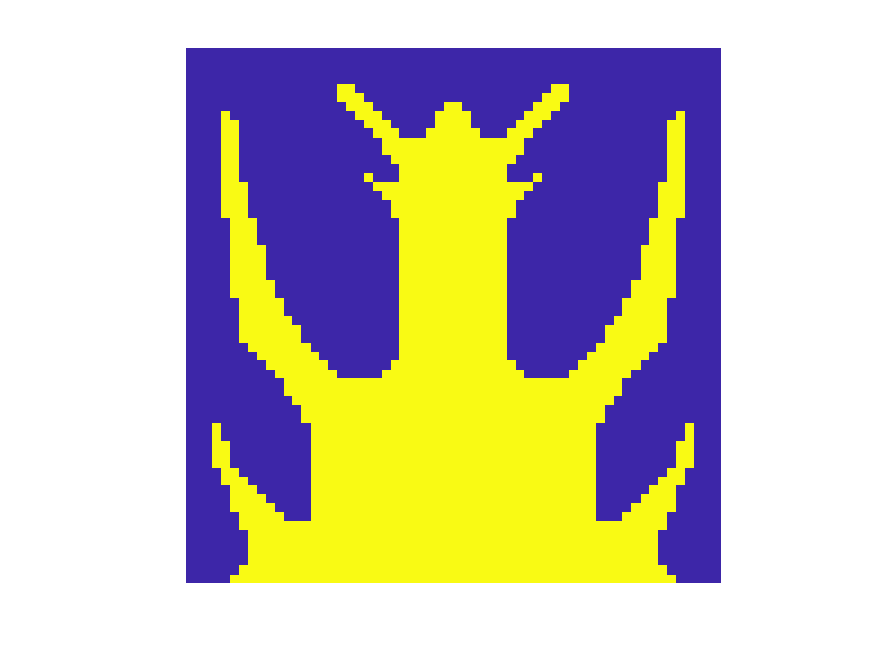}
    \includegraphics[width=0.3\linewidth]{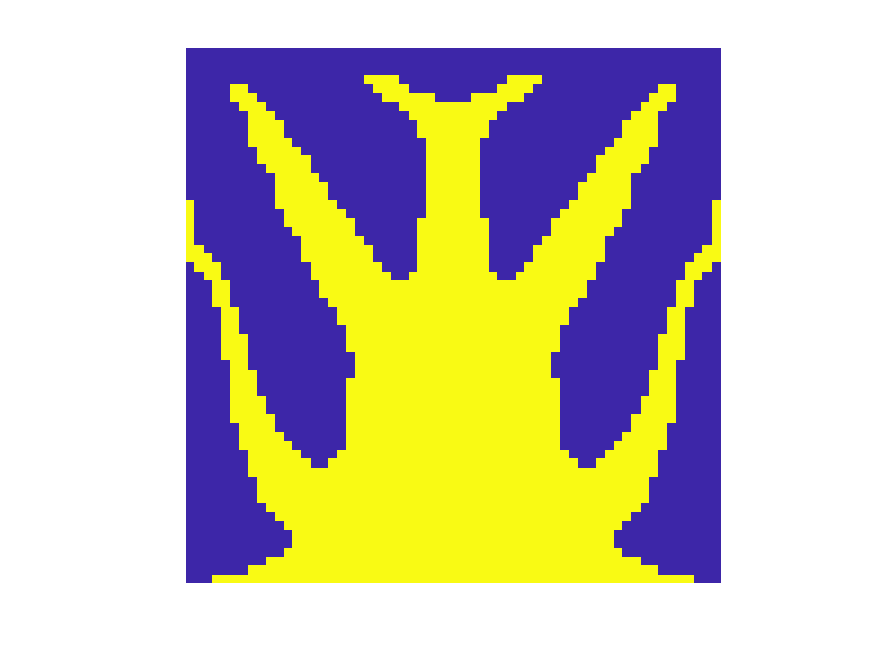}\\
     \includegraphics[width=0.3\linewidth]{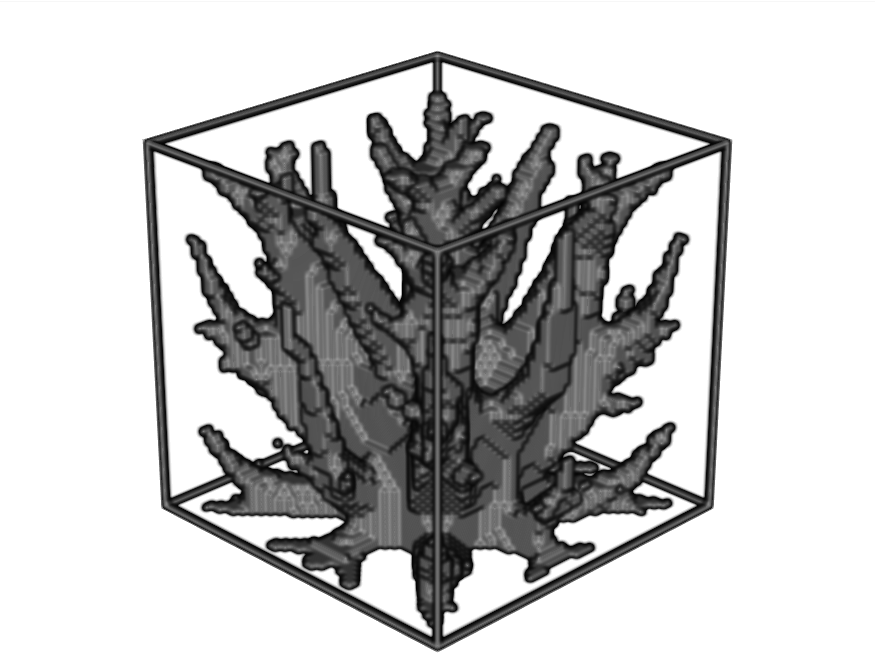}
    \includegraphics[width=0.3\linewidth]{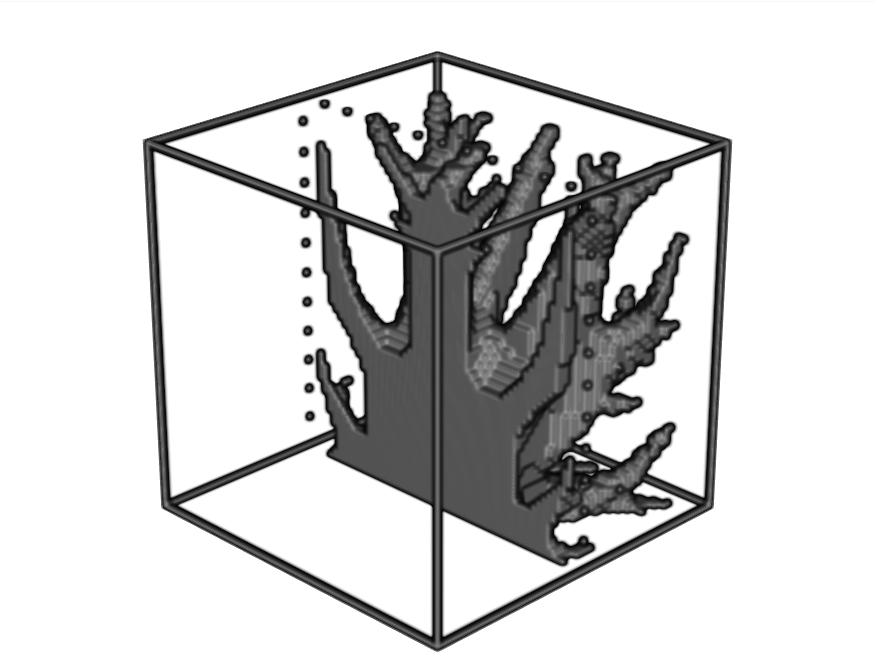}
    \includegraphics[width=0.3\linewidth,trim=2.5cm 9cm 3cm 10cm,clip]{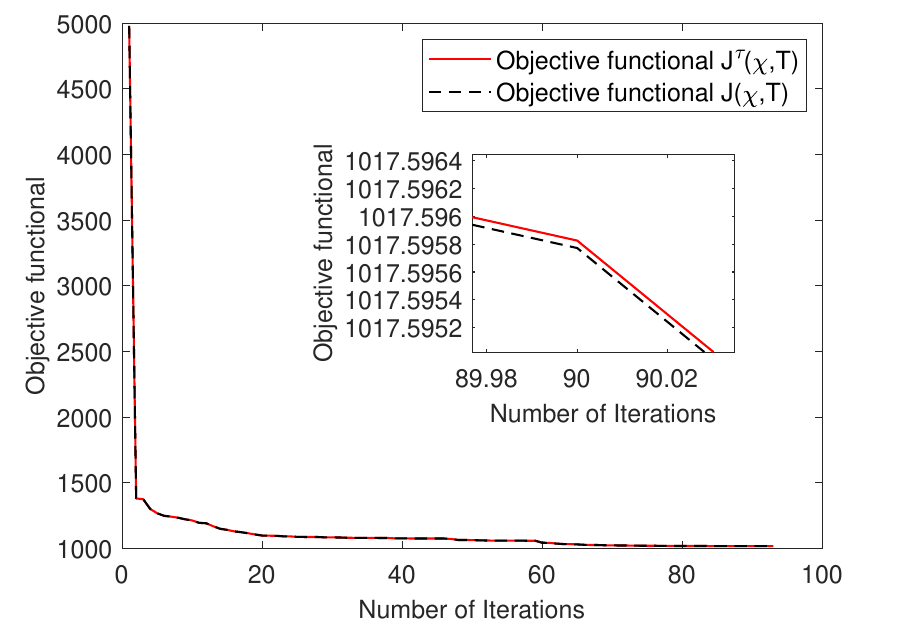}
    \caption{Optimal results for $\kappa_1=40,\ \kappa_2=1$ and volume fraction $\beta = 0.2$ on a $60\times 60 \times 60$ grid. Top-left: A cutting graph of $\chi$ on the plane $z=0$. Top-middle: A cutting graph of $\chi$ on the plane $y=1/2$. Top-right: A cutting graph of $\chi$ on the plane $x=1/2$. The distribution of high-conduction materials $\chi$. Bottom-left: Isometric view. Bottom-middle: Sectional isometric view. Bottom-right: Objective functional values. See Section~\ref{sec:3d}.}
    \label{fig:k40v2}
\end{figure}



\begin{figure}
    \centering
    \includegraphics[width=0.3\linewidth]{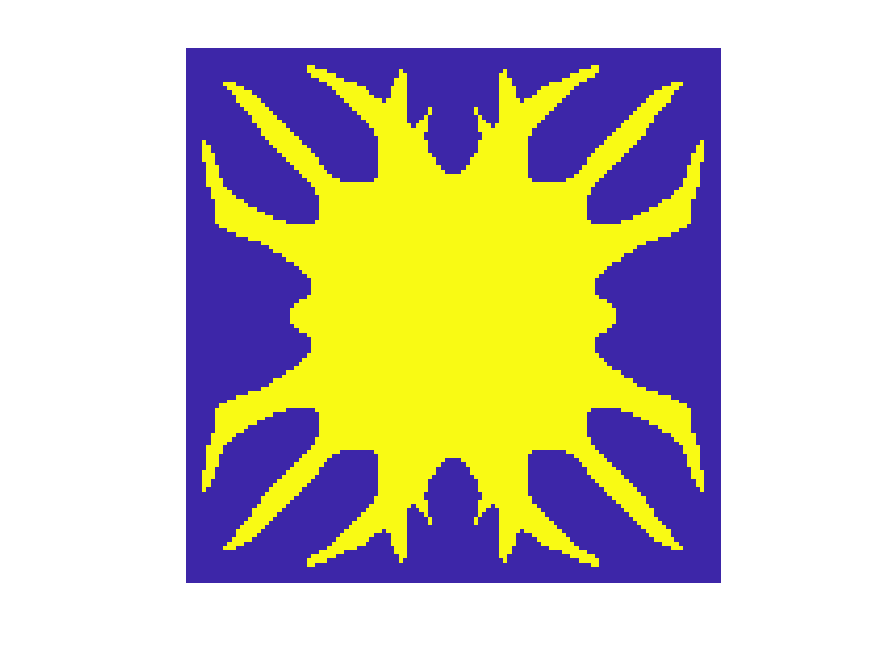}
    \includegraphics[width=0.3\linewidth]{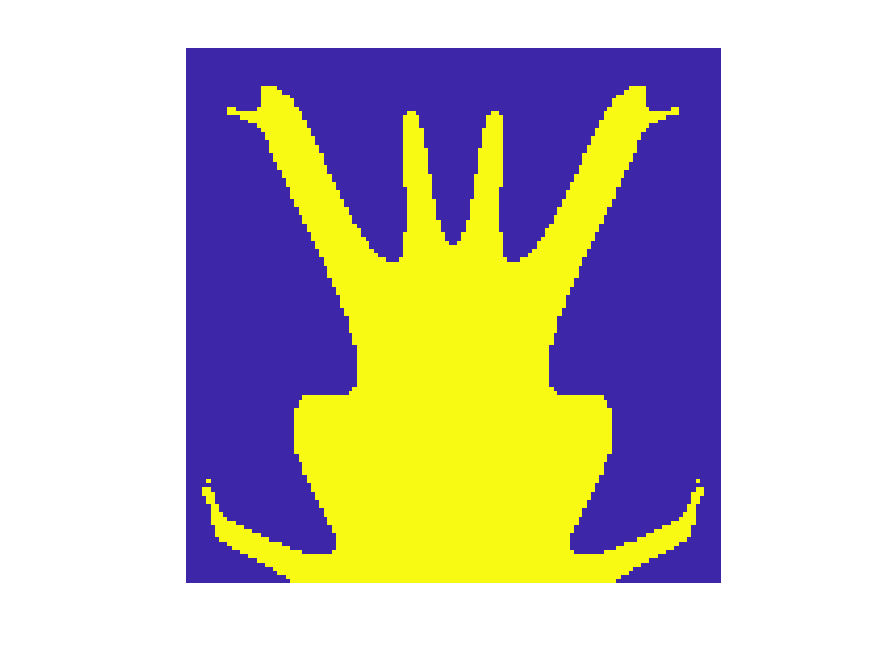}
    \includegraphics[width=0.3\linewidth]{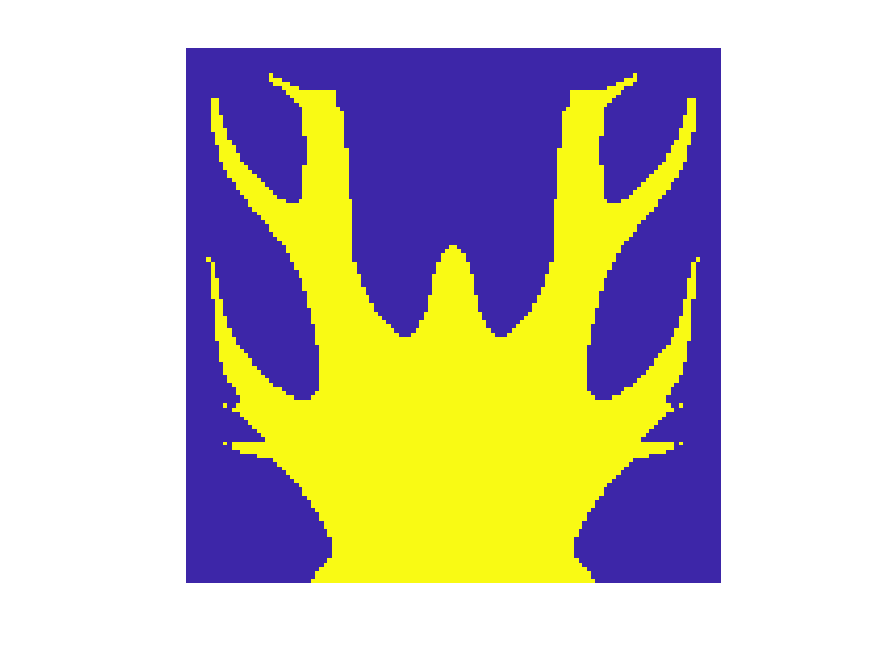}\\
    \includegraphics[width=0.3\linewidth]{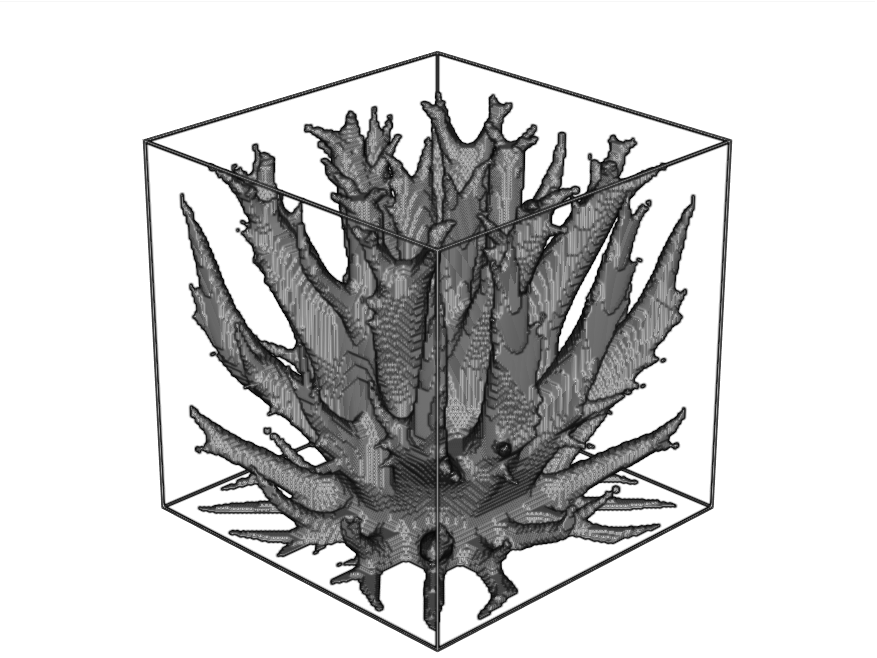}
    \includegraphics[width=0.3\linewidth]{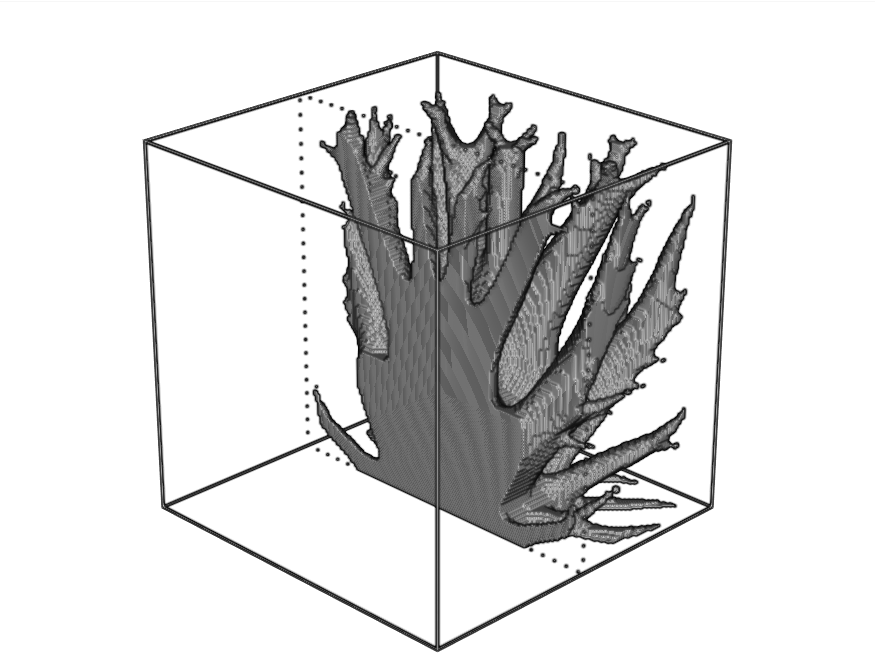}
    \includegraphics[width=0.3\linewidth,trim=2.5cm 9cm 3cm 10cm,clip]{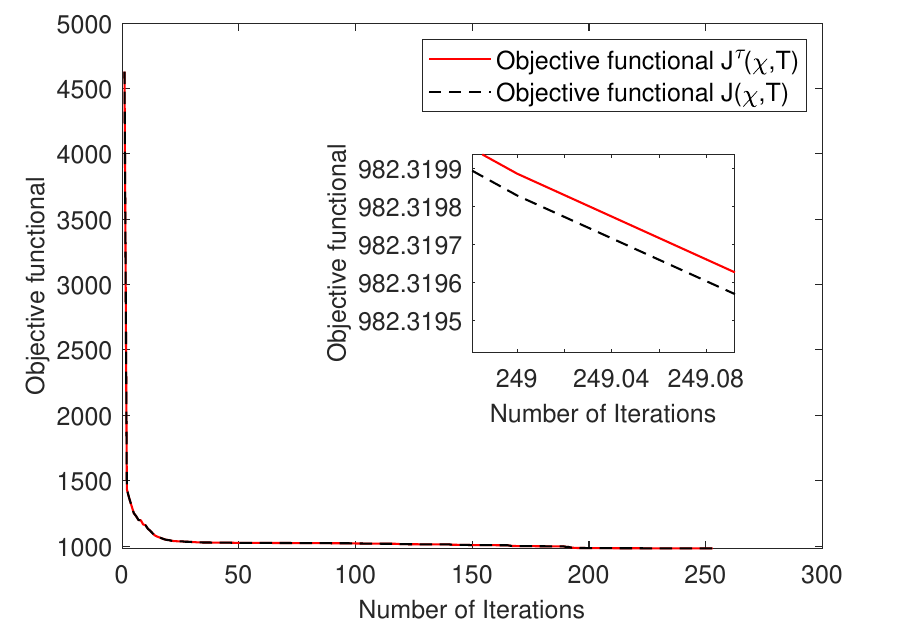}
    \caption{Optimal results for $\kappa_1=40,\kappa_2=1$ and volume fraction $\beta = 0.2$ on a $ 128\times128\times 128$ grid. Top-left: A cutting graph of $\chi$ on the plane $z=0$. Top-middle: A cutting graph of $\chi$ on the plane $y=1/2$. Top-right: A cutting graph of $\chi$ on the plane $x=1/2$. Bottom-left: Isometric view. Bottom-middle: Sectional isometric view. Bottom-right: Objective functional values. See Section~\ref{sec:3d}.}
    \label{fig:N128}
\end{figure}

\section{Conclusions}\label{sec:conclusion}

In this paper, we developed a prediction-correction based iterative convolution-thresholding method (ICTM) for topology optimization of steady-state heat transfer problem. Our objective is to minimize a total energy problem comprising the complementary energy and a perimeter/surface area regularization term, which is constrained by the heat transfer equation. The iterative convolution-thresholding method employs a coordinate descent approach to minimize the approximate energy with an unchanged constraint set. However, the constraint set of the topology optimization problem proposed in this paper can vary. Fortunately, we have effectively resolved the issue of energy non-decrease arising from the changing constraint set by employing the prediction-correction based ICTM. Additionally, we employ a simple thresholding method for updating the indicator functions, which significantly reduces the computational complexity. The numerical results demonstrate the effectiveness and robustness of this method. We have also investigated the impact of the thermal conductivity ratio, the heat generation rate ratio, the mesh size, volume fraction and other parameters on the optimal distribution. We believe that our proposed method can also be extended to muti-physics problems such as heat sink, fluid-structure interaction problems, which will be the focus of our future research.

\section*{Acknowledgements}
This research is supported by the Guangdong Key Lab of Mathematical Foundations for Artificial Intelligence. Authors would like to thank Dinghang Tan for the support on 3D numerical experiments. H. Chen acknowledges support from National Natural Science Foundation of China (Grant No. 12122115, 11771363, 12271457) and Natural Science Foundation of Fujian Province (Grant No. 2023J02003). D. Wang acknowledges support from National Natural Science Foundation of China (Grant No. 12101524), Guangdong Basic and Applied Basic Research Foundation (Grant No. 2023A1515012199) and Shenzhen Science and Technology Innovation Program (Grant No. JCYJ20220530143803007, RCYX20221008092843046). X.-P. Wang acknowledges support from the National Natural Science Foundation of China (NSFC) (No. 12271461), the key project of NSFC (No. 12131010), Shenzhen Science and Technology Innovation Program (Grant: C10120230046), the Hong Kong Research Grants Council GRF (grants 16308421) and the University Development Fund from The Chinese University of Hong Kong, Shenzhen (UDF01002028).


\bibliography{reference}

\end{document}